\documentclass[a4paper]{article}

\usepackage{dsfont}
\usepackage{amssymb}
\usepackage{amsthm}
\usepackage{latexsym}
\usepackage{amsmath}
\usepackage{color}
\usepackage{comment}
\usepackage{ulem}
\usepackage{enumerate}%\begin{enumerate}[{(}i{)}]
\usepackage{bm}
\usepackage{bbm} %\mathbbm{ABCabc\nabla}
\usepackage{authblk}

\usepackage{graphicx}
\usepackage[T1]{fontenc}
\usepackage{latexsym}
\usepackage{xypic}
\usepackage{eufrak}
\usepackage{euscript}
\usepackage{amsfonts,amsmath}
\usepackage{verbatim}
\usepackage{fancyhdr}
\usepackage[english]{babel}
\usepackage{mathrsfs}
\usepackage{units}

\newif\ifrs
\rstrue
\ifrs \usepackage{mathrsfs} \fi

\newif\ifcol
\colfalse
\ifcol
\newcommand{\colred}{\color[rgb]{0.8,0,0}}
\newcommand{\colorr}{\color[rgb]{0.8,0,0}}
\newcommand{\colorg}{\color[rgb]{0,0.5,0}}
\newcommand{\colorb}{\color[rgb]{0,0,0.8}}

\newcommand{\colorn}{\color[rgb]{1,1,1}}

\newcommand{\coloro}{\color[rgb]{1,0.4,0}}%{1,0.851,0}}

\else
\newcommand{\colred}{\color{black}}%{\color[rgb]{0.8,0,0}}
\newcommand{\colorb}{\color{black}}%{\color[rgb]{0,0,0.8}}

\newcommand{\colorr}{\color{black}}% {{\color[rgb]{0.8,0,0}}
\newcommand{\colorg}{\color{black}}% {\color[rgb]{0,0.5,0}}

\newcommand{\colorn}{\color{black}}% {\color[rgb]{1,1,1}}
\newcommand{\coloro}{\color{black}}% {\color[rgb]{1,0.851,0}}

\fi

\newtheorem{theorem*}{Theorem}[subsection]

\newtheorem{note*}[theorem*]{Note}
\newtheorem{lemma*}[theorem*]{Lemma}
\newtheorem{definition*}[theorem*]{Definition}
\newtheorem{proposition*}[theorem*]{Proposition}
\newtheorem{corollary*}[theorem*]{Corollary}
\newtheorem{remark*}[theorem*]{Remark}
\newtheorem{example*}[theorem*]{Example}
\numberwithin{equation}{section}
\newtheorem{prop}{Proposition}[section]

\newtheorem{lemme}[prop]{Lemma}
\newtheorem{rem}[prop]{Remark}

\newtheorem{thm}[prop]{Theorem}

\excludecomment{en-text}
\includecomment{jp-text}

\excludecomment{comment}
\def\bd{\begin{description}}
\def\ed{\end{description}}
\def\partialbs{\backslash\!\!\!\partial}
\def\D2{\bbD_{2,\infty-}}

\def\tj{{t_j}}
\def\tjm{{t_{j-1}}}

\def\D{{\bf D}}

\def\calb{{\cal B}}
\def\calc{{\cal C}}

\def\cale{{\cal E}}
\def\calf{{\cal F}}
\def\calg{{\cal G}}

\def\cali{{\cal I}}

\def\cals{{\cal S}}

\def\yeq{\>=\>}

\def\sfd{{\sf d}}
\def\sfp{{\sf p}}

\def\simleq{\ \raisebox{-.7ex}{$\stackrel{{\textstyle <}}{\sim}$}\ }

\def\ep{\epsilon}
\def\half{\frac{1}{2}}

\def\up{\uparrow}
\def\down{\downarrow}

\def\halflineskip{\vspace*{3mm}}
\def\nn{\nonumber}
\def\be{\begin{equation}}
\def\ee{\end{equation}}
\def\bea{\begin{eqnarray}}
\def\eea{\end{eqnarray}}
\def\beas{\begin{eqnarray*}}
\def\eeas{\end{eqnarray*}}
\def\bi{\begin{itemize}}
\def\ei{\end{itemize}}
\def\im{\item}
\def\bd{\begin{description}}
\def\ed{\end{description}}
\def\dotc{\stackrel{\circ}{C}}

\def\dotw{\stackrel{\circ}{W}}

\newcommand{\bbC}{{\mathbb C}}
\newcommand{\bbD}{{\mathbb D}}

\newcommand{\bbN}{{\mathbb N}}

\newcommand{\bbR}{{\mathbb R}}

\newcommand{\bbW}{{\mathbb W}}

\newcommand{\bbZ}{{\mathbb Z}}

\def\csfd{{\check{\sfd}}}
\def\mfh{{\EuFrak H}}
\def\tti{{\tt i}}
\def\onelineskip{\halflineskip\halflineskip}

\newcommand{\sfx}{{\sf x}}
\newcommand{\sfz}{{\sf z}}

\def\sfp{{\sf p}}
\def\sfd{{\sf d}}

\def\tj{{t_j}}
\def\tjm{{t_{j-1}}}

\def\tkm{{t_{k-1}}}

\def\tlm{{t_{\ell-1}}}

\def\bd{\begin{description}}
\def\ed{\end{description}}
\def\partialbs{\backslash\!\!\!\partial}
\def\D2{\bbD_{2,\infty-}}

\def\dotc{\stackrel{\circ}{C}}

\def\dotw{\stackrel{\circ}{W}}

\def\dotx{\stackrel{\circ}{X}}
\def\dotg{\stackrel{\circ}{G}}
\def\HH{\EuFrak H}

\def\cal{\mathcal}
\def\1{{\mathbf{1}}}

\def\1{{\mathbf{1}}}
\def\0.5{{\frac{1}{2}}}

\begin{document}
%\pagecolor{black}
%\color{white}

\title{%Expansion of the Asymptotically Conditionally Normal Law %for ISM Research Memo. 
Asymptotic expansion of Skorohod integrals%arXiv
%Asymptotic expansion of the quasi maximum likelihood estimator for volatility%Submit. 
\footnote{
This work was in part supported by 
 the    NSF grant  DMS 1512891;
Japan Society for the Promotion of Science 
Grants-in-Aid for Scientific Research No. 17H01702 (Scientific Research), 
%No. 26540011 (Challenging Exploratory Research); 
%the Global COE program ``The Research and Training Center for New Development in Mathematics'' 
%of the Graduate School of Mathematical Sciences, University of Tokyo; 
Japan Science and Technology Agency CREST  JPMJCR14D7; 
%NS Solutions Corporation; 
and by a Cooperative Research Program of the Institute of Statistical Mathematics.
%{\colorred The authors thank the referee for valuable comments.}
}
}
\author[1]{David Nualart}
\affil[1]{Department of Mathematics, Kansas University\footnote
        {Lawrence, Kansas, 66045, USA
%Email: nualart@math.ku.edu
%Home page: http://www.math.ku.edu/\UTF{223C}nualart
}}

\author[2,3,4]{Nakahiro Yoshida}
\affil[2]{Graduate School of Mathematical Sciences, University of Tokyo\footnote
        {Graduate School of Mathematical Sciences, University of Tokyo: 3-8-1 Komaba, Meguro-ku, Tokyo 153-8914, Japan}}
\affil[3]{Japan Science and Technology Agency CREST%\footnote{}
        }
\affil[4]{The Institue of Statistical Mathematics
%%\footnote{Institue of Statistical Mathematics: 10-3 Midori-cho, Tachikawa, Tokyo 190-8562, Japan}
        }

\date{%November 16, 2016 \\
%April 21, 2017\\
%June 15, 2017\\
%November 18, 2017
%December 14, 2017
December 30, 2017
%Revised February 25, 2012
}
\maketitle
\ \\
{\it Summary} 
Asymptotic expansion of the distribution of a perturbation $Z_n$ of a Skorohod integral 
jointly with a reference variable $X_n$ is derived. 
We introduce a second-order interpolation formula in frequency domain 
to expand a characteristic functional and combine it with the scheme developed in the martingale expansion. 
The second-order interpolation and Fourier inversion give asymptotic expansion of the expectation 
$E[f(Z_n,X_n)]$ for differentiable functions $f$ and also measurable functions $f$. 
In the latter case, the interpolation method connects the two non-degeneracies of variables 
for finite $n$ and $\infty$. 
Random symbols are used for expressing the asymptotic expansion formula. 
Quasi tangent, quasi torsion and modified quasi torsion are introduced in this paper. 
We identify these random symbols for a certain quadratic form of a fractional Brownian motion and 
for a quadratic from of a fractional Brownian motion with random weights. 
For a quadratic form of a Brownian motion with random weights, 
we observe that our formula reproduces the formula originally obtained by the martingale expansion. \ 

%%%%%%%%%%%%%%%%%%%%%%%%%%%%%%%%%%%%%%%%%%%%%%%%%%%%%%%%%%%%
%%%%%%%%%%%%%%%%%%%%%%%%%%%%%%%%%%%%%%%%%%%%%%%%%%%%%%%%%%%%
%%%%%%%%%%%%%%%%%%%%%%%%%%%%%%%%%%%%%%%%%%%%%%%%%%%%%%%%%%%%
%%%%%%%%%%%%%%%%%%%%%%%%%%%%%%%%%%%%%%%%%%%%%%%%%%%%%%%%%%%%
%%%%%%%%%%%%%%%%%%%%%%%%%%%%%%%%%%%%%%%%%%%%%%%%%%%%%%%%%%%%

\section{Introduction}

Asymptotic expansion of distributions is one of the fundamentals of theoretical statistics. 
Its applications spread over higher order approximation of probability distributions, 
theory of higher order asymptotic efficiency of estimators, 
prediction, information criteria for model selection, saddle point approximation, 
bootstrap and resampling methods, information geometry and so on. 
Bhattacharya and Rao \cite{bhattacharya2010normal} give an excellent exposition of the probabilistic aspects 
of asymptotic expansion for independent variables, while we omit a huge amount of literature 
on statistical applications of asymptotic expansion methods. 
Asymptotic expansion has a long history even for dependent models. 
A celebrated paper G\"otze and Hipp \cite{GotzeHipp1983} is a  compilation of the studies 
of asymptotic expansion for Markovian or nearly Markovian chains with mixing property. 
It was followed by G\"otze and Hipp \cite{GotzeHipp1994}, that applied their result to time series models 
more explicitly. 

For functionals in stochastic analysis, there are two ways: 
martingale approach and mixing approach. 
Yoshida \cite{Yoshida1997,yoshida2001malliavin} gave asymptotic expansions for martingales 
with the Malliavin calculus on Wiener/Wiener-Poisson space 
and applied them to an ergodic diffusion, a volatility estimation over a finite time interval in central limit case, 
and a stochastic regression model with an explanatory process having long memory. 
Regularity of the distribution is critical to validate an asymptotic expansion. 
Thus, naturally the Malliavin calculus was used there to ensure a decay rate of 
 characteristic type functionals. 
In connection, though the regularity problem does not occur there, 
Mykland \cite{Mykland1992} is a pioneering work on expansion of moments for a smooth function of a martingale. 
The mixing approach is more efficient if a sufficiently fast mixing property is available. 
Kusuoka and Yoshida \cite{KusuokaYoshida2000} and 
Yoshida \cite{yoshida2004partial} developed asymptotic expansions 
for $\ep$-Markov processes possessing a mixing property. 
The Malliavin calculus was used to estimate 
certain conditional characteristic functionals defined locally in time with the assistance of support theorems. 
See e.g. Yoshida \cite{yoshida2016asymptotic} for an overview  
and references therein. 

In the last three decades, along with the developments in statistics for high frequency data, 
stable limit theorems have attracted a lot of attention. 
Estimation of volatility from high frequency data under finite time horizon typically becomes non-ergodic statistics. 
Then, the asymptotic expansion of functionals of increments of stochastic processes 
is once again an issue after recent tremendous progresses in limit theorems in this area. 
Even though big data is available, 
the problem of microstructure noise motivates  the use of asymptotic expansion. For example,
Yoshida \cite{yoshida2013martingale} extended \cite{Yoshida1997} to martingales with mixed Gaussian limit\footnote{An updated version is arXiv:1210.3680 (2012).}, 
Podolskij and Yoshida \cite{podolskij2016edgeworth} derived  a distributional asymptotic expansion of 
the p-variation of a diffusion process and
Podolskij, Veliyev and Yoshida \cite{podolskij2017edgeworth} gave an Edgeworth expansion 
for the pre-averaging estimator for a diffusion process sampled under microstructure noise.

Beyond semimartingales theory, special attention has been focused in recent years 
on limit theorems for objects in the Malliavin calculus. 
Nualart and Peccati \cite{nualart2005central} 
established the fourth moment theorem and 
characterized the central limit theorem for a sequence of multiple stochastic integrals of a fixed order. 
Nualart and Oriz-Latorre \cite{nualart2008central} extended the result in Nualart and Peccati \cite{nualart2005central}. 
%diagram formulae [see Surgailis (2000) for a detailed survey].
%
Peccati and Tudor \cite{peccati2005gaussian} presented necessary and sufficient conditions for 
the central limit theorem for vectors of multiple stochastic integrals and showed that 
componentwise convergence implies joint convergence. 
Nourdin, Nualart and Peccati \cite{nourdin2016quantitative} introduced an interpolation technique 
and proved quantitative stable limit theorems where the limit distribution is a mixture of 
Gaussian distributions. 
Power variation, stable convergence and Berry-Esseen type inequality are also in the scope of this trend. 
%
%

% power variation
Nourdin and Peccati \cite{nourdin2008weighted} showed  
the asymptotic behavior of a weighted power variation processes associated
with the so-called iterated Brownian motion. 
Corcuera, Nualart and Woerner \cite{corcuera2006power} 
gave a mixture type central limit theorem for the power variation of a stochastic integral with respect to a fractional Brownian motion. 
Nourdin, Nualart and Tudor \cite{nourdin2010central} derived central and non-central limit theorems 
for certain weighted power variations of the fractional Brownian motion. 
Nourdin \cite{nourdin2008asymptotic} showed various asymptotic behavior of weighted quadratic and 
cubic variations of a fractional Brownian motion having a small Hurst index. 

% stable convergence 
By connecting a martingale approach and deforming a nesting condition, 
Peccati and Taqqu \cite{peccati2008stable} showed stable convergence of %Skorohod Integrals
multiple Wiener-It\^o integrals. 
Nourdin and Nualart \cite{nourdin2010central} proved a central limit theorem for a sequence of multiple
Skorohod integrals and applied it to renormalized weighted Hermite
variations of the fractional Brownian motion. 
Related to stable convergence are Harnett and Nualart \cite{harnett2012weak,harnett2013central} 
on weak convergence of the Stratonovich integral with respect to a class of Gaussian processes. 

% Berry-Esseen bound for CLT
Based on Stein's method, among many others, 
Nourdin and Peccati \cite{nourdin2009stein} presented a Berry-Esseen bound for multiple Wiener-It\^o integrals, %%%
and 
Edan and V\'iquez \cite{eden2015nourdin} obtained central limit theorems with Wiener-Poisson space. 
Kusuoka and Tudor \cite{kusuoka2012stein} proposed Stein's method for invariant measures of diffusions. 
An advantage of Stein's method is that it provides fairly explicit error bounds of approximation. 
The interpolation method recently introduced by Nourdin, Nualart and Peccati \cite{nourdin2016quantitative} 
keeps this merit. 

After observing these developments, the aim of this paper is to derive asymptotic expansions
for Skorohod integrals by means of the Malliavin calculus. 
It is worth recalling the  terminology in the martingale expansion of \cite{yoshida2013martingale} 
though our discussion will be apart from the martingale theory.  
For a sequence of continuous martingales $M^n=\{M^n_t, t\in[0,1]\}$, 
denote by $C^n=\langle M^n\rangle$ the quadratic variation of $M^n$. % evaluated at $t=1$. 
When $C_n:=C^n_1\to^pC_\infty$ as $n\to\infty$, stable convergence of $M_n:=M^n_1$ to 
a mixed normal limit $M_\infty\sim N(0,C_\infty)$ 
usually takes place even if $C_\infty$ is random. 
More precise an evaluation of the gap $C_n-C_\infty$ is necessary to go up to an asymptotic expansion. 
The variable $\dotc_n=r_n^{-1}(C_n-C_\infty)$ is called {\it tangent}, where  
$r_n$ is a positive number tending to zero as $n\to\infty$. 
The effect of $\dotc_n$ appears in the first order asymptotic expansion,
and it gives everything in the classical case of constant $C_\infty$. 
On the other hand, if $C_\infty$ is random, as it is the case of non-ergodic statistics, then 
the exponential local martingale $e^n_t(\sfz)=\exp\big(\tti\sfz M^n+2^{-1}\sfz^2C^n_t\big)$ 
is no longer a local martingale under the transformed measure $\exp\big(-2^{-1}\sfz^2C_\infty\big)dP/E[\exp\big(-2^{-1}\sfz^2C_\infty\big)]$. 
This effect remains in the asymptotic expansion, called the {\it torsion} the exponential martingale suffers from. 
Two random symbols $\underline{\sigma}$ and $\overline{\sigma}$ are defined 
for tangent and torsion, respectively, and 
the asymptotic expansion formula is given in terms of the Gaussian density $\phi(z;0,C_\infty)$ 
with variance $C_\infty$ and the adjoint operation of $\underline{\sigma}$ and $\overline{\sigma}$. 
In this article, we will make an expansion formula for 
the Skorohod integral $M_n=\delta(u_n)$ of $u_n$ in a similar way through certain random symbols. 
However, since we do not have any self-evident martingale structure, 
we introduce new random symbols called {\it quasi tangent}, {\it quasi torsion} and {\it modified quasi torsion} 
defined only by Malliavin derivatives of functionals. 

We will take a Fourier analytic approach. 
It is because the formula is a perturbation of a Gaussian density and the actions of random symbols are 
simply expressed, as it was the case in classical theory. 
Moreover, if we extend such a result, 
the limit is possibly related to infinitely divisible distributions even if their mixture appears, and then  
formulation by random symbols seems natural from an operational point of view. 
We use an interpolation method in the frequency domain 
and 
expand a characteristic function of the interpolation. 
The second-order interpolation is provided to relate the distribution of $M_n$ with the random symbols 
that determine the expansion formula. 

In this paper, we combine the interpolation method and the scheme originating from martingale expansion. 
Non-degeneracy of distributions plays an essential role to validate the asymptotic expansion 
of the expectation $E[f(M_n)]$ for measurable functions $f$. 
For that, it is necessary to connect the two non-degeneracies of $M_n$ and $M_\infty$. 
The interpolation method serves as a homotopy between the two random functions, 
just as the interpolation along the real time $t$ was used in 
the martingale expansion \cite{yoshida2013martingale}. 
We require only a local non-degeneracy of $M_n$, not the full non-degeneracy. 
There is a big difference between them. For the latter, we need large deviation estimates and 
that plot often fails in practice. In parametric estimation, we quite often meet a situation where the estimator is not 
defined on the whole $\Omega$ but defined locally as a smooth functional. Then localization is inevitable.

Finally, related to this article, 
we mention a recent work by 
Tudor and Yoshida \cite{tudoryoshida2017} on  
asymptotic expansion of multiple stochastic integrals. 
 
The organization of this paper is as follows. 
We will work with the variable $Z_n$ defined in Section \ref{20160813-1}  
as a perturbation of a Skorohod integral $M_n=\delta(u_n)$ since such 
a stochastic expansion appears 
when statistical estimators are considered. 
A reference variable $X_n$ is also considered. 
This formulation is natural because Studentization is common in non-ergodic statistics, 
and also because the principal part of the normalized estimator is often expressed as the ratio of a Skorohod integral and 
Fisher information. 
Section \ref{20160813-1} introduces the interpolation method and an expansion of a characteristic type functional 
along the interpolation, 
as well as notion of quasi tangent, quasi tosion and modified quasi torsion. 
Section \ref{170813-1} gives asymptotic expansion of $E[f(Z_n,X_n)]$ for differentiable functions $f$. 
We compute the random symbols for a functional of a fractional Brownian motion in Section \ref{170813-2}. 
Since the Skorohod integral generalizes the It\^o integral, our formula should reproduce the same formula as 
that of \cite{yoshida2013martingale} 
if applied to the quadratic form of a Brownian motion with random weights. We will see this in Section \ref{170813-3} 
but the derivation is more complicated than the direct use of the martingale expansion %of \cite{yoshida2013martingale} 
for 
the double It\^o integrals. 
In Section \ref{170813-4}, the random symbols are computed for a quadratic form of a fractional Brownian motion 
with random weights. 
Finally, Section \ref{170813-5} validates asymptotic expansions of $E[f(Z_n,X_n)]$ for measurable functions $f$. 
As mentioned above, we carry out this task by using two non-degeneracies of the Malliavin covariances, 
with the help of the interpolation.

\section{Second-order interpolation formula in frequency domain}
%{Expansion of a smooth functional}
\label{20160813-1}
\subsection{Perturbation of a Skorohod integral}
Given a probability space $(\Omega,\calf,P)$, 
we consider an isonormal Gaussian process $\bbW=\{\bbW(h), h\in\mfh\}$ 
on a real separable Hilbert space $\mfh$. 
For any Hilbert space ${\sf E}$,  any real number $p\ge 1$ and any integer $k\ge 1$, we denote by  $\bbD^{k,p}({\sf E})$ the Sobolev space of
${\sf E}$-valued random variables which are $k$ times differentiable in the sense of Malliavin calculus and the derivatives up to order $k$  have finite moments of order $p$.  We denote by $D$ the  derivative operator in the framework of Malliavin calculus. Its adjoint, denoted by  $\delta$, is called the divergence or the Skorokod integral. We refer to Nualart \cite{Nualart2006} for a detailed account on Malliavin calculus.
We simply write $\bbD^{s,p}$ for $\bbD^{s,p}(\bbR)$. 
Moreover we write $\bbD^{s,\infty}({\sf E})=\cap_{p\ge 1}\bbD^{s,p}({\sf E})$.

For $n\in\bbN$, suppose that $u_n\in\bbD^{1,p}(\mfh\otimes\bbR^\sfd)$ for some $p \ge 2$, i.e., 
$u_n=(u_n^i)_{i=1}^\sfd$ with each $u_n^i\in\bbD^{1,\sfp}(\mfh)$, $i=1,\dots,\sfd$. {\color{black} Let  $M_n=\delta(u_n)$, where
 $\delta(u_n)=(\delta(u_n^i) )_{i=1}^\sfd$. 
Consider random vectors $W_n$ ($n\in\overline{\bbN}=\bbN\cup\{\infty\}$) and  $N_n:\Omega\to\bbR^\sfd$ ($n\in\bbN$).} For a sequence of positive numbers $(r_n)_{n\in\bbN}$ tending to $0$ as $n\to\infty$, 
we will consider a perturbation $Z_n$ of $M_n$ given by 
\beas 
Z_n&=& M_n+W_n+r_n N_n. 
\eeas
Let $G_\infty:\Omega\to\bbR^\sfd\otimes_+\bbR^\sfd$ {\color {black}be a random matrix}, where 
$\bbR^\sfd\otimes_+\bbR^\sfd$ is the set of $\sfd\times\sfd$ nonnegative symmetric matrices. 
The random matrix $G_\infty$ will be the asymptotic random variance matrix of $Z_n$. 
A reference variable is denoted by $X_n:\Omega\to\bbR^{\sfd_1}$, $n\in\overline{\bbN}$. 
%\im $S\in\bbD^{2,4\sfp}(\bbR^\sfd\otimes_+\bbR^\sfd)$ {\colorr [This $S$ is ``$S_\ep$''.]}
We are interested in a higher-order approximation of the joint distribution of $(Z_n,X_n)$. 

For a tensor $T=(T_{i_1,\dots,i_k})_{i_1,\dots ,{\color {black} i_k}}$, we write 
\begin{en-text}
\beas 
T[u_1,\dots,u_r,\cdot]
&=&
T[u_1\otimes\cdots\otimes u_r,\cdot]
\>=\>
\bigg(\sum_{i_1,\dots,i_r}T_{i_1,\dos,i_r,i_{r+1},\dots,i_k}
u_1^{i_1}\cdots u_r^{i_r}\bigg)_{i_{r+1},\dots,i_k}
\eeas
\end{en-text}
\beas 
T[u_1,\dots,u_k]
&=&
T[u_1\otimes\cdots\otimes u_k]
\>=\>
\sum_{i_1,\dots,i_{\colred k}}T_{i_1,\dots,i_k}
u_1^{i_1}\cdots u_k^{i_k}
\eeas
for $u_1=(u_1^{i_1})_{i_1}$, $\dots$, $u_k=(u_k^{i_k})_{i_k}$. 
Brackets $[\ \ ]$ stand for multilinear mappings. 
We simply  denote $u^{\otimes r}=u\otimes\cdots\otimes u$ ($r$ times).

\subsection{Interpolation and expansion}\label{170805-10}
Define $W_n(\theta)$ and $X_n(\theta)$ by 
$W_n(\theta)=\theta W_n+(1-\theta)W_\infty$ and $X_n(\theta)=\theta X_n+(1-\theta)X_\infty$, respectively,   
for $\theta\in[0,1]$. 
Moreover, the tangent variables are defined by 
$\dotw_n=r_n^{-1}(W_n-W_\infty)$ and $\dotx_n=r_n^{-1}(X_n-X_\infty)$. 
We will construct an interpolation like that of 
Nourdin, Nualart and Peccati \cite{nourdin2016quantitative} 
but in the frequency domain. 
%Let $(\sfz,\sfx)\in\bbR^\sfd\times\bbR^{\sfd_1}=\bbR^{\check{\sfd}}$, $\check{\sfd}=\sfd+\sfd_1$. 
%
Define $\lambda_n(\theta;\sfz,\sfx)$ by 
\bea\label{20160920-1} 
\lambda_n(\theta;\sfz,\sfx) &=& \theta M_n[\tti\sfz]
+2^{-1}(1- \theta^2)G_\infty[(\tti\sfz)^{\otimes2}]
+W_n(\theta)[\tti\sfz]+\theta r_nN_n[\tti\sfz]+X_n(\theta)[\tti\sfx]
\eea
for $\theta\in[0,1]$, $\sfz\in\bbR^\sfd$ and $\sfx\in\bbR^{\sfd_1}$. 
%Denote $\lambda_n(\theta;\sfz,\sfx)$ by $\lambda_n(\theta)$. 
In particular, 
\beas 
\lambda_n(0;\sfz,\sfx) &=&
2^{-1}G_\infty[(\tti\sfz)^{\otimes2}]+W_\infty[\tti\sfz]+X_\infty[\tti\sfx]
\>=:\>\lambda_\infty(0;\sfz,\sfx),
\eeas
\beas 
\lambda_n(1;\sfz,\sfx) &=&
Z_n[{\colorb \tti\sfz}]+X_n[\tti\sfx]
\eeas
and
\beas 
\big|e^{\lambda_n(\theta;\sfz,\sfx)}\big| &\leq& 1. 
\eeas
Let $\check{\sfd}=\sfd+\sfd_1$.

\begin{rem}\rm More generally, for 
$\gamma_i\in C^1([0,1];[0,1])$ such that 
$\gamma_i(0)=0$ and $\gamma_i(1)=1$ ($i=0,1,2,3$), we can consider 
an interpolation 
\beas 
\lambda_n(\theta;\sfz,\sfx) &=& \gamma_0(\theta)M_n[\tti\sfz]
+2^{-1}(1- \gamma_0(\theta)^2)G_\infty[(\tti\sfz)^{\otimes2}]
+W_n(\gamma_1(\theta))[\tti\sfz]+\gamma_2(\theta)\>r_nN_n[\tti\sfz]+X_n(\gamma_3(\theta))[\tti\sfx]
\eeas
for $\theta\in[0,1]$. 
However, it turns out that the derived formula does not depend on a choice of $\gamma_i$. So we will take the identity function for $\gamma_i$, i.e., (\ref{20160920-1}) 
as $\lambda_n(\theta;\sfz,\sfx)$. 
\end{rem}

{\colred 
\begin{rem}\rm 
We could start with the decomposition 
\beas 
Z_n&=& M_n+W_\infty+r_n\tilde{N}_n
\eeas
of $Z_n$, by taking $\tilde{N}_n=\dotw_n+N_n$. 
This decomposition would be expected to slightly simplify the presentation but 
%the tangent of $W_\infty$ is null. 
the complexity would be the same because, 
as a matter of fact, $\dotw_n$ and $N_n$ will always be treated as a set 
like in $\check{G}^{(1)}_n$ and $\hat{G}^{(1)}_n$ defined below. 
\end{rem}
}

Consider  a sequence $\psi_n\in\bbD^{1,p_1}(\bbR)$, and for a while we suppose that 
$u_n\in\bbD^{2,p}(\mfh\otimes\bbR^\sfd)$, 
$G_\infty\in\bbD^{1,p}(\bbR^\sfd\otimes_+\bbR^\sfd)$, 
$W_n,W_\infty,N_n\in\bbD^{1,p}(\bbR^\sfd)$ and  
$X_n,X_\infty\in\bbD^{1,p}(\bbR^{\sfd_1})$  with ${\colorr 2p^{-1}}+p_1^{-1}\leq1$. 
In the special case $\psi_n\equiv1$, we let $p_1=\infty$. 

We write 
\beas 
\varphi_n(\theta;\psi_n) 
{\colorb \>=\> \varphi_n(\theta,\sfz,\sfx;\psi_n) }
&=& 
E\big[e^{\lambda_n(\theta;\sfz,\sfx)}\psi_n\big]. 
\eeas
The random matrix $G_n$ $(\sfd\times\sfd)$ is defined by 
\beas 
G_n[(\tti\sfz)^{\otimes2}] &=& \langle DM_n[\tti\sfz],u_n[\tti\sfz]\rangle_\mfh. 
%G_\infty &=& S^2
\eeas

% Let 
%\beas 
%\dotg_n &=& r_n^{-1}(G_n-G_\infty). 
%\eeas

\begin{rem}\rm 
The variable $G_n$ is different  from $C_n=\langle M^n\rangle_T$ 
in the martingale expansion since, in general,
\beas 
\langle DM_n,u_n\rangle_\mfh &\not=& 
\langle u_n,u_n\rangle_\mfh.
\eeas
 That is, $\dotg_n= r_n^{-1}(G_n-G_\infty) $ is not 
necessarily the tangent variable $\dotc_n$, and  the sequences  $\dotg_n$ and $\dotc_n$ have  different limits,  in general.
However, 
the limit $G_\infty$ of $G_n$ may coincide with the limit $C_\infty$ of $C_n$. 
In short,  {\color {black} we may have} $G_\infty=C_\infty$, however, in general, $ \dotg_\infty\not=\dotc_\infty$.
In particular, $\dotg_n$ may converge to $0$. 
\end{rem}

The random tensor 
\beas 
{\colorb \tt{qTan}[(\tti\sfz)^{\otimes2}]}&=&
r_n^{-1}\bigg(\big\langle DM_n[\tti\sfz],u_n[\tti\sfz]\big\rangle_\mfh-G_\infty[(\tti\sfz)^{\otimes2}]\bigg)
\eeas
on $(\tti\sfz)^{\otimes2}\in\bbR^{\sfd}\otimes\bbR^{\sfd}$,
is called the {\bf quasi tangent} (q-tangent), and 
the random tensor 
\beas 
{\colorb \tt{qTor}[(\tti\sfz)^{\otimes3}]}&=&
r_n^{-1}\bigg\langle D\big\langle DM_n[\tti\sfz],u_n[\tti\sfz]\big\rangle_\mfh,u_n[\tti\sfz]\bigg\rangle_\mfh
\eeas
{\colorb on $(\tti\sfz)^{\otimes3}\in(\bbR^{\sfd})^{\otimes3}$} 
\noindent{\colorr when $u_n\in\bbD^{3,p}(\mfh\otimes\bbR^\sfd)$,}
is called the {\bf quasi torsion} (q-torsion). 
Moreover, we call the random tensor  
\beas 
{\colorb \tt{mqTor}[(\tti\sfz)^{\otimes3}]}&=&
r_n^{-1}\big\langle DG_\infty[(\tti\sfz)^{\otimes2}],u_n[\tti\sfz]\big\rangle_\mfh
\eeas
{\colorb on $(\tti\sfz)^{\otimes3}\in(\bbR^{\sfd})^{\otimes3}$},
the {\bf modified quasi torsion} (modified q-torsion). 
{\colorb 
Then 
\beas 
\big\langle D\>\tt{qTan}[(\tti\sfz)^{\otimes2}],u_n[\tti\sfz]\big\rangle_\mfh 
&=& 
\tt{qTor}[(\tti\sfz)^{\otimes3}]-\tt{mqTor}[(\tti\sfz)^{\otimes3}].
\eeas
}

Let 
\bea  \label{psi}
\Psi(\sfz,\sfx)&=& 
\exp\bigg(2^{-1}G_\infty[(\tti\sfz)^{\otimes2}]+W_\infty[\tti\sfz]+X_\infty[\tti\sfx]\bigg)
\>\equiv\> e^{\lambda_\infty(0;\sfz,\sfx)}.
\eea
 Then 
\beas 
E\big[\exp\big(Z_n[\tti\sfz]+X_n[\tti\sfx]\big)\psi_n\big]
-E[\Psi(\sfz,\sfx)\psi_n]
&=&
\varphi_n(1;\psi_n)-\varphi_n(0;\psi_n)
\>=\>
\int_0^1\partial_\theta \varphi_n(\theta;\psi_n)d\theta.
\eeas
The derivative of $\varphi_n(\theta;\psi_n)$ is computed as follows 
\beas 
\partial_\theta \varphi_n(\theta;\psi_n)
&=&
E\bigg[e^{\lambda_n(\theta;\sfz,\sfx)}
\bigg\{\delta(u_n[\tti\sfz])-\theta G_\infty[(\tti\sfz)^2]
+r_n\dotw_n[\tti\sfz]+r_n N_n[\tti\sfz]
+r_n\dotx_n[\tti\sfx]\bigg\}\psi_n\bigg]
\\&=& 
E\bigg[e^{\lambda_n(\theta;\sfz,\sfx)}
\bigg\{\delta(u_n[\tti\sfz])-\theta G_\infty[(\tti\sfz)^2]
+  \check{G}^{(1)}_n(\sfz,\sfx) \bigg\} \psi_n\bigg],
\eeas
where 
\beas 
 \check{G}^{(1)}_n(\sfz,\sfx)
&=&
r_n\dotw_n[\tti\sfz]+r_n N_n[\tti\sfz]
+r_n\dotx_n[\tti\sfx].
\eeas
Applying the duality relationship between the Skorohod integral $\delta$ and the derivative operator $D$ (we also call this duality relationship integration by parts  (IBP) formula), yields
\beas
 \partial_\theta \varphi_n(\theta;\psi_n)
&=&%
E\bigg[e^{\lambda_n(\theta;\sfz,\sfx)}
\bigg\{\theta\langle DM_n[\tti\sfz],u_n[\tti\sfz]\rangle_\mfh
+2^{-1}(1-\theta^2)\langle DG_\infty[(\tti\sfz)^2],u_n[\tti\sfz]\rangle_\mfh
\\&&
+\langle DW_n(\theta)[\tti\sfz],u_n[\tti\sfz]\rangle_\mfh
+\theta r_n\langle DN_n{\colorb [\tti\sfz]},u_n[\tti\sfz]\rangle_\mfh
+\langle DX_n(\theta)[\tti\sfx],u_n[\tti\sfz]\rangle_\mfh\bigg\}\psi_n\bigg]
\\&&
+E\big[e^{\lambda_n(\theta;\sfz,\sfx)}\langle D\psi_n,u_n[\tti\sfz]\rangle_\mfh\big]
\\&&
+E\bigg[e^{\lambda_n(\theta;\sfz,\sfx)}
\left\{-\theta G_\infty[(\tti\sfz)^2]
+  \check{G}^{(1)}_n(\sfz,\sfx) \right\}\psi_n\bigg].
\eeas
This expression can be written as
\beas
 \partial_\theta \varphi_n(\theta;\psi_n)
&=&%
E\big[e^{\lambda_n(\theta;\sfz,\sfx)}\langle D\psi_n,u_n[\tti\sfz]\rangle_\mfh\big]
\\&&
+\theta E\bigg[e^{\lambda_n(\theta;\sfz,\sfx)}
\bigg(\langle DM_n[\tti\sfz],u_n[\tti\sfz]\rangle_\mfh
-G_\infty[(\tti\sfz)^2]\bigg)\psi_n\bigg]
\\&&
+2^{-1}(1-\theta^2)E\bigg[e^{\lambda_n(\theta;\sfz,\sfx)}\langle DG_\infty[(\tti\sfz)^2],u_n[\tti\sfz]\rangle_\mfh\psi_n\bigg]
\\&&
+E\big[e^{\lambda_n(\theta;\sfz,\sfx)}G^{(1)}_n(\theta;\sfz,\sfx)\psi_n\big]
\\&=&%
\varphi_n\big(\theta;\langle D\psi_n,u_n[\tti\sfz]\rangle_\mfh\big)
\\&&
+\theta \varphi_n\bigg(\theta;\big(\langle DM_n[\tti\sfz],u_n[\tti\sfz]\rangle_\mfh
-G_\infty[(\tti\sfz)^2]\big)\psi_n\bigg)
\\&&
+2^{-1}(1-\theta^2)\varphi_n\bigg(\theta;\langle DG_\infty[(\tti\sfz)^2],u_n[\tti\sfz]\rangle_\mfh\psi_n\bigg)
\\&&
+\varphi_n\big(\theta;G^{(1)}_n(\theta;\sfz,\sfx)\psi_n\big),
\eeas
%%%%%%
where 
\bea \label{eq1a}
G^{(1)}_n(\theta;\sfz,\sfx)&=& \hat{G}^{(1)}_n(\theta;\sfz,\sfx)+  \check{G}^{(1)}_n(\sfz,\sfx)
\eea
 with 
\beas 
\hat{G}^{(1)}_n(\theta;\sfz,\sfx)
&=&
\langle DW_\infty[\tti\sfz],u_n[\tti\sfz]\rangle_\mfh
+\langle DX_\infty[\tti\sfx],u_n[\tti\sfz]\rangle_\mfh
\\&&
+\theta r_n\langle D\dotw_n[\tti\sfz],u_n[\tti\sfz]\rangle_\mfh
+\theta r_n\langle DN_n[\tti\sfz],u_n[\tti\sfz]\rangle_\mfh
+\theta r_n\langle D\dotx_n[\tti\sfx],u_n[\tti\sfz]\rangle_\mfh.
\eeas

\begin{en-text}
\beas 
G^{(1)}_n(\theta;\sfz,\sfx)
&=&
\langle DW_n(\theta)[\tti\sfz],u_n[\tti\sfz]\rangle_\mfh
+\theta r_n\langle DN_n[\tti\sfz],u_n[\tti\sfz]\rangle_\mfh
+\langle DX_n(\theta)[\tti\sfx],u_n[\tti\sfz]\rangle_\mfh
\\&&
+
r_n\dotw_n[\tti\sfz]+r_n N_n[\tti\sfz]
+r_n\dotx_n[\tti\sfx]
\\&=&
\langle DW_\infty[\tti\sfz],u_n[\tti\sfz]\rangle_\mfh
+\langle DX_\infty[\tti\sfx],u_n[\tti\sfz]\rangle_\mfh
\\&&
+\theta r_n\langle D\dotw_n[\tti\sfz],u_n[\tti\sfz]\rangle_\mfh
+\theta r_n\langle DN_n[\tti\sfz],u_n[\tti\sfz]\rangle_\mfh
+\theta r_n\langle D\dotx_n[\tti\sfx],u_n[\tti\sfz]\rangle_\mfh
\\&&
+
r_n\dotw_n[\tti\sfz]+r_n N_n[\tti\sfz]
+r_n\dotx_n[\tti\sfx].
\eeas
\end{en-text}

 Let 
\bea\label{20170419-1} 
G^{(2)}_n(\sfz) &=& 
\langle DM_n[\tti\sfz],u_n[\tti\sfz]\rangle_\mfh-G_\infty[(\tti\sfz)^{\otimes2}]
{\colorb\>=\> r_n\>\tt{qTan}[(\tti\sfz)^{\otimes2}]}
\eea
and 
\bea\label{20170419-2} 
G^{(3)}_n(\sfz) &=& 
\langle DG_\infty[(\tti\sfz)^{\otimes2}],u_n[\tti\sfz]\rangle_\mfh
{\colorb \>=\> r_n\>\tt{mqTor}[(\tti\sfz)^{\otimes3}]}. 
\eea
Thus, we obtained the following lemma. 
\begin{lemme}\label{20160920-5}
Suppose that $u_n\in\bbD^{2,p}(\mfh\otimes\bbR^\sfd)$, 
$G_\infty\in\bbD^{1,p}(\bbR^\sfd\otimes_+\bbR^\sfd)$, 
$W_n,W_\infty,N_n\in\bbD^{1,p}(\bbR^\sfd)$, 
$X_n,X_\infty\in\bbD^{1,p}(\bbR^{\sfd_1})$ and 
$\psi_n\in\bbD^{1,p_1}(\bbR)$, ${\colorr 2p^{-1}}+p_1^{-1}\leq1$. 
Then 
\bea\label{20160922-1} 
\partial_\theta \varphi_n(\theta;\psi_n)
\nn&=&%
\varphi_n\big(\theta;\langle D\psi_n,u_n[\tti\sfz]\rangle_\mfh\big)
+\theta \varphi_n\big(\theta;G^{(2)}_n(\sfz)\psi_n\big)
\\&&
+2^{-1}(1-\theta^2)\varphi_n\big(\theta;G^{(3)}_n(\sfz)\psi_n\big)
+\varphi_n\big(\theta;G^{(1)}_n(\theta;\sfz,\sfx)\psi_n\big).
\eea
\end{lemme}

\vspace*{3mm}
In order to establish a second-order interpolation formula we need to  further expand  the last three summands in the right-hand side of
{\colred (\ref{20160922-1})}.  To do this, we  will denote by $\calg_n$ any one of the terms
$G^{(1)}_n(\theta;\sfz,\sfx)$, 
$G^{(2)}_n(\sfz)$ and $G^{(3)}(\sfz)$. 
Suppose, in addition,  that 
$G^{(1)}_n(\theta;\sfz,\sfx)\psi_n$, 
$G^{(2)}_n(\sfz)\psi_n$, $G^{(3)}(\sfz)\psi_n$, 
$\langle D(G^{(1)}_n(\theta;\sfz,\sfx)\psi_n),u_n[\tti\sfz]\rangle_\mfh$,  \newline
$\langle D(G^{(2)}_n(\theta;\sfz)\psi_n),u_n[\tti\sfz]\rangle_\mfh$ and 
$\langle D(G^{(3)}_n(\theta;\sfz)\psi_n),u_n[\tti\sfz]\rangle_\mfh$
are in $\bbD^{1,{\colorr p_2}}(\bbR)$ {\colorr with $p_2=(3p^{-1}+p_1^{-1})^{-1}$ 
with $p$ such that $5p^{-1}+p_1^{-1}\leq1$}. 
Then, by Lemma \ref{20160920-5}, we have 
\beas
\varphi_n\big(\theta;\calg_n\psi_n\big)-\varphi_n\big(0;\calg_n\psi_n\big)
&=&%
\int_0^\theta\partial_{\theta_1} \varphi_n\big(\theta_1;\calg_n\psi_n\big)d\theta_1
\\&=&%
\int_0^\theta\bigg\{
\varphi_n\big(\theta_1;\langle D(\calg_n\psi_n),u_n[\tti\sfz]\rangle_\mfh\big)
+\theta_1 \varphi_n\big(\theta_1;G^{(2)}_n(\sfz,\sfx)\calg_n\psi_n\big)
\\&&
+2^{-1}(1-\theta_1^2)\varphi_n\big(\theta_1;G^{(3)}_n(\sfz,\sfx)\calg_n\psi_n\big)
+\varphi_n\big(\theta_1;G^{(1)}_n(\theta_1;\sfz,\sfx)\calg_n\psi_n\big)
\bigg\}d\theta_1
\\&=&%
\int_0^\theta
\varphi_n\big(\theta_1;\langle D(\calg_n\psi_n),u_n[\tti\sfz]\rangle_\mfh\big)
d\theta_1
+R^{(1)}_n(\theta;\sfz,\sfx,\calg_n),
\eeas
where 
\beas 
R^{(1)}_n(\theta;\sfz,\sfx,\calg_n)
&=&
\int_0^\theta\bigg\{\theta_1 \varphi_n\big(\theta_1;G^{(2)}_n(\sfz)\calg_n\psi_n\big)
+2^{-1}(1-\theta_1^2)\varphi_n\big(\theta_1;G^{(3)}_n(\sfz)\calg_n\psi_n\big)
\\&&
+\varphi_n\big(\theta_1;G^{(1)}_n(\theta_1;\sfz,\sfx)\calg_n\psi_n\big)
\bigg\}d\theta_1.
\eeas
Therefore, once again by Lemma \ref{20160920-5},  we obtain 
\bea\label{20160920-3}
\varphi_n\big(\theta;\calg_n\psi_n\big)-\varphi_n\big(0;\calg_n\psi_n\big)
 &=& %
\theta \varphi_n\big(0;\langle D(\calg_n\psi_n),u_n[\tti\sfz]\rangle_\mfh\big)
+R^{(2)}_n(\theta;\sfz,\sfx,\calg_n),
\eea
where
\beas 
R^{(2)}_n(\theta;\sfz,\sfx,\calg_n)
&=&%
\int_0^\theta\int_0^{\theta_1}  
\partial_{\theta_2}\varphi_n\big(\theta_2;\langle D(\calg_n\psi_n),u_n[\tti\sfz]\rangle_\mfh\big)d\theta_2{\colorb d\theta_1}
+R^{(1)}_n(\theta;\sfz,\sfx,\calg_n)
\\&=&%
\int_0^\theta\int_0^{\theta_1}  
\bigg\{
\varphi_n\big(\theta_2;\langle D\big(\langle D(\calg_n\psi_n),u_n[\tti\sfz]\rangle_\mfh\big),u_n[\tti\sfz]\rangle_\mfh\big)
\\&&
+\theta_2 \varphi_n\big(\theta_2;G^{(2)}_n(\sfz)\langle D(\calg_n\psi_n),u_n[\tti\sfz]\rangle_\mfh\big)
\\&&
+2^{-1}(1-\theta_2^2)\varphi_n\big(\theta_2;G^{(3)}_n(\sfz)\langle D(\calg_n\psi_n),u_n[\tti\sfz]\rangle_\mfh\big)
\\&&
+\varphi_n\big(\theta_2;G^{(1)}_n(\theta_2;\sfz,\sfx)\langle D(\calg_n\psi_n),u_n[\tti\sfz]\rangle_\mfh\big)
\bigg\}d\theta_2{\colorb d\theta_1}
+R^{(1)}_n(\theta;\sfz,\sfx,\calg_n).
\eeas
By (\ref{20160922-1}) and (\ref{20160920-3}),  we can write
\beas 
\partial_\theta \varphi_n(\theta;\psi_n)
&=&%
\varphi_n\big(\theta;\langle D\psi_n,u_n[\tti\sfz]\rangle_\mfh\big)
\\&&
+\theta\bigg\{ \varphi_n\big(0;G^{(2)}_n(\sfz)\psi_n\big)
+\theta \varphi_n\big(0;\langle D(G^{(2)}_n(\sfz)\psi_n),u_n[\tti\sfz]\rangle_\mfh\big)
+R^{(2)}_n(\theta;\sfz,\sfx,G^{(2)}_n(\sfz))\bigg\}
\\&&
+2^{-1}(1-\theta^2)\bigg\{\varphi_n\big(0;G^{(3)}_n(\sfz)\psi_n\big)
+\theta \varphi_n\big(0;\langle D(G^{(3)}_n(\sfz)\psi_n),u_n[\tti\sfz]\rangle_\mfh\big)
+R^{(2)}_n(\theta;\sfz,\sfx,G^{(3)}_n(\sfz))\bigg\}
\\&&
+\bigg\{\varphi_n\big(0;G^{(1)}_n(\theta;\sfz,\sfx)\psi_n\big)
+\theta \varphi_n\big(0;\langle D(G^{(1)}_n(\theta;\sfz,\sfx)\psi_n),u_n[\tti\sfz]\rangle_\mfh\big)
+R^{(2)}_n(\theta;\sfz,\sfx,G^{(1)}_n(\theta;\sfz,\sfx))\bigg\}.
\eeas
Let 
{\colorg 
\beas 
R^{(3)}_n(\sfz,\sfx)
&=&
\int_0^1 \bigg\{\varphi_n\big(\theta;\langle D\psi_n,u_n[\tti\sfz]\rangle_\mfh\big)
+\theta R^{(2)}_n(\theta;\sfz,\sfx,G^{(2)}_n(\sfz))
+2^{-1}(1-\theta^2)R^{(2)}_n(\theta;\sfz,\sfx,G^{(3)}_n(\sfz))
\\&&
+R^{(2)}_n(\theta;\sfz,\sfx,G^{(1)}_n(\theta;\sfz,\sfx))
\\&&
+2^{-1}(1-\theta^2)\theta \varphi_n\big(0;\langle D(G^{(3)}_n(\sfz)\psi_n),u_n[\tti\sfz]\rangle_\mfh\big)
+\theta \varphi_n\big(0;\langle D(\hat{G}^{(1)}_n(\theta;\sfz,\sfx)\psi_n),u_n[\tti\sfz]\rangle_\mfh\big)
\bigg\}d\theta.
%\\&&
%+\frac{1}{3} \varphi_n\big(0;G^{(2)}_n(\sfz)\langle D\psi_n,u_n[\tti\sfz]\rangle_\mfh\big).
\eeas
}
Then, integrating    $\partial_\theta \varphi_n(\theta;\psi_n)$ and using the expression for  $R^{(3)}_n(\sfz,\sfx)$  and  the decomposition {\colred (\ref{eq1a})}, yields
\beas 
\int_0^1 \partial_\theta \varphi_n(\theta;\psi_n)d\theta
&=&%
\half \varphi_n\big(0;G^{(2)}_n(\sfz)\psi_n\big)
+\frac{1}{3} \varphi_n\big(0;\langle DG^{(2)}_n(\sfz),u_n[\tti\sfz]\rangle_\mfh\psi_n\big)
+\frac{1}{3} \varphi_n\big(0;G^{(3)}_n(\sfz)\psi_n\big)
\\&&
+\int_0^1\varphi_n\big(0;G^{(1)}_n(\theta;\sfz,\sfx)\psi_n\big)d\theta+R^{(3)}_n(\sfz,\sfx)
\\&&
{\colorg 
+\int_0^1 \theta\varphi_n\big(0;\langle D(\check{G}^{(1)}_n(\sfz,\sfx)\psi_n),u_n[\tti\sfz]\rangle_\mfh\big)d\theta
%+R^{(4)}_n(\sfz,\sfx)
}
\\&=&%
\half \varphi_n\big(0;G^{(2)}_n(\sfz)\psi_n\big)
+\frac{1}{3} \varphi_n\big(0;\big\langle D\langle DM_n[\tti\sfz],u_n[\tti\sfz]\rangle_\mfh,u_n[\tti\sfz]\big\rangle_\mfh\psi_n\big)
\\&&
+\int_0^1\varphi_n\big(0;G^{(1)}_n(\theta;\sfz,\sfx)\psi_n\big)d\theta
{\colorg +
\half\varphi_n\big(0;\langle D(\check{G}^{(1)}_n(\sfz,\sfx)\psi_n),u_n[\tti\sfz]\rangle_\mfh\big)
}
\\&&
+R^{(3)}_n(\sfz,\sfx).
\eeas
Consequently, 
\beas 
%E\big[\exp\big(Z_n[\tti\sfz]+X_n[\tti\sfx]\big)\psi_n\big]-E[\Psi(\sfz,\sfx)\psi_n]
\varphi_n(1;\psi_n)-\varphi_n(0;\psi_n)
&=&
\half \varphi_n\big(0;\big(DM_n[\tti\sfz],u_n[\tti\sfz]\rangle_\mfh-G_\infty[(\tti\sfz)^2]\big)\psi_n\big)
\\&&
+\frac{1}{3} \varphi_n\big(0;\big\langle D\langle DM_n[\tti\sfz],u_n[\tti\sfz]\rangle_\mfh,u_n[\tti\sfz]\big\rangle_\mfh\psi_n\big)
\\&&
+\varphi_n\bigg(0;\int_0^1G^{(1)}_n(\theta;\sfz,\sfx)d\theta\psi_n\bigg)
{\colorg +
\half\varphi_n\big(0;\langle D(\check{G}^{(1)}_n(\sfz,\sfx)\psi_n),u_n[\tti\sfz]\rangle_\mfh\big)
}
\\&&
+R^{(3)}_n(\sfz,\sfx).
\eeas
Using  the definition of $\Psi(\sfz,\sfx)$ given in (\ref{psi}) and the decomposition of $G_n^{(1)}$ given in {\colred (\ref{eq1a})}, we obtain
\beas
\varphi_n(1;\psi_n)-\varphi_n(0;\psi_n)&=&%
\frac{1}{3}E\bigg[\Psi(\sfz,\sfx)\bigg\langle D\big\langle DM_n[\tti\sfz],u_n[\tti\sfz]\big\rangle_\mfh,u_n[\tti\sfz]\bigg\rangle_\mfh\psi_n\bigg]
\\&&
+\half E\bigg[\Psi(\sfz,\sfx)\bigg(\big\langle DM_n[\tti\sfz],u_n[\tti\sfz]\big\rangle_\mfh-G_\infty[(\tti\sfz)^2]\bigg)\psi_n\bigg]
\\&&
+E\bigg[\Psi(\sfz,\sfx)\bigg\{
\big\langle DW_\infty[\tti\sfz],u_n[\tti\sfz]\big\rangle_\mfh
+\big\langle DX_\infty[\tti\sfx],u_n[\tti\sfz]\big\rangle_\mfh
\\&&
\qquad\qquad\qquad+
r_n\dotw_n[\tti\sfz]+r_n N_n[\tti\sfz]
+r_n\dotx_n[\tti\sfx]
\\&&
\qquad\qquad\qquad{\colorg 
+r_n
\langle D\dotw_n[\tti\sfz],u_n[\tti\sfz]\rangle_\mfh
+r_n\langle DN_n[\tti\sfz],u_n[\tti\sfz]\rangle_\mfh
}
\\&&
\qquad\qquad\qquad{\colorg 
+r_n\langle D\dotx_n[\tti\sfx],u_n[\tti\sfz]\rangle_\mfh
}
\bigg\}
\psi_n\bigg]
\\&&
{\colorb +R^{(3)}_n(\sfz,\sfx)}
+R^{(4)}_n(\sfz,\sfx),
\eeas
where
{\colorg 
\beas
R^{(4)}_n(\sfz,\sfx) 
&=&
\half\varphi_n\big(0;\check{G}^{(1)}_n(\sfz,\sfx)   \langle D\psi_n,u_n[\tti\sfz]\rangle_\mfh\big).
\eeas
}
\begin{en-text}
where
\beas 
{\colorb R^{(4)}_n(\sfz,\sfx) }
&=& 
2^{-1} r_n\varphi_n(0;\langle D\dotw_n[\tti\sfz],u_n[\tti\sfz]\rangle_\mfh\psi_n)
+2^{-1}  r_n\varphi_n(0;\langle DN_n,u_n[\tti\sfz]\rangle_\mfh\psi_n)
\\&&
+2^{-1}  r_n\varphi_n(0;\langle D\dotx_n[\tti\sfx],u_n[\tti\sfz]\rangle_\mfh\psi_n)
\eeas
\end{en-text}
%%

% Suppose $\psi_n$ is bounded. %takes values in $[0,1]$. 
{\colorr 
Suppose that there are random symbols 
${\mathfrak S}^{(3,0)}$, %\in L^1(\Omega;\bbR^\sfd\otimes\bbR^\sfd\otimes\bbR^\sfd)$, 
${\mathfrak S}^{(2,0)}_0$, %\in L^1(\Omega;\bbR^\sfd\otimes\bbR^\sfd)$, 
${\mathfrak S}^{(2,0)}$, %\in L^1(\Omega;\bbR^\sfd\otimes\bbR^\sfd)$, 
${\mathfrak S}^{(1,1)}$, %\in L^1(\Omega;\bbR^\sfd\otimes\bbR^{\sfd_1})$, 
${\mathfrak S}^{(1,0)}$, %\in L^1(\Omega;\bbR^\sfd)$ 
${\mathfrak S}^{(0,1)}$, %\in L^1(\Omega;\bbR^{\sfd_1})$. 
{\colorg 
${\mathfrak S}^{(2,0)}_1$
and 
${\mathfrak S}^{(1,1)}_1$, 
}
i.e., they are polynomials in $(\tti\sfz,\tti\sfx)$ with coefficients in $L^1(\Omega)$. 
}
 Let 
\beas 
R^{(5)}_n(\sfz,\sfx)
&=&
{\colorb 
r_n E\bigg[\Psi(\sfz,\sfx)\>3^{-1}{\sf qTor}[(\tti\sfz)^{\otimes3}]
%\bigg\langle D\big\langle DM_n[\tti\sfz],u_n[\tti\sfz]\big\rangle_\mfh,u_n[\tti\sfz]\bigg\rangle_\mfh
\psi_n\bigg]
}
-
r_n E\big[\Psi(\sfz,\sfx){\mathfrak S}^{(3,0)}(\tti\sfz,\tti\sfx)%[(\tti\sfz)^{\otimes3}]
\big],
\eeas
\beas 
R^{(6)}_n(\sfz,\sfx)
&=&
{\colorb 
r_n E\bigg[\Psi(\sfz,\sfx)\>2^{-1}{\sf qTan}[(\tti\sfz)^{\otimes2}]
%\bigg(\big\langle DM_n[\tti\sfz],u_n[\tti\sfz]\big\rangle_\mfh-G_\infty[(\tti\sfz)^2]\bigg)
\psi_n\bigg]
}
-
r_n E\big[\Psi(\sfz,\sfx){\mathfrak S}^{(2,0)}_0(\tti\sfz,\tti\sfx)%[(\tti\sfz)^{\otimes2}]
\big],
\eeas
\beas
R^{(7)}_n(\sfz,\sfx)
&=&
E\bigg[\Psi(\sfz,\sfx)\big\langle DW_\infty[\tti\sfz],u_n[\tti\sfz]\big\rangle_\mfh\psi_n\bigg]
-
r_n E\big[\Psi(\sfz,\sfx){\mathfrak S}^{(2,0)}(\tti\sfz,\tti\sfx)%[(\tti\sfz)^{\otimes2}]
\big],
\eeas
\beas
R^{(8)}_n(\sfz,\sfx)
&=&
E\bigg[\Psi(\sfz,\sfx)\big\langle DX_\infty[\tti\sfx],u_n[\tti\sfz]\big\rangle_\mfh\psi_n\bigg]
-
r_n E\big[\Psi(\sfz,\sfx){\mathfrak S}^{(1,1)}(\tti\sfz,\tti\sfx)%[\tti\sfz,\tti\sfx]
\big],
\eeas
\beas
R^{(9)}_n(\sfz,\sfx)
&=&
r_n E\bigg[\Psi(\sfz,\sfx)\bigg\{\dotw_n[\tti\sfz]+N_n[\tti\sfz]\bigg\}\psi_n\bigg]
-
r_n E\big[\Psi(\sfz,\sfx){\mathfrak S}^{(1,0)}(\tti\sfz,\tti\sfx)%[\tti\sfz]
\big],
\eeas
\beas
R^{(10)}_n(\sfz,\sfx)
&=&
r_n E\bigg[\Psi(\sfz,\sfx)\dotx_n[\tti\sfx]\psi_n\bigg]
-
r_n E\big[\Psi(\sfz,\sfx){\mathfrak S}^{(0,1)}(\tti\sfz,\tti\sfx)%[\tti\sfx]
\big],
\eeas
{\colorg 
\beas
R^{(11)}_n(\sfz,\sfx)
&=&
r_n E\bigg[\Psi(\sfz,\sfx)\big\langle \{D\dotw_n[\tti\sfz]+DN_n[\tti\sfz]\},u_n[\tti\sfz]\big\rangle_\HH
\psi_n\bigg]
-
r_n E\big[\Psi(\sfz,\sfx){\mathfrak S}^{(2,0)}_1(\tti\sfz,\tti\sfx),%[\tti\sfz]
\big],
\eeas
}
{\colorg 
\beas
R^{(12)}_n(\sfz,\sfx)
&=&
r_n E\bigg[\Psi(\sfz,\sfx)\big\langle D\dotx_n[\tti\sfx],u_n[\tti\sfz]\big\rangle_\HH
\psi_n\bigg]
-
r_n E\big[\Psi(\sfz,\sfx){\mathfrak S}^{(1,1)}_1(\tti\sfz,\tti\sfx)%[\tti\sfz]
\big],
\eeas
}
and 
\beas 
R_n(\sfz,\sfx) &=& 
{\colorb 
\sum_{i=3}^{12}R_n^{(i)}(\sfz,\sfx). }
\eeas

{\colred 
\begin{rem}\label{20171117-1}\rm 
We expect 
\beas
E\big[\Psi(\sfz,\sfx)\>3^{-1}{\sf qTor}[(\tti\sfz)^{\otimes3}]\psi_n\big]
-
E\big[\Psi(\sfz,\sfx){\mathfrak S}^{(3,0)}(\tti\sfz,\tti\sfx)%[(\tti\sfz)^{\otimes3}]
\big]\to0. 
\eeas 
However this does not mean that 
\beas 
E\big[\big|3^{-1}{\sf qTor}[(\tti\sfz)^{\otimes3}]\psi_n-{\mathfrak S}^{(3,0)}(\tti\sfz,\tti\sfx)\big|\big]
\to0.
\eeas
In fact, ${\mathfrak S}^{(3,0)}$ is not necessarily of third order in $\sfz$, 
as we will see in this paper. 
\end{rem}
}

{\colorr 
Define the random symbol  ${\mathfrak S}(\tti\sfz,\tti\sfx)$ by 
\beas 
{\mathfrak S}(\tti\sfz,\tti\sfx)
&=&
{\mathfrak S}^{(3,0)}(\tti\sfz,\tti\sfx)%[(\tti\sfz)^{\otimes3}]
+{\mathfrak S}^{(2,0)}_0(\tti\sfz,\tti\sfx)%[(\tti\sfz)^{\otimes2}]
+{\mathfrak S}^{(2,0)}(\tti\sfz,\tti\sfx)%[(\tti\sfz)^{\otimes2}]
\\&&
+{\mathfrak S}^{(1,1)}(\tti\sfz,\tti\sfx)%[(\tti\sfz)^{\otimes2}]
+{\mathfrak S}^{(1,0)}(\tti\sfz,\tti\sfx)%[\tti\sfz]
+{\mathfrak S}^{(0,1)}(\tti\sfz,\tti\sfx)%[\tti\sfx]
\\&&
{\colorg
+{\mathfrak S}^{(2,0)}_1(\tti\sfz,\tti\sfx)
+{\mathfrak S}^{(1,1)}_1(\tti\sfz,\tti\sfx).
}
\eeas
}

We will assume  the following hyporthesis:
\bd\im[[A\!\!]]
$u_n\in\bbD^{{\colorg 4},p}(\mfh\otimes\bbR^\sfd)$, 
$G_\infty\in\bbD^{3,p}(\bbR^\sfd\otimes_+\bbR^\sfd)$, 
$W_n,W_\infty,N_n\in\bbD^{3,p}(\bbR^\sfd)$, 
$X_n,X_\infty\in\bbD^{3,p}(\bbR^{\sfd_1})$ and 
$\psi_n\in\bbD^{2,p_1}(\bbR)$ for some $p$ and $p_1$ satisfying $5p^{-1}+p_1^{-1}\leq1$. 
\ed

{From the above argument, we obtain the  following second-order interpolation formula.} 
\begin{prop}\label{170806-1}
Under Condition $[A]$, 
\bea \label{inter}
\varphi_n(1;\psi_n)
&=&
\varphi_n(0;\psi_n)
%+r_n E\big[\Psi(\sfz,\sfx){\mathfrak S}(\tti\sfz,\tti\sfx) \psi_n\big]
+r_n {\colred E\big[\Psi(\sfz,\sfx){\mathfrak S}(\tti\sfz,\tti\sfx)\big]}
+R_n(\sfz,\sfx).
\eea
\end{prop}
\section{Asymptotic expansion for differentiable functions}\label{170813-1}
\subsection{\colorb Expansion formula for $E[f(Z_n,X_n)]$}
Let $f\in\cals(\bbR^{\check{\sfd}})$, the set of rapidly decreasing smooth functions. 
Let $\varsigma(\tti\sfz,\tti\sfx)=\sum_{k,m} c_{k,m}[(\tti\sfz)^{\otimes k}\otimes(\tti\sfx)^{\otimes m}]$ (finite sum) 
be a random symbol 
with $L^1$ coefficients $c_{k,m}$.  The factor $\psi_n$, if exists, can be included in $\varsigma$. 
Let $\zeta\sim N_\sfd(0,I_\sfd)$ be a random vector independent of $\calf$, defined in an extended the probability space, if necesssary. 
We denote by $\hat{f}$  the Fourier transform of $f$: 
\beas 
\hat{f}(\sfz,\sfx) &=& %(2\pi)^{-\check{\sfd}}
\int_{\bbR^{\check{\sfd}}}
f(z,x)e^{-\tti z\cdot\sfz-\tti x\cdot\sfx}dzdx.
\eeas
Then, we can  write
\bea  \nn
&&(2\pi)^{-\check{\sfd}}
\int_{\bbR^{\check{\sfd}}}\hat{f}(\sfz,\sfx)\varphi_n(\theta;\varsigma(\tti\sfz,\tti\sfx))
d\sfz d\sfx
\\ &=&   \label{eq2a}
E\bigg[\varsigma(\partial_z,\partial_x)
f\bigg(\theta M_n+\sqrt{1-\theta^2}G_\infty^{1/2}\zeta+W_n(\theta)
+\theta r_nN_n,X_n(\theta)\bigg)\bigg].
\eea
In particular, {\colorb for $\theta=0$,} we obtain 
\bea\label{170417-1}
{\colorb \rho_n^{(8)}(f)} 
&{\colorb :=}& 
(2\pi)^{-\check{\sfd}}
\int_{\bbR^{\check{\sfd}}}\hat{f}(\sfz,\sfx)R^{(8)}_n(\sfz,\sfx)d\sfz d\sfx
\nn\\&=&
E\big[\psi_n\big\langle DX_\infty[\partial_x],u_n[\partial_z]\big\rangle_\mfh 
f\big(G_\infty^{1/2}\zeta+W_\infty,X_\infty\big)\big]
\nn\\&&
-
r_n E\big[{\mathfrak S}^{(1,1)}{\colorr (\partial_z,\partial_x)}%[\partial_z^{\otimes2}]
f\big(G_\infty^{1/2}\zeta+W_\infty,X_\infty\big)\big]
\eea
and similar formulas for 
{\colorb 
\bea\label{170814-1}
\rho_n^{(i)}(f)
&:=& 
(2\pi)^{-\check{\sfd}}
\int_{\bbR^{\check{\sfd}}}\hat{f}(\sfz,\sfx)R^{(i)}_n(\sfz,\sfx)d\sfz d\sfx
\eea
}%$R^{(i)}_n$ 
for $i=5, 6, 7, 9$, $10$, {\colorg $11$ and $12$. 
For $i=3,4$, we define $\rho_n^{(i)}$ by (\ref{170814-1}). 
}
\begin{en-text}
{\colorb 
Let 
$\rho_n^{(2)}(f) = \sum_{i=3}^{10}\rho_n^{(i)}(f)$. 
}
\end{en-text}
{\colorb
Then 
\bea \label{eq3}
\sum_{i=3}^{{\colorg 12}}\rho_n^{(i)}(f) &=& (2\pi)^{-\check{\sfd}}\int_{\bbR^{\check{\sfd}}}\hat{f}(\sfz,\sfx)R_n(\sfz,\sfx)d\sfz d\sfx. 
\eea 
}
{\colorb 
Let 
\bea\label{20170504-1} 
\rho^{(2)}_n(f)
&=&
E[f(Z_n,X_n)(1-\psi_n)]-E[f(G_\infty^{1/2}\zeta+W_\infty,X_\infty)(1-\psi_n)]
\eea
and let 
\bea\label{20170504-2}
\rho^{(1)}_n(f)
&=&
\sum_{i=2}^{{\colorg 12}}\colorb \rho^{(i)}_n(f). 
%\\&&
%+(2\pi)^{-\check{\sfd}}\int_{\bbR^{\check{\sfd}}}\hat{f}(\sfz,\sfx)R_n(\sfz,\sfx)d\sfz d\sfx. 
\eea
}
Applying the second-order interpolation formula  (\ref{inter}), we can write
\beas
E[f(Z_n,X_n)] &=& E[f(Z_n,X_n)(1- \psi_n)] + E[f(Z_n,X_n)\psi_n]
\\&=& 
E[f(Z_n,X_n)(1- \psi_n)] +
(2\pi)^{-\check{\sfd}}
\int_{\bbR^{\check{\sfd}}}\hat{f}(\sfz,\sfx) \varphi_n(1; \psi_n)  d\sfz d\sfx 
\\&=& E[f(Z_n,X_n)(1- \psi_n)] 
\\&&+
(2\pi)^{-\check{\sfd}}
\int_{\bbR^{\check{\sfd}}}\hat{f}(\sfz,\sfx) \Big[ \varphi_n(0; \psi_n)   
+r_n {\colred E\big[\Psi(\sfz,\sfx){\mathfrak S}(\tti\sfz,\tti\sfx)\big]}+R_n(\sfz,\sfx) \Big]d\sfz d\sfx.
\eeas
Then, using {\colred (\ref{eq2a})} and (\ref{eq3})  leads to
\beas
E[f(Z_n,X_n)] &=&
E[f(Z_n,X_n)(1- \psi_n)] + E[ f\big(G_\infty^{1/2}\zeta+W_\infty,X_\infty\big) \psi_n]
\\&&+r_n {\colred E[ {\mathfrak S}(\partial_z,\partial_x)   f\big(G_\infty^{1/2}\zeta+W_\infty,X_\infty\big)]}
+\sum_{i=3}^{12}\rho_n^{(i)}(f).
\eeas
Finally, taking into account the definitions of $\rho^{(2)}_n(f)$ and $\rho^{(1)}_n(f)$ given in  (\ref{20170504-1}) and (\ref{20170504-2}), respectively, we get
\bea\label{20160926-1} 
E[f(Z_n,X_n)]
&=& 
E[ f\big(G_\infty^{1/2}\zeta+W_\infty,X_\infty\big)]+
r_n{\colred E[{\mathfrak S}(\partial_z,\partial_x)f\big(G_\infty^{1/2}\zeta+W_\infty,X_\infty\big)]}
+\rho^{(1)}_n(f)
\eea
for $f\in\cals(\bbR^{\check{\sfd}})$. 
%\end{theorem}

%\newpage
Let 
%\beas 
$\zeta_n(\theta)
=
\big(\theta M_n+\sqrt{1-\theta^2}G_\infty^{1/2}\zeta+W_n(\theta)
+\theta r_nN_n,X_n(\theta)\big)
$. %\eeas
In particular, $\zeta_n(0)=(G_\infty^{1/2}\zeta+W_\infty,X_\infty)=:\zeta_\infty(0)$. 
Let $\partialbs_z=\tti^{-1}\partial_z$ and $\partialbs_x=\tti^{-1}\partial_x$. 
 Among the summands in (\ref{20170504-2}), $\rho_n^{(2)}(f)$ is given by (\ref{20170504-1}). 
 A  precise expression of $\rho^{(3)}_n(f)$ is as follows
\beas 
 \rho^{(3)}_n(f)
&=&
\int_0^1 \bigg\{
E\big[\langle D\psi_n,u_n[\partial_z]\rangle_\mfh f(\zeta_n(\theta))\big]
+\theta   \rho^{{\colred (3,1)}}_n(f;\theta)
 +2^{-1}(1-\theta^2)  \rho^{{\colred (3,2)}}_n(f;\theta)
\\&&
+ \rho^{{\colred (3,3)}}_n(f;\theta)
 +2^{-1}(1-\theta^2)\theta 
E\big[\langle D(G^{(3)}_n(\partialbs_z)\psi_n),u_n[\partial_z]\rangle_\mfh f(\zeta_\infty(0))\big]
\\&&
\hspace{10mm}+\theta 
E\big[\langle D(\hat{G}^{(1)}_n(\theta;\partialbs_z,\partialbs_x)\psi_n),u_n[\partial_z]\rangle_\mfh f(\zeta_\infty(0))\big]
\bigg\}d\theta,
\eeas
where
\beas
 \rho^{(3,1)}_n(f;\theta) &=&
\int_0^\theta\int_0^{\theta_1}  
\bigg\{
E\bigg[\bigg\langle D\bigg(\big\langle D(G^{(2)}_n(\partialbs_z)\psi_n),u_n[\partial_z]\big\rangle_\mfh\bigg),u_n[\partial_z]\bigg\rangle_\mfh f(\zeta_n(\theta_2))\bigg]
\\&&
\hspace{20mm}+\theta_2 
E\big[G^{(2)}_n(\partialbs_z)\big\langle D(G^{(2)}_n(\partialbs_z)\psi_n),u_n[\partial_z]\big\rangle_\mfh f(\zeta_n(\theta_2))\big]
\\&&
\hspace{20mm}+2^{-1}(1-\theta_2^2)
E\big[G^{(3)}_n(\partialbs_z)\big\langle D(G^{(2)}_n(\partialbs_z)\psi_n),u_n[\partial_z]\big\rangle_\mfh f(\zeta_n(\theta_2))\big]
\\&&
\hspace{20mm}+
E\big[G^{(1)}_n(\theta_2;\partialbs_z,\partialbs_x)\langle D(G^{(2)}_n(\partialbs_z)\psi_n),u_n[\partial_z]\rangle_\mfh f(\zeta_n(\theta_2))\big]
\bigg\}d\theta_2
\\&&
\hspace{10mm}+ \int_0^\theta\bigg\{
\theta_1E\big[\psi_nG^{(2)}_n(\partialbs_z)G^{(2)}_n(\partialbs_z) f(\zeta_n(\theta_1))\big]
\\&&
\hspace{20mm}+2^{-1}(1-\theta_1^2)E\big[\psi_nG^{(3)}_n(\partialbs_z)G^{(2)}_n(\partialbs_z) f(\zeta_n(\theta_1))\big]
\\&&
\hspace{20mm}+E\big[\psi_nG^{(1)}_n(\theta_1;\partialbs_z,\partialbs_x)G^{(2)}_n(\partialbs_z) f(\zeta_n(\theta_1))\big]\bigg\}d\theta_1,
 \eeas

\beas
 \rho^{(3,2)}_n(f;\theta) &=&
\int_0^\theta\int_0^{\theta_1}  
\bigg\{
E\bigg[\bigg\langle D\bigg(\big\langle D(G^{(3)}_n(\partialbs_z)\psi_n),u_n[\partial_z]\big\rangle_\mfh\bigg),u_n[\partial_z]\bigg\rangle_\mfh f(\zeta_n(\theta_2))\bigg]
\\&&
\hspace{20mm}+\theta_2 
E\big[G^{(2)}_n(\partialbs_z)\big\langle D(G^{(3)}_n(\partialbs_z)\psi_n),u_n[\partial_z]\big\rangle_\mfh f(\zeta_n(\theta_2))\big]
\\&&
\hspace{20mm}+2^{-1}(1-\theta_2^2)
E\big[G^{(3)}_n(\partialbs_z)\big\langle D(G^{(3)}_n(\partialbs_z)\psi_n),u_n[\partial_z]\big\rangle_\mfh f(\zeta_n(\theta_2))\big]
\\&&
\hspace{20mm}+
E\big[G^{(1)}_n(\theta_2;\partialbs_z,\partialbs_x)\langle D(G^{(3)}_n(\partialbs_z)\psi_n),u_n[\partial_z]\rangle_\mfh f(\zeta_n(\theta_2))\big]
\bigg\}d\theta_2
\\&&
\hspace{10mm}+ \int_0^\theta\bigg\{
\theta_1E\big[\psi_nG^{(2)}_n(\partialbs_z)G^{(3)}_n(\partialbs_z) f(\zeta_n(\theta_1))\big]
\\&&
\hspace{20mm}+2^{-1}(1-\theta_1^2)E\big[\psi_nG^{(3)}_n(\partialbs_z)G^{(3)}_n(\partialbs_z) f(\zeta_n(\theta_1))\big]
\\&&
\hspace{20mm}+E\big[\psi_nG^{(1)}_n(\theta_1;\partialbs_z,\partialbs_x)G^{(3)}_n(\partialbs_z) f(\zeta_n(\theta_1))\big]\bigg\}d\theta_1,
\eeas
and
\beas
 \rho^{(3,3)}_n(f;\theta) &=&
\int_0^\theta\int_0^{\theta_1}  
\bigg\{
E\bigg[\bigg\langle D\bigg(\big\langle D(G^{(1)}_n(\theta;\partialbs_z,\partialbs_x)\psi_n),u_n[\partial_z]\big\rangle_\mfh\bigg),u_n[\partial_z]\bigg\rangle_\mfh f(\zeta_n(\theta_2))\bigg]
\\&&
\hspace{20mm}+\theta_2 
E\big[G^{(2)}_n(\partialbs_z)\big\langle D(G^{(1)}_n(\theta;\partialbs_z,\partialbs_x)\psi_n),u_n[\partial_z]\big\rangle_\mfh f(\zeta_n(\theta_2))\big]
\\&&
\hspace{20mm}+2^{-1}(1-\theta_2^2)
E\big[G^{(3)}_n(\partialbs_z)\big\langle D(G^{(1)}_n(\theta;\partialbs_z,\partialbs_x)\psi_n),u_n[\partial_z]\big\rangle_\mfh f(\zeta_n(\theta_2))\big]
\\&&
\hspace{20mm}+
E\big[G^{(1)}_n(\theta;\partialbs_z,\partialbs_x)\langle D(G^{(1)}_n(\theta_2;\partialbs_z,\partialbs_x)\psi_n),u_n[\partial_z]\rangle_\mfh f(\zeta_n(\theta_2))\big]
\bigg\}d\theta_2
\\&&
\hspace{10mm}+ \int_0^\theta\bigg\{
\theta_1E\big[\psi_nG^{(2)}_n(\partialbs_z)G^{(1)}_n(\theta;\partialbs_z,\partialbs_x) f(\zeta_n(\theta_1))\big]
\\&&
\hspace{20mm}+2^{-1}(1-\theta_1^2)E\big[\psi_nG^{(3)}_n(\partialbs_z)G^{(1)}_n(\theta;\partialbs_z,\partialbs_x) f(\zeta_n(\theta_1))\big]
\\&&
\hspace{20mm}+E\big[\psi_nG^{(1)}_n(\theta;\partialbs_z,\partialbs_x)G^{(1)}_n(\theta_1;\partialbs_z,\partialbs_x) f(\zeta_n(\theta_1))\big]\bigg\}d\theta_1.
\eeas
For $i=4,\dots, 12$, a precise expression for  $\rho_n^{(i)}(f)$ is given below
\beas
\rho_n^{(4)}(f) &=& 
r_n \bigg\{E\big[2^{-1}\dotw_n[\partial_z]\langle D\psi_n,u_n[\partial_z]\rangle_\mfh f(\zeta_\infty(0))\big]\bigg\}
\\&&
\hspace{10mm}+E\big[2^{-1}N_n[\partial_z]\langle D\psi_n,u_n[\partial_z]\rangle_\mfh f(\zeta_\infty(0))\big]\bigg\}
\\&&
\hspace{10mm}+E\big[2^{-1}\dotx_n[\partial_x]\langle D\psi_n,u_n[\partial_z]\rangle_\mfh f(\zeta_\infty(0))\big]\bigg\}{\colorb ,}
\eeas

\begin{en-text}
\beas
\rho_n^{(4)}(f) &=& 
r_n \bigg\{E\big[2^{-1}\psi_n\langle D\dotw_n[\partial_z],u_n[\partial_z]\rangle_\mfh f(\zeta_\infty(0))\big]\bigg\}
\\&&
\hspace{10mm}+E\big[2^{-1}\psi_n\langle DN_n[\partial_z],u_n[\partial_z]\rangle_\mfh f(\zeta_\infty(0))\big]\bigg\}
\\&&
\hspace{10mm}+E\big[2^{-1}\psi_n\langle D\dotx_n[\partial_x],u_n[\partial_z]\rangle_\mfh f(\zeta_\infty(0))\big]\bigg\}{\colorb ,}
\eeas
\end{en-text}
\beas 
\rho_n^{(5)}(f) &=& 
r_n \bigg\{
E\bigg[3^{-1}\psi_n{\colorg r_n^{-1}}\bigg\langle D\big\langle DM_n[\partial_z],u_n[\partial_z]\big\rangle_\mfh,u_n[\partial_z]\bigg\rangle_\mfh f(\zeta_\infty(0))\bigg]
\\&&
\hspace{10mm}-
E\big[{\mathfrak S}^{(3,0)}{\colorr (\partial_z,\partial_x)}%[(\partial_z)^{\otimes3}]
f(\zeta_\infty(0))\big]\bigg\}{\colorb ,}
\eeas
\beas 
{\colorb \rho_n^{(6)}(f)} &=& 
r_n \bigg\{E\big[2^{-1}\psi_n{\colorg r_n^{-1}}\big(\big\langle DM_n[\partial_z],u_n[\partial_z]\big\rangle_\mfh-G_\infty[(\partial_z)^2]\big) f(\zeta_\infty(0))\big]
\\&&
\hspace{10mm}-
E\big[{\mathfrak S}^{(2,0)}_0{\colorr (\partial_z,\partial_x)}%[(\partial_z)^{\otimes2}]
f(\zeta_\infty(0))\big]\bigg\}{\colorb ,}
\eeas
\beas 
{\colorb \rho_n^{(7)}(f)} &=& 
r_n \bigg\{E\big[\psi_n{\colorg r_n^{-1}}\big\langle DW_\infty[\tti\sfz],u_n[\tti\sfz]\big\rangle_\mfh f(\zeta_\infty(0))\big]
-
E\big[{\mathfrak S}^{(2,0)}{\colorr (\partial_z,\partial_x)}%[(\partial_z)^{\otimes2}]
f(\zeta_\infty(0))\big]\bigg\}{\colorb ,}
\eeas
\beas 
{\colorb \rho_n^{(8)}(f)} &=& 
r_n \bigg\{E\big[\psi_n{\colorg r_n^{-1}}\big\langle DX_\infty[\partial_x],u_n[\partial_z]\big\rangle_\mfh f(\zeta_\infty(0))\big]
-
E\big[{\mathfrak S}^{(1,1)}{\colorr (\partial_z,\partial_x)}%[\partial_z,\partial_x]
f(\zeta_\infty(0))\big]\bigg\}{\colorb ,}
\eeas
\beas 
{\colorb \rho_n^{(9)}(f)} &=& 
r_n \bigg\{E\big[\psi_n\>(\dotw_n[\partial_z]+N_n[\partial_z]) f(\zeta_\infty(0))\big]
-
E\big[{\mathfrak S}^{(1,0)}{\colorr (\partial_z,\partial_x)}%[\partial_z]
f(\zeta_\infty(0))\big]\bigg\},
\eeas
\beas 
{\colorb \rho_n^{(10)}(f)} &=& 
r_n \bigg\{E\big[\psi_n\dotx_n[\partial_x]f(\zeta_\infty(0))\big]
-
E\big[{\mathfrak S}^{(0,1)}{\colorr (\partial_z,\partial_x)}%[\partial_x]
f(\zeta_\infty(0))\big]\bigg\},
\eeas
{\colorg
\beas 
 \rho_n^{(11)}(f) &=& 
 r_n \bigg\{E\big[\psi_n
 \big\langle D\dotw_n[\partial_z]+DN_n[\partial_z],u_n[\partial_z]\big\rangle_\HH
 f(\zeta_\infty(0))\big]
-
E\big[{\mathfrak S}^{(2,0)}_1 (\partial_z,\partial_x)
f(\zeta_\infty(0))\big]\bigg\},
\eeas
and 
\beas 
 \rho_n^{(12)}(f) &=& 
 r_n \bigg\{E\big[\psi_n
 \big\langle D\dotx_n[\partial_x],u_n[\partial_z]\big\rangle_\HH
 f(\zeta_\infty(0))\big]
-
E\big[{\mathfrak S}^{(1,1)}_1 (\partial_z,\partial_x)
f(\zeta_\infty(0))\big]\bigg\}.
\eeas
}
%
%+E[f(Z_n,X_n)(1-\psi_n)]-E[f(\zeta_\infty(0))(1-\psi_n)]. 
%\eeas

{\colorr Let $\beta=\max\{{\colorg 7},\beta_0\}$, where $\beta_0$ is the degree in $(\sfz,\sfx)$ of 
the random symbol ${\mathfrak S}$.} 
\begin{thm}\label{20160924-1}
Suppose that Condition $[A\>]$ is satisfied. 
Then 
\beas
E[f(Z_n,X_n)]
&=& 
E[ f\big(G_\infty^{1/2}\zeta+W_\infty, X_\infty\big)]+
r_n{\colred E[{\mathfrak S}(\partial_z,\partial_x)f\big(G_\infty^{1/2}\zeta+W_\infty,X_\infty\big)]}
+\rho^{(1)}_n(f)
\eeas
for $f\in C_b^{{\colorr \beta}}(\bbR^{\check{\sfd}})$. 
\end{thm}
\proof 
Since $\cals(\bbR^{\check{\sfd}})\ni f\mapsto E[f(Z_n,X_n)]$,  $E[ f\big(G_\infty^{1/2}\zeta+W_\infty, X_\infty\big)]$,
$E[{\mathfrak S}(\partial_z,\partial_x)f(G_\infty^{1/2}\zeta+W_\infty,X_\infty)]$ 
and $\rho^{(1)}_n(f)$ are 
{\colorg expressed as a sum of}
bounded signed measures {\colorg applied to the derivatives of $f$}, 
Equation (\ref{20160926-1}) holds for all $f\in C_b^{{\colorr \beta}}(\bbR^{\check{\sfd}})$ 
with the same expression for $\rho^{(1)}_n(f)$. 
\footnote{On each bounded ball, take a sequence $f_k\in\cals(\bbR^{\check{\sfd}})$ that 
converges to $f$ in $C_b^{{\colorr \beta}}$.  
The differentiability condition can usually be relaxed since not all {\colorb orders of} monomials appear in ${\mathfrak S}$. }
\qed

\begin{en-text}
\begin{rem}\rm 
\beas 
\bigg|(2\pi)^{-\check{\sfd}}
\int_{\bbR^{\check{\sfd}}}\hat{f}(\sfz,\sfx)\varphi_n(\theta;\varsigma(\tti\sfz,\tti\sfx))
d\sfz d\sfx\bigg|
&\leq&
%\|\psi_n\|_\infty
|\!|\!|\varsigma|\!|\!|_1
\|f\|_{C^\varsigma_b},
\eeas
where 
$|\!|\!|\varsigma|\!|\!|_1$ is the sum of $L^1$-norms of the coefficients of $\varsigma$, 
and $\|f\|_{C^\varsigma_b}$ is the sum of sup-norms of the derivatives of $f$ 
appearing in $\varsigma$. 
\end{rem}
\end{en-text}
\onelineskip
%\begin{rem}\rm 
In order to prove $\rho^{(1)}_n(f)=o(r_n)$,  
{\colorb we can apply either} stable convergence $A_n\to^{d_s}A_\infty$, 
$L^1$-convergence $\|A_n-A_\infty\|_1\to0$, 
or the integration-by-parts formula 
to evaluate the error terms of the form $E[A_nf(\zeta_n(\theta))]-E[A_\infty f(\zeta_n(\theta))]$ 
whether $A_\infty=0$ or not. 
{\colorb We will present an estimate for $\rho_n^{(1)}(f)$ in Section \ref{20170504-3}.} 
%\end{rem}

\subsection{\colorb Estimate of $\rho^{(1)}_n(f)$}\label{20170504-3}
{\colorb 
Define the following random symbols
\beas
{\colred {\mathfrak S}^{(3,0)}_n(\tti\sfz)} &=&{\mathfrak S}^{(3,0)}_n(\tti\sfz,\tti\sfx) 
\yeq 
\frac{1}{3}r_n^{-1}\bigg\langle D\big\langle DM_n[\tti\sfz],u_n[\tti\sfz]\big\rangle_\mfh,u_n[\tti\sfz]\bigg\rangle_\mfh
\>\equiv\>\frac{1}{3}{\sf qTor}[(\tti\sfz)^{\otimes3}],
\\
{\colred {\mathfrak S}^{(2,0)}_{0,n}(\tti\sfz)}
&=& 
{\mathfrak S}^{(2,0)}_{0,n}(\tti\sfz,\tti\sfx)
\yeq
\half r_n^{-1}G^{(2)}_n(\sfz)
\>=\>
\half r_n^{-1}\bigg(\big\langle DM_n[\tti\sfz],u_n[\tti\sfz]\big\rangle_\mfh-G_\infty[(\tti\sfz)^2]\bigg)
\>\equiv\>\half {\sf qTan}[(\tti\sfz)^{\otimes2}],
\\
{\colred {\mathfrak S}^{(2,0)}_n(\tti\sfz)}
&=&
{\mathfrak S}^{(2,0)}_n(\tti\sfz,\tti\sfx)
\yeq
r_n^{-1}\big\langle DW_\infty[\tti\sfz],u_n[\tti\sfz]\big\rangle_\mfh,
\\
{\mathfrak S}^{(1,1)}_n(\tti\sfz,\tti\sfx)
&=&
r_n^{-1}\big\langle DX_\infty[\tti\sfx],u_n[\tti\sfz]\big\rangle_\mfh,
\\
{\colred {\mathfrak S}^{(1,0)}_n(\tti\sfz)}
&=&
{\mathfrak S}^{(1,0)}_n(\tti\sfz,\tti\sfx)
\yeq
\dotw_n[\tti\sfz]+N_n[\tti\sfz],
\\
{\colred {\mathfrak S}^{(0,1)}_n(\tti\sfx)}
&=&
{\mathfrak S}^{(1,0)}_n(\tti\sfz,\tti\sfx)
\yeq
\dotx_n[\tti\sfx],
\\
{\colred {\colorg {\mathfrak S}^{(2,0)}_{1,n}(\tti\sfz)}}
&=&
{\colorg {\mathfrak S}^{(2,0)}_{1,n}(\tti\sfz,\tti\sfx)}
{\colorg \yeq}
{\colorg 
\bigg\langle D\dotw_n[\tti\sfz]+DN_n[\tti\sfz],u_n[\tti\sfz]\bigg\rangle_\HH, }
\\
{\colorg {\mathfrak S}^{(1,1)}_{1,n}(\tti\sfz,\tti\sfx)}
&{\colorg =}&
{\colorg 
\bigg\langle D\dotx_n[\tti\sfx],u_n[\tti\sfz]\bigg\rangle_\HH.
}
\eeas

\begin{rem}\rm 
As mentioned in Remark \ref{20171117-1}, 
the order of the random symbol ${\mathfrak S}^{(3,0)}(\tti\sfz,\tti\sfx)$ appearing as the limit  of the
corresponding sequence ${\mathfrak S}^{(3,0)}_n(\tti\sfz,\tti\sfx)$ does not necessarily coincide with 
that of the latter because ${\mathfrak S}^{(3,0)}(\tti\sfz,\tti\sfx)$ is determined by 
the action of ${\mathfrak S}^{(3,0)}_n(\tti\sfz,\tti\sfx)$ to $\Psi(\sfz,\sfx)$ under expectation. 
It is also the case for other symbols. 
\end{rem}

For $\HH$-valued tensors ${\sf S}=({\sf S}_{i})$ and ${\sf T}=({\sf T}_{j})$, 
$\langle{\sf S},{\sf T}\rangle_\HH$ denotes the tensor with components 
$(\langle{\sf S}_{i},{\sf T}_{j}\rangle_\HH)_{i,j}$. 

\bd\im[[B\!\!]] 
{\bf (i)}  
$u_n\in\bbD^{{\colorg 4},p}(\mfh\otimes\bbR^\sfd)$, 
$G_\infty\in\bbD^{3,p}(\bbR^\sfd\otimes_+\bbR^\sfd)$, 
$W_n,W_\infty,N_n\in\bbD^{3,p}(\bbR^\sfd)$, 
$X_n,X_\infty\in\bbD^{3,p}(\bbR^{\sfd_1})$ and 
$\psi_n\in\bbD^{2,p_1}(\bbR)$ for some $p$ and $p_1$ satisfying $5p^{-1}+p_1^{-1}\leq1$. 

\bd
\im[\hspace{-2mm}(ii)] The following estimates hold: %{\colorr still need to check!}
\bea\label{b1}
\|u_n\|_{{\colred 1},p} &=& O(1)
% $\rho[9]$ needs the derivative of $u_n$ (2017.11.17). 
\eea
\bea\label{b2}
\sum_{k=2,3}
\|G_n^{(k)}\|_{p/2} &=& O(r_n)%\qquad(k=2,3)
\eea
{\colorg 
\bea\label{b32}
\|\big\langle DG^{(2)}_n,u_n\big\rangle_\mfh\|_{p/3}
&=& O(r_n)%\qquad(k=2,3)
\eea
\bea\label{b33}
\|\big\langle DG^{(3)}_n,u_n\big\rangle_\mfh\|_{p/3}
&=& o(r_n)%\qquad(k=2,3)
\eea
}
\bea\label{b4}
\sum_{k=2,3}
\bigg\|\bigg\langle D\bigg(\big\langle DG^{(k)}_n,u_n\big\rangle_\mfh\bigg),u_n\bigg\rangle_\mfh\bigg\|_{p/4}
&=& o(r_n)%\qquad(k=2,3)
\eea
\bea\label{b5}
%W_n,X_n\text{ for }G_n^{(1)}
\sum_{{\sf A}=W_\infty[\sfz],X_\infty[\sfx]}
\big\|\langle D{\sf A},u_n\rangle_\mfh\big\|_p &=& O(r_n)
\eea
%\bea\label{b6}
%\big\|\langle DX_\infty[\sfx],u_n[\sfz]\rangle_\mfh\big\|_p &=& O(r_n)
%\eea
\bea\label{b7}
\sum_{{\sf A}=W_\infty,X_\infty}
\|\big\langle D\langle D{\sf A},u_n[\sfz]\rangle_\mfh,u_n[\sfz]\big\rangle_\mfh\|_{p/3}
&=& o(r_n)
\eea
%\bea\label{b8}
%\|\big\langle D\langle DX_\infty[\sfx],u_n[\sfz]\rangle_\mfh,u_n[\sfz]\big\rangle_\mfh\|_{p/3}
%&=& o(r_n)
%\eea
{\colorg 
\bea\label{b9}
\sum_{{\sf A}=W_\infty,X_\infty}
\bigg\|\bigg\langle D\bigg(\big\langle D\langle D{\sf A},u_n\rangle_\mfh,u_n\big\rangle_\mfh\bigg),u_n\bigg\rangle_\mfh\bigg\|_{p/4}
&=& o(r_n)
\eea
}
\bea\label{b10}
\|\dotw_n\|_{3,p}+\|N_n\|_{3,p}+\|\dotx_n\|_{3,p}&=&O(1)
\eea
{\colorg 
\bea\label{b10.1}
\sum_{{\sf B}=\dotw_n,N_n,\dotx_n}
\bigg\| \big\langle D\langle D{\sf B},u_n\rangle_\HH ,u_n\big\rangle_\HH\bigg\|_{p/3}&=& o(1)
\eea
\bea\label{b10.2}
\sum_{{\sf B}=\dotw_n,N_n,\dotx_n}
\bigg\| \bigg\langle D\big\langle D\langle D{\sf B},u_n\rangle_\HH ,u_n\big\rangle_\HH
 ,u_n\bigg\rangle_\HH\bigg\|_{p/4}&=& o(1)
\eea
}
\bea\label{b11}
\|1-\psi_n\|_{2,p_1}&=&o(r_n)
\eea
\im[\hspace{-2mm}(iii)] For every $\sfz\in\bbR^\sfd$ and $\sfx\in\bbR^{\sfd_1}$, 
it holds that 
\beas 
\lim_{n\to\infty}
E\big[\Psi(\sfz,\sfx){\mathfrak T}_n(\tti\sfz,\tti\sfx)\psi_n\big]
&=&
E\big[\Psi(\sfz,\sfx){\mathfrak T}(\tti\sfz,\tti\sfx)\big]
\eeas
for $({\mathfrak T}_n,{\mathfrak T})=({\mathfrak S}^{(3,0)}_n,{\mathfrak S}^{(3,0)})$, 
$({\mathfrak S}^{(2,0)}_{0,n},{\mathfrak S}^{(2,0)}_0)$, 
$({\mathfrak S}^{(2,0)}_n,{\mathfrak S}^{(2,0)})$, 
$({\mathfrak S}^{(1,1)}_n,{\mathfrak S}^{(1,1)})$, 
$({\mathfrak S}^{(1,0)}_n,{\mathfrak S}^{(1,0)})$,  \break
$({\mathfrak S}^{(0,1)}_n,{\mathfrak S}^{(0,1)})$, 
{\colorg 
$({\mathfrak S}^{(2,0)}_{1,n},{\mathfrak S}^{(2,0)}_1)$ and 
$({\mathfrak S}^{(1,1)}_{1,n},{\mathfrak S}^{(1,1)}_1)$. }
\ed
\ed
\onelineskip

We say that a family of random symbols 
$\{\varsigma^\lambda(\tti\sfz,\tti\sfx)=\sum_{k,m} c_{k,m}^\lambda[(\tti\sfz)^{\otimes k}\otimes(\tti\sfx)^{\otimes m}];\>\lambda\in\Lambda\}$ is uniformly integrable (u.i.)
if the degrees of polynomials are bounded %(so $c_{k,m}^\lambda$ 
and the family $\{c_{k,m}^\lambda;\>\lambda\in\Lambda\}$ of 
tensor-valued random variables 
is uniformly integrable for every $(k,m)$. 
Recall that $\beta=\max\{{\colorg 7},\beta_0\}$, where $\beta_0$ is the degree in $(\sfz,\sfx)$ of 
the random symbol ${\mathfrak S}$.

\begin{thm}    \label{thm2}
Suppose that Condition $[B]$ is fulfilled. Then 
$\rho_n^{(1)}(f) = o(r_n)$ for every $f\in C^\beta_b(\bbR^{\check{\sfd}})$. 
Moreover, 
\beas 
\sup_{f\in \calb}|\rho_n^{(1)}(f)| &=& o(r_n)
\eeas
for every bounded set $\calb$ in $C^{\beta+1}_b(\bbR^{\check{\sfd}})$. 
\end{thm}

\proof
The sequence $\{{\mathfrak S}^{(3,0)}_n; n\in \mathbb{N}\}$,  is u.i. from (\ref{b32}) %for $k=2$ 
and (\ref{b2}) for $k=3$. 
$\{{\mathfrak S}^{(2,0)}_{0,n};  n\in \mathbb{N}\}$ is u.i. from (\ref{b2}) for $k=2$. 
$\{{\mathfrak S}^{(2,0)}_n ; n\in \mathbb{N}\}$ and $\{{\mathfrak S}^{(1,1)}_n; n\in \mathbb{N}\}$ are u.i. from (\ref{b5}). 
$\{{\mathfrak S}^{(1,0)}_n; n\in \mathbb{N}\}$ and $\{{\mathfrak S}^{(0,1)}_n; n\in \mathbb{N}\}$ are u.i. from (\ref{b10}). 
{\colorg 
$\{{\mathfrak S}^{(2,0)}_{1,n}; n\in \mathbb{N}\}$ and $\{{\mathfrak S}^{(1,1)}_{1,n}; n\in \mathbb{N}\}$ are u.i. from (\ref{b1}) and (\ref{b10}).} 

We consider first the case of $\rho^{(8)}_n(f)$, given by
\beas
\rho^{(8)}_n(f)= r_n \left\{ E\big[\psi_n{\mathfrak S}^{(1,1)}_n(\partial_z,\partial_x)
f(\zeta_\infty(0))\big] - E\big[{\mathfrak S}^{(1,1)}(\partial_z,\partial_x)%[\partial_z^{\otimes2}]
f(\zeta_\infty(0))\big]\right\}.
\eeas
By $[B]$ (iii) and the formula (\ref{170417-1}), we obtain 
\bea\label{170417-2}
\lim_{n\to\infty} E\big[\psi_n{\mathfrak S}^{(1,1)}_n(\partial_z,\partial_x)
f(\zeta_\infty(0))\big]
&=&
E\big[{\mathfrak S}^{(1,1)}(\partial_z,\partial_x)%[\partial_z^{\otimes2}]
f(\zeta_\infty(0))\big]
\eea
for $f\in\cals(\bbR^{\check{\sfd}})$. 
Let $\chi:\bbR^{\check{\sfd}}\to[0,1]$ be a smooth function with a compact support. 
Since for $f\in C^\beta_b(\bbR^{\check{\sfd}})$, the function $\chi f$ is uniformly approximated 
in $C^\beta_b(\bbR^{\check{\sfd}})$ 
by some 
 function in $\cals(\bbR^{\check{\sfd}})$, we have 
\bea\label{170417-3}
\lim_{n\to\infty} \bigg|E\big[\psi_n{\mathfrak S}^{(1,1)}_n(\partial_z,\partial_x) 
(\chi f)(\zeta_\infty(0))\big]
-
E\big[{\mathfrak S}^{(1,1)}(\partial_z,\partial_x)%[\partial_z^{\otimes2}]
(\chi f)(\zeta_\infty(0))\big]\bigg|
&=&
0
\eea
for $f\in C^\beta_b(\bbR^{\check{\sfd}})$.  
Let $\calb^k$ be a bounded set in $C^k_b(\bbR^{\check{\sfd}})$. Let $\ep>0$. 
Due to the uniformly integrability of the family $\{{\mathfrak S}^{(1,1)}_n; n\in \mathbb{N} \}$,  we can write
\bea\label{170417-4}
\sup_{f\in\calb^\beta,n\in\bbN}
\bigg| E\big[\psi_n{\mathfrak S}^{(1,1)}_n(\partial_z,\partial_x) 
((1-\chi) f)(\zeta_\infty(0))\big]\bigg|
+\sup_{f\in\calb^\beta}\bigg|
E\big[{\mathfrak S}^{(1,1)}(\partial_z,\partial_x)%[\partial_z^{\otimes2}]
((1-\chi) f)(\zeta_\infty(0))\big]\bigg|
&<& 
\ep
\nn\\&&
\eea
if we choose $\chi$ satisfying $\chi=1$ on a sufficiently large compact set. 
Combining (\ref{170417-3}) and (\ref{170417-4}), we obtain 
\bea\label{170417-5}
\lim_{n\to\infty} \bigg|E\big[\psi_n{\mathfrak S}^{(1,1)}_n(\partial_z,\partial_x) 
f(\zeta_\infty(0))\big]
-
E\big[{\mathfrak S}^{(1,1)}(\partial_z,\partial_x)%[\partial_z^{\otimes2}]
f(\zeta_\infty(0))\big]\bigg|
&=&
0
\eea
for $f\in C^\beta_b(\bbR^{\check{\sfd}})$.  
Moreover, we have 
\bea\label{170417-6}
\lim_{n\to\infty} 
\sup_{f\in\calb^{\beta+1}}
\bigg|E\big[\psi_n{\mathfrak S}^{(1,1)}_n(\partial_z,\partial_x) 
(\chi f)(\zeta_\infty(0))\big]
-
E\big[{\mathfrak S}^{(1,1)}(\partial_z,\partial_x)%[\partial_z^{\otimes2}]
(\chi f)(\zeta_\infty(0))\big]\bigg|
&=&
0,
\eea
since 
{\colorg the functionals in $|\ |$ are equi-continuous in $f$ in $C^\beta(\bbR^{\check{\sfd}})$ and}
$\calb^{\beta+1}$ is relatively compact in $\calb^\beta$ if the domain is restricted 
to $\text{supp}( \chi)$.  
Therefore we showed that $\rho^{(8)}_n(f)=o(r_n)$ for every $f\in C^\beta_b(\bbR^{\check{\sfd}})$, and that 
$\sup_{f\in \calb^{\beta+1}}|\rho^{(8)}_n(f)|=o(r_n)$. 
Similarly, we obtain the same estimate for  $\rho^{(i)}_n(f)$ for  $i=5,\dots,12$. 

Moreover, we  have
$\sup_{f\in\calb^0}|\rho^{(2)}_n(f)|=o(r_n)$ from (\ref{b11}), and 
$\sup_{f\in\calb^2}|\rho^{(4)}_n(f)|=o(r_n)$ from (\ref{b1}), (\ref{b10}) and (\ref{b11}). 

{\colorg 
The tensor $\hat{G}^{(1)}_n(\theta,\cdot,\cdot)$ is denoted by $\hat{G}^{(1)}_n(\theta)$. 
The estimate 
$ 
\|\hat{G}^{(1)}_n(\theta)\|_{p/2} = O(r_n)
$
follows from (\ref{b5}), (\ref{b10}) and (\ref{b1}), and 
the estimate 
$ 
\|\check{G}^{(1)}_n\|_{p/2} = O(r_n)
$
follows from (\ref{b10}). Therefore 
\bea\label{esta} 
\|G^{(1)}_n(\theta)\|_{p/2} &=& O(r_n).
\eea
We obtain 
\bea\label{estb1} 
\big\|\big\langle D\hat{G}^{(1)}_n(\theta),u_n\big\rangle_\HH\big\|_{p/3} &=& o(r_n) 
\eea
from (\ref{b7}) and (\ref{b10.1}) for every $\theta$ (or $\theta=0,1$), and 
\bea\label{estb2} 
\big\|\big\langle D\check{G}^{(1)}_n,u_n\big\rangle_\HH\big\|_{p/2} &=& O(r_n)
\eea
from (\ref{b10}). 
Moreover, we have 
\beas
\bigg\|\bigg\langle D\big\langle D\hat{G}^{(1)}_n(\theta),u_n\big\rangle_\HH,u_n\bigg\rangle_\HH\bigg\|_{p/3} &=& o(r_n) 
\eeas
by (\ref{b9}) {\colorr and (\ref{b10.2})}, and 
\beas
\bigg\|\bigg\langle D\big\langle D\check{G}^{(1)}_n,u_n\big\rangle_\HH,u_n\bigg\rangle_\HH\bigg\|_{p/3} &=& o(r_n) 
\eeas
by {\colorr (\ref{b10.1}),} so that 
\bea\label{estc}
\bigg\|\bigg\langle D\big\langle DG^{(1)}_n(\theta),u_n\big\rangle_\HH,u_n\bigg\rangle_\HH\bigg\|_{p/3} &=& o(r_n) 
\eea

Let us investigate the order of $\rho^{(3)}_n(f)$. 
Denote by $\rho[i]$ ($i=1,\dots,24$) 
the 24 linear functionals of $f$ appearing in the expression of $\rho^{(3)}_n(f)$.  
Though not explicitly mentioned in what follows, 
Condition (\ref{b11}) is used every time in estimation of $\rho[i]$ to ensure either 
$\|\psi_n\|_{p_1}=O(1)$, $\|D\psi_n\|_{p_1}=o(r_n)$ or $\|D^2\psi_n\|_{p_1}=o(r_n)$. 
The estimate $\rho[1]=o(r_n)$ follows from (\ref{b1}) (and (\ref{b11})).
{\color {black} The term $\rho[2]$  corresponds to the first term in the expression of $\rho^{(3,1)}_n$; then}
$\rho[2]=o(r_n)$ from (\ref{b4}), (\ref{b2}), (\ref{b32}) and (\ref{b1}) with the aid of the Leibniz rule; 
$\rho[3]=O(r_n^2)$ from (\ref{b2}), (\ref{b32}) and (\ref{b1}); 
$\rho[4]=O(r_n^2)$ from (\ref{b2}), (\ref{b32}) and (\ref{b1}); 
$\rho[5]=O(r_n^2))$ from (\ref{esta}), (\ref{b32}) and (\ref{b1}); 
$\rho[6]=O(r_n^2)$ from (\ref{b2}); 
$\rho[7]=O(r_n^2)$ from (\ref{b2}); 
$\rho[8]=O(r_n^2)$ from (\ref{esta}) and (\ref{b2}); 
$\rho[9]=o(r_n)+o(r_n)o(r_n)+O(r_n)o(r_n)=o(r_n)$ from (\ref{b4}), (\ref{b33}), (\ref{b1}) and (\ref{b2}); 
$\rho[10]=o(r_n^2)+o(r_n^3)=o(r_n^2)$ from (\ref{b33}), (\ref{b2}) and (\ref{b1}); 
$\rho[11]=o(r_n^2)+o(r_n^3)=o(r_n^2)$ from (\ref{b33}), (\ref{b2}) and (\ref{b1}); 
$\rho[12]=o(r_n^2)+o(r_n^3)=o(r_n^2)$ from (\ref{esta}), (\ref{b33}), (\ref{b2}) and (\ref{b1}); 
$\rho[13]=O(r_n^2)$ from (\ref{b2}); 
$\rho[14]=O(r_n^2)$ from (\ref{b2}); 
$\rho[15]=O(r_n^2)$ from (\ref{esta}) and (\ref{b2}); 
$\rho[16]=o(r_n)+O(r_n)o(r_n)+O(r_n)o(r_n)=o(r_n)$ from (\ref{estc}), (\ref{estb1}), (\ref{estb2}), (\ref{esta}) and (\ref{b1}); 
$\rho[17]=O(r_n^2)+o(r_n^3)=O(r_n^2)$ from (\ref{b2}), (\ref{estb1}), (\ref{estb2}), (\ref{esta}) and (\ref{b1}); 
$\rho[18]=O(r_n^2)+o(r_n^3)=O(r_n^2)$ from (\ref{b2}), (\ref{estb1}), (\ref{estb2}), (\ref{esta}) and (\ref{b1}); 
$\rho[19]=O(r_n^2)+o(r_n^3)=O(r_n^2)$ from (\ref{esta}), (\ref{estb1}), (\ref{estb2}) and (\ref{b1}); 
$\rho[20]=O(r_n^2)$ from (\ref{b2}) and (\ref{esta}); 
$\rho[21]=O(r_n^2)$ from (\ref{b2}) and (\ref{esta}); 
$\rho[22]=O(r_n^2)$ from (\ref{esta}); 
$\rho[23]=o(r_n)+o(r_n^2)=o(r_n)$ from (\ref{b33}), (\ref{b2}) and (\ref{b1}); 
$\rho[24]=o(r_n)+o(r_n^2)=o(r_n)$ from (\ref{estb1}) and (\ref{b1}). 
We remark that some $\rho[i]$'s involve a product of five $L^p$ variables besides $\psi_n$ or its derivative, 
and that the seventh derivative of $f$ appears in $\rho[11]$. 
These estimates give $\sup_{f\in\calb^7}|\rho^{(3)}_n(f)|=o(r_n)$. 
This completes the proof. 
}
\qed
\vspace*{5mm}

\begin{en-text}
%$5/p+1/p_1=1$
\beas 
&&
\|u_n\|_p = O(1)
\\&&
\|G_n^{(k)}\|_{1,p/2} = O(r_n^{1/2})\qquad(k=2,3)
\\&&
.........
\eeas
Then 
\beas
&&
\|u_n\|_p = O(1)
\\&&
\|G_n^{(k)}\|_{p/2} = o(r_n^{1/2})\qquad(k=1,2,3)
\\&&
\|\big\langle DG^{(k)}_n(\sfz_1),u_n[\sfz_2]\big\rangle_\mfh\|_{p/3}
= o(r_n^{1/2})\qquad(k=1,2,3)
\\&&
\bigg\|\bigg\langle D\bigg(\big\langle DG^{(k)}_n(\sfz_1),u_n[\sfz_2]\big\rangle_\mfh\bigg),u_n[\sfz_3]\bigg\rangle_\mfh\bigg\|_{p/4}
= o(r_n)\qquad(k=1,2,3)
\\&&
%W_n,X_n\text{ for }G_n^{(1)}
\big\|\langle DW_\infty[\sfz_1],u_n[\sfz_2]\rangle_\mfh\big\|_p = o(r_n)
\\&&
\big\|\langle DX_\infty[\sfz_1],u_n[\sfz_2]\rangle_\mfh\big\|_p = o(r_n)
\\&&
\|1-\psi_n\|_{2,p_1}=o(r_n)
\eeas
\end{en-text}
}

%%%%%%%%%%%%%%%%%%%%%%%%%%%%%%%
\section{A functional of a fractional Brownian motion}\label{170813-2}
 
In this section, we shall consider a functional of 
a fractional Brownian motion (fBm) with Hurst parameter $H\in (0,1)$ on the time interval $[0,1]$. {\color {black} The fBm is}
a centered Gaussian process $B=\{B_{t}, t\in [0,1] \}$ defined on 
a probability space $(\Omega,\calf,P)$ with covariance
function
\[
R_H(t,s)\yeq
E[B_{s}B_{t}]\yeq \frac{1}{2}\left( t^{2H}+s^{2H}-|t-s|^{2H}\right).
\]
The process $B$ is a standard Brownian motion for $H=\frac 12$.
Denote by 
$\mathcal{E}$ the set of step functions on $[0,1] $. 
Then it is possible to introduce an inner product {\color {black} in} $\mathcal{E}$ such that 
\begin{equation*}
\left\langle {\mathbf{1}}_{[0,t]},{\mathbf{1}}_{[0,s]}\right\rangle _{%
\EuFrak H}=E[B_{s}B_{t}].
\end{equation*}%
Let $\|\cdot\|_\mfh=\langle\cdot,\cdot\rangle_\mfh^{1/2}$. 
Hilbert space $\EuFrak H$ is defined as the closure of $\mathcal{E}$ with respect to $\|\cdot\|_\mfh$. 

It is known that the mapping $\mathbf{1}_{[0,t]}\mapsto B_{t}$ can be extended to a
linear isometry between $\EuFrak H$ and the Gaussian space
spanned by $B$ in $L^2=L^2(\Omega,\calf,P)$.
We denote this isometry by $\phi \mapsto B(\phi )$. 
The process $\{B(\phi ), \phi \in \EuFrak H\}$ is  an isonormal Gaussian process.
We refer to \cite{Nualart2006} for a detailed account on the basic properties of {\color {black} the fBm. Assume again that $\mathcal{F}$ is the $\sigma$-field generated by $B$}.

In the case $H>\frac 12$, the space  $\EuFrak H$ contains 
the linear space $|\mfh|$ of all measurable functions $\varphi: [0,1] \rightarrow \mathbb{R}$ satisfying 
\beas\label{20170504-7}
\int_0^1 \int_0^1 |\varphi(s)| |\varphi(t)| |t-s| ^{2H-2} dsdt <\infty. 
\eeas
In this case, the inner product $\left\langle \varphi, \phi  \right\rangle _{\EuFrak H} $ is represented by 
\begin{equation}  \label{eq7}
\left\langle \varphi, \phi  \right\rangle _{\EuFrak H} =
H(2H-1)  \int_0^1 \int_0^1  \varphi(s)   \phi(t) |t-s| ^{2H-2} dsdt
\end{equation}
for $\varphi,\phi\in|\mfh|$. 
Furthermore, $L^{\frac 1H}([0,1])$ is continuously embedded into $\EuFrak H$. 
The following lemma provides useful formulas for the inner product in the Hilbert space $\HH$.

\begin{lemme}\label{20170616-1}  {\bf (i)} Let $H \not = \frac 12$. For any piecewise continuous function $\varphi$ on $[0,1]$,
 the inner product  $ \langle \varphi ,   \mathbf{1}_{[0,s]}\rangle_{\HH}$ is given by
 \bea \nn
\langle \varphi ,   \mathbf{1}_{[0,s]}\rangle_{\HH}
&=&
\int_0^1  \varphi(t) \frac {\partial R_H} {\partial t} (t,s) dt
\\&=&  \label{hk1}
 \int_0^1  \varphi(t) H\big\{
t^{2H-1}-|t-s|^{2H-1}{\rm sign}(t-s)
\big\} dt. 
\eea 
 \bd 
 \im[(ii)]  Let $H \not = \frac 12$. For  any piecewise continuous function $\varphi$ on $[0,1]$ and $\psi\in C_b^1({\colred [0,1]})$, 
 \bea 
 \langle \varphi ,  \psi \rangle_{\HH} &=&  \nn
 \int_0^1  \varphi(t) \bigg\{\frac {\partial R_H} {\partial t} (t,1)\psi(1)
- \int_0^1 \frac {\partial R_H} {\partial t} (t,s) \psi'(s) ds\bigg\}dt
\\&=& \nn
\int_0^1\varphi(t)H\bigg\{t^{2H-1}\psi(0)+(1-t)^{2H-1}\psi(1)
\\&&  \label{hk11}
+\int_0^1|t-s|^{2H-1}{\rm sign}(t-s)\psi'(s)ds\bigg\}dt.
 \eea
 \im[(iii)] Let $0<H<1/2$, $0<a<b<1$. Then 
 for any piecewise continuous function $\varphi$ on $[0,1]$, 
 \beas 
 \big|\langle \varphi,\mathbf{1}_{[a,b]}\rangle_\HH\big|
 &\leq& 
 \|\varphi\|_\infty(b-a)^{2H}.
 \eeas
 \ed
  \end{lemme}
  \proof 
{\color {black} Approximating $\varphi$ by  step functions we can derive (i). 
For (ii),  using (i) we can write}
\beas 
 \langle \varphi ,  \psi \rangle_{\HH} &=& 
 \bigg\langle \varphi, \psi(1)\mathbf{1}_{[0,1]}-\int_0^1\mathbf{1}_{[0,s]}\psi'(s)ds\bigg\rangle_\HH
\\&=&
\psi(1) \int_0^1  \varphi(t) \frac {\partial R_H} {\partial t} (t,1) dt
- \int_0^1\int_0^1  \varphi(t) \frac {\partial R_H} {\partial t} (t,s) \psi'(s) dsdt.
\eeas
Simple calculus with (i) gives (iii). 
\qed

In what follows, we shall write 
\begin{equation}\label{e:ch}
 c_H=  \sqrt{H\Gamma(2H)}, \quad H\in (0,1).
\end{equation}
Notice that  $c_{\frac12}  =\frac{1}{\sqrt{2}}$ and for $H>1/2$, $c_H =\sqrt{H(2H-1)\Gamma (2H-1)}$.

We will consider the sequence  of random variables $Z_n =\delta(u_n)$,  $n\geq 1$, where
 \[
u_n(t)= n^H t^{n}B_{t} \mathbf{1}_{[0,1]} (t).
\]
The following provides the convergence in law of the sequence $Z_n$.
For the Brownian motion case, this result goes back to Peccati and Yor \cite{peccatiyor2004}. 
For $H>\frac 12$, it was proved by  Peccati and Tudor \cite{peccati2008stable} and a rate of convergence in the total variation distance was established {\color {black} in}   \cite{nourdin2016quantitative}.   For $H<\frac 12$, the convergence {\color {black} in law of $Z_n$} it is a consequence of our asymptotic expansion proved below. 
\begin{prop}    \label{prop4.1}
The sequence $Z_n$ converges stably in law to $\zeta \sqrt{G_\infty} $, where $G_\infty= c^2_H B^2_1$ and $\zeta$ is a $N(0,1)$ random variable, independent of $\{B_t, t \in [0,1]\}$.
\end{prop}

 In the setting of Section \ref{20160813-1}, the variables are now  
$Z_n = M_n$, $W_n=W_\infty=0$,   $X_n=X_\infty =0$,
  $\psi_n=1$ and 
$G_\infty=c_H^2B_1^2$.
  We are interested in investigating  
the asymptotic behavior of the three basic terms:  modified quasi torsion,   quasi tangent and quasi torsion. 
  We denote by $C_H$ a generic constant depending on $H$, that may vary in different lines. 

Consider first the case of the modified quasi torsion.

\noindent {\it (i) Case $H=\frac 12$}.  We have  $G_\infty= \frac 12 B_1^2$ and $D_t G_\infty = B_1 \mathbf{1}_{[0,1]}(t)$. Therefore,
\beas
G_n^{(3)} &=&  {\color {black}  \langle  DG_\infty , u_n \rangle_{\HH} }=\sqrt{n} B_1 \int_0^1   t^n B_t dt.
\eeas
As a consequence, taking  $r_n=n^{-1/2}$, we obtain  
\bea \label{hk2}
{\sf mqTor}  &=&\sqrt{n} G_n^{(3)} \overset{L^p} {\rightarrow} B_1^2,
\eea
for all $p\ge 2$.

\medskip
\noindent {\it (ii) Case $H\not = \half$}. We have $ G_\infty= c_H^2 B_1^2$,  where $c_H$ is the constant defined in  (\ref{e:ch})  and 
$D_t G_\infty= 2c_H^2 B_1 \mathbf{1}_{[0,1]}(t)$. 
Applying  (\ref{hk1}) yields
 \beas
 \langle DG_\infty, u_n \rangle_{\HH} 
&=& 
2c_H^2n^{H}B_1\langle t^nB_t{\bf 1}_{[0,1]}(t),{\bf 1}_{[0,1]}(t)\rangle_\HH
\\&=& 
2c_H^2n^{H}B_1\int_0^1 B_t t^nH\big\{t^{{\colred 2H-1}}+(1-t)^{2H-1}\big\}dt.
\eeas
Using the decomposition $B_1B_t= B_1^2 + B_1(B_t-B_1)$, we can write
\beas
 \langle DG_\infty, u_n \rangle_{\HH} 
&=&
2c_H^2n^{H}B_1^2\int_0^1t^nH\big\{t^{{\colred 2H-1}}+(1-t)^{2H-1}\big\}dt
\\&&
+2c_H^2n^{H}B_1\int_0^1t^n(B_t-B_1)H\big\{t^{{\colred 2H-1}}+(1-t)^{2H-1}\big\}dt
\\&=&
2c_H^2 B_1^2{\colred H}\bigg\{ {\colred n^{H-1}} +n^H\frac{\Gamma(n+1)\Gamma(2H)}{\Gamma(n+1+2H)}\bigg\}+R_n,
 \eeas
 where $\|R_n \|_p= O(n^{-1})+O(n^{-2H})$ for all $p\ge 2$.
 As a consequence, we obtain the following results:
 
 \noindent
 If $H>\frac 12$ we take $r_n= n^{H-1}$ and
 \bea \label{hk3}
 {\sf mqTor} &=&  r_n^{-1}  G_n^{(3)} \overset{L^p} {\rightarrow}  {\colred 2}
 H^2\Gamma(2H)B_1^2.
 \eea
 
  \noindent
 If $H<\frac 12$ we take $r_n= n^{-H}$ and
 \bea  \label{hk4}
 {\sf mqTor} =  r_n^{-1}  G_n^{(3)} \overset{L^p} {\rightarrow}  2{\colred H^2}
 \Gamma(2H)^2B_1^2.
 \eea
 
Notice that {\colred the limit is discontinuous at $H=\frac 12$.} %and is continuous from the left. 
With there preliminaries, we can {\color {black} now}  proceed to deduce the asymptotic expansion for  $E[f(Z_n)]$ {\colred without classification for $H$}.
%%%%
\begin{en-text}
\medskip
\noindent {\it (ii) Case $H\not = \half$}. We have $ G_\infty= c_H^2 B_1^2$,  where $c_H$ is the constant defined in  (\ref{e:ch})  and 
$D_t G_\infty= 2c_H^2 B_1 \mathbf{1}_{[0,1]}(t)$. 
Applying  (\ref{hk1}) yields
 \beas
 \langle DG_\infty, u_n \rangle_{\HH} 
&=& 
2c_H^2n^{H}B_1\langle t^nB_t{\bf 1}_{[0,1]}(t),{\bf 1}_{[0,1]}(t)\rangle_\HH
\\&=& 
2c_H^2n^{H}B_1\int_0^1 B_t t^nH\big\{t^{2n-1}+(1-t)^{2H-1}\big\}dt.
\eeas
Using the decomposition $B_1B_t= B_1^2 + B_1(B_t-B_1)$, we can write
\beas
 \langle DG_\infty, u_n \rangle_{\HH} 
&=&
2c_H^2n^{H}B_1^2\int_0^1t^nH\big\{t^{2n-1}+(1-t)^{2H-1}\big\}dt
\\&&
+2c_H^2n^{H}B_1\int_0^1t^n(B_t-B_1)H\big\{t^{2n-1}+(1-t)^{2H-1}\big\}dt
\\&=&
2c_H^2 B_1^2\bigg\{ \frac H3 n^{H-1} +n^H\frac{\Gamma(n+1)\Gamma(2H)}{\Gamma(n+1+2H)}\bigg\}+R_n,
 \eeas
 where $\|R_n \|_p= O(n^{-1})+O(n^{-2H})$ for all $p\ge 2$.
 As a consequence, we obtain the following results:
 
 \noindent
 If $H>\frac 12$ we take $r_n= n^{H-1}$ and
 \bea \label{hk3}
 {\sf mqTor} &=&  r_n^{-1}  G_n^{(3)} \overset{L^p} {\rightarrow}  \frac{2}{3}H^2\Gamma(2H)B_1^2.
 \eea
 
  \noindent
 If $H<\frac 12$ we take $r_n= n^{-H}$ and
 \bea  \label{hk4}
 {\sf mqTor} =  r_n^{-1}  G_n^{(3)} \overset{L^p} {\rightarrow}  2H\Gamma(2H)^2B_1^2.
 \eea
 
Notice that   the limit is discontinuous at $H=\frac 12$ and is continuous from the left. With there preliminaries, we can not proceed to deduce the asymptotic expansion for  $E[f(Z_n)]$.
\end{en-text}
%%%%

\subsection{Brownian motion case} %(2017.04.18)
 
 We will analyze the asymptotic behavior of the quasi  tangent and the quasi torsion, which are the main ingredients in the asymptotic  expansions. 
\subsubsection{Quasi tangent}\label{170805-2} 
Let us now establish  the asymptotic  behavior of the quasi tangent, defined by
\[
{\sf qTan}=\sqrt{n} G_n^{(2)}=\sqrt{n} \left( \langle DZ_n, u_n \rangle_{\HH} -G_\infty \right).
\]
We have, for $s\in [0,1]$,
\[
D_sZ_n=\sqrt{n}  s^n B_s + \sqrt{n} \int_s^1 t^n dB_t.
\]
Therefore,
\[
 \langle DZ_n, u_n \rangle_{\HH}= n \int_0^1 s^{2n} B_s^2 ds + n \int_0^1 s^n B_s \left( \int_s^1 t^n dB_t \right) ds.
 \]
 Then, 
 \begin{eqnarray*}
 G^{(2)}_n&=&  \langle DZ_n, u_n \rangle_{\HH} -\frac 12 B_1^2 \\
  &=&
  n \int_0^1 s^{2n} (B_s^2-B_1^2) ds + B_1^2 \left( \int_0^1 ns^{2n} ds -\frac 12 \right) \\
  &&+n \int_0^1 s^n (B_s-B_1) \left( \int_s^1 t^n dB_t \right) ds + n B_1 \int_0^1 s^n \left( \int_s^1 t^n dB_t \right) ds. 
 \end{eqnarray*}
 Using the decomposition $(B_s^2-B_1^2) =(B_s-B_1)^2- 2B_1 (B_1-B _s)$, yields
  \begin{eqnarray*}
 G^{(2)}_n&=& n \int_0^1 s^{2n}(B_s-B_1)^2ds
 {\colorb \>-\>}2nB_1 \int_0^1 s^{2n}(B_1-B_s)ds -\frac {B_1^2} {4n+2}\\
 &&
 {\colorb \>-\>}n \int_0^1 s^n (B_1-B_s) \left( \int_s^1 t^n dB_t \right) ds + B_1 \frac n{n+1} \int_0^1t^{2n+1} dB_t\\
 &=& Z_{n,1} + Z_{n,2} + Z_{n,3},
  \end{eqnarray*}
  where
   \begin{eqnarray*}
  Z_{n,1} &=&n \int_0^1 s^{2n}[(B_1-B_s)^2- (1-s)]ds 
  {\colred -}n \int_0^1 s^n \left[ (B_1-B_s) \left( \int_s^1 t^n dB_t \right) -\frac {1-s^{n+1}}{n+1} \right] ds , \\
  Z_{n,2} &=& -\frac { n^2} {2(n+1)^2 (2n+1)} -\frac {B_1^2} {4n+2}, \\
  Z_{n,3}&=& - \frac { n} {( n+1) (2n+1)} B_1 \int_0^1 t^{2n+1} dB_t.
  \end{eqnarray*}
  The term $Z_{n,1}$ belongs to the second Wiener chaos and  it can be written as a double stochastic integral:
   \begin{eqnarray*}
   Z_{n,1} &=&  \int_0^1 2ns^{2n} \left( \int_{[s,1]^2} dB_r dB_u \right) ds 
  {\colred -} \int_0^1 ns^{n} \left( \int_{[s,1]^2} (r^n+u^n)dB_r dB_u \right) ds= I_2(f_n),
     \end{eqnarray*}
     where
%     \[
%     f_n(r,u)= \frac {2n}{2n+1} \min(r,u)^{2n+1}+ \frac n{n+1}  \min(r,u)^{n+1} \max(r,u)^n.
%     \]
      \beas
     f_n(r,u)&=& \frac{n}{2n+1} \min(r,u)^{2n+1}- \frac{n}{2(n+1)}  \min(r,u)^{n+1} (r^n+u^n).
     \eeas
     It is not difficult to check that $n^2\|f_n \|_{\mfh^{\otimes2}}^2$ converges to a constant as $n$ tends to infinity. 
     Therefore, $\|Z_{n,1}\|_2 =O(n^{-1})$.
     Clearly, we also have  $\|Z_{n,2}\|_2=O(n^{-1})$. 
Finally,  $\|Z_{n,3}\|_2=O(n^{-3/2})$. 
Consequently, 
$\|\sqrt{n}G_n^{(2)}\|_2=O(n^{-1/2})$, and hence the effect of the quasi tangent disappears 
in the limit, that is, ${\mathfrak S}^{(2,0)}_0=0$.

\begin{en-text}
$Z_{n,3}$ converges in law to $B_1 \eta$, where $\eta$ is a  $N(0,1)$-random variable independent of the process $B$.
     As a consequence, we have proved the following result.
     
     \begin{prop}
     The random variable $\sqrt{n} G^{(2)}_n$ converges in law to  $B_1 \eta$, where $\eta$ is a  $N(0,1)$-random variable independent of the process $B$.
     \end{prop}
     
     However, the tangent does not contribute the the asymptotic expansion and the corresponding symbol $G^{2,0}_0$ vanishes. Indeed,
     \[
     \sqrt{n} E\left[\Psi(z) B_1 \int_0^1 t^{2n+1} dB_t \right] =\sqrt{n} E\left[  \int_0^1 D_t \left( \exp(-\frac 12 z^2G_\infty) B_1 \right) t^{2n+1} dt \right] =O(n^{-0.5}).
     \]
\end{en-text}

     \medskip
     \noindent
\subsubsection{Quasi torsion}\label{170805-3}
Let us first recall the definition of the quasi torsion
    \[
 {\sf qTor} =\sqrt{n}    \langle D \langle DZ_n, u_n \rangle_{\HH}, u_n \rangle_{\HH} =\sqrt{n}  \langle DG^{(2)}_n, u_n \rangle_{\HH} +\sqrt{n}  \langle DG_\infty, u_n \rangle_{\HH}.
     \]
Since $G_n^{(2)}$ is in the second chaos, it follows that 
$\sqrt{n}\|\|DG_n^{(2)}\|_\mfh\|_2=O(n^{-1/2})$ from 
$\sqrt{n}\|G_n^{(2)}\|_2=O(n^{-1/2})$ in Section \ref{170805-2}. 
Therefore, from (\ref{hk2}) we deduce
\beas 
\sqrt{n}\langle D \langle DZ_n, u_n \rangle_{\HH}, u_n \rangle_{\HH}
&\overset {L^p} {\rightarrow}& 
B_1^2
\eeas
for all $p\ge 2$.
Thus we obtain ${\mathfrak S}^{(3,0)}=3^{-1}B_1^2$.

 \subsubsection{Asymptotic expansion}
From the computations in Sections \ref{170805-2} and \ref{170805-3}, we deduce that conditions (\ref{b1}), (\ref{b2}), (\ref{b32}), (\ref{b33}) and (\ref{b4}) are satisfied for all $p\ge 2$. Thus, taking $\psi_n=1$,  assumption [B] holds and we can apply  Theorems  \ref{20160924-1} and \ref{thm2}. In this way, we obtain
\bea  \nn
E[f(Z_n)] &=& E[f(G_\infty^{1/2}\zeta)]+\frac{1}{\sqrt{n}}E[{\mathfrak S}^{(3,0)}(\partial_z^3) f(G_\infty^{1/2}\zeta)] 
+\rho_n^{(1)}(f) \\
&=&
  E[f(\frac 12 |B_1| \zeta)]+\frac{1}{3\sqrt{n}}E[B_1^2 f ^{(3)}(\frac 12 |B_1| \zeta)] 
+\rho_n^{(1)}(f)
\label{20170504-6} 
\eea
for $f\in C^3_b(\bbR)$, 
where   $\zeta\sim N(0,1)$  is independent of $B_1$, and  $\rho_n^{(1)}(f)=o(n^{-\frac 12})$.
\onelineskip

\begin{en-text}
The second term on the right-hand side is the modified quasi torsion 
$G_n^{(3)}$ investigated in (i), 
      it suffices to analyze the term  $\langle DG^{(2)}_n, u_n \rangle_{\HH} $. We can write
        \[
        \langle DG^{(2)}_n, u_n \rangle_{\HH} =\langle DF_{n,1}, u_n \rangle_{\HH}  +\langle DF_{n,2}, u_n \rangle_{\HH}  +\langle DF_{n,3}, u_n \rangle_{\HH}  .
        \]
        Then, 
        \[
\|  \langle DF_{n,1}, u_n \rangle_{\HH} \|_p\le  \| \|  F_{n,1}\|_{\HH}\|_{2p} \| \| u_n \|_{\HH}  \|_{2p} = \| F_{n,1}\|_{1,2p} \| u_n \|_{\HH}  \|_{2p}.
     \]
   Because $F_{n,1}$ is in as fixed chaos, $ \| F_{n,1}\|_{1,2p}$ is bounded by a constant times $ \| F_{n,1}\|_{2p}$, which is of order  $n^{-1}$ and $ \| u_n \|_{\HH}  \|_{2p}$ is bounded by a constant. The same argument applies to  $\langle DF_{n,2}, u_n \rangle_{\HH} $. Finally, writing $c_n=\frac {4n^2+3n} {( n+1)(2n+2)} $, we obtain
   \begin{eqnarray*}
   \langle DF_{n,3}, u_n \rangle_{\HH} &=& c_n\langle DB_1, u_n \rangle_{\HH}\left( \int_0^1 t^{2n+1} dB_t\right) +c_n B_1 \left\langle D\left( \int_0^1 t^{2n+1} dB_t\right), u_n \right \rangle _{\HH} \\
   &=&c_n \sqrt{n} \int_0^1 t^n B_t dt \left( \int_0^1 t^{2n+1} dB_t\right)+ c_nB_1 \sqrt{n} \int_0^1 t^{2n+1} B_t dt \\
   &=& A_{1,n} + A_{2,n}.
   \end{eqnarray*}
   The term $ A_{1,n} $ is of order  $n^{-1}$ and it does not contribute to the limit, but
   \[
   \sqrt{n}  A_{2,n}
   \]
   converges in $L^p$ to  $B_1^2$.  In conclusion,  the quasi torsion $\sqrt{n}      \langle D \langle DF_n, u_n \rangle_{\HH}, u_n \rangle_{\HH}$  converges in $L^p$ to  $2B_1^2$.
\end{en-text}

 \subsection{Fractional Brownian motion. Case $H>\frac 12$}% Revised 2017.05.04
 
 Recall that in that case $r_n= n^{H-1}$.

\subsubsection{Quasi tangent}\label{170805-5}
 We are going to establish the convergence in law of the tangent and show that it  does not contribute to the asymptotic expansion.  We have, for $s\in [0,1]$,
\[
D_sZ_n= n^H  s^n B_s +  n^H \int_s^1 t^n dB_t.
\]
Therefore,
\bea\label{20170430-2}
 \langle DZ_n, u_n \rangle_{\HH}=   \|u_n\|_{\HH}^2 + n^H 
 \bigg\langle u_n,  \int_{\cdot}^1  t^n dB_t  \bigg\rangle_{\HH}
 =:  \|u_n\|_{\HH}^2+ \Phi_n,
 \eea
 and the  quasi tangent  {\sf qTan} is given by
 \[
  {\sf qTan}= n^{1-H} 
 G^{(2)}_n=  n^{1-H} \left( \|u_n\|_{\HH}^2- c_H^2 B_1^2 +   \Phi_n\right).
 \]
 
 Let 
\bea\label{20170420-1}
 \sigma_{H,1}^2&=& 2 H^2(2H-1)^2\int_{[0,1]^4} |\log y_1 -\log x_1|^{2H-2}   |\log y_2-\log x_2|^{2H-2}  \nn\\
 &&\times 
 ( | \log y_1|^{2H} + | \log y_2|^{2H} - |\log y_1 -\log y_2 |^{2H} ) dx_1dx_2 dy_1 dy_2.
  \eea
and 
\bea\label{20170420-2} 
   \sigma^2_{H,2} =     H^3(2H-1)^3 \int_{[0,1]^2} |1-s_1  |^{2H-2} |1-s_2 |^{2H-2} ds_1 ds_2  \int_{[0,1]^2} | \log x_1 -\log x_2 |^{2H-2} dx_1 dx_2.
\eea

 \begin{prop}
 For the term $ \|u_n\|_{\HH}^2$ we have
 \begin{equation} \label{eq1}
n^{H} (  \|u_n\|_{\HH}^2 - c_H^2 B_1^2) \overset{\mathcal{L}} {\rightarrow}  \sigma_{H,1} B_1 \zeta
\end{equation}
where $\zeta$ is a $N(0,1)$-random variable  independent of $B$. On the other hand,
 \begin{equation} \label{eq2}
n^{1-H}  \Phi_n  \overset{\mathcal{L}} {\rightarrow}  \sigma_{H,2} B_1 \zeta,
\end{equation}
where $\zeta$ is a $N(0,1)$-random variable  independent of $B$. As a consequence, taking into account that $H>\frac 12$, we obtain
\bea\label{20170430-1}
 {\sf qTan\>=\>}n^{1-H} G^{(2)}_n \overset{\mathcal{L}} {\rightarrow}   \sigma_{H,2} B_1 \zeta.
\eea
 \end{prop}
 
 \noindent
 {\it Proof:}  We first show (\ref{eq1}). We can write
     \begin{eqnarray*}
     \|u_n\|_{\HH}^2 - c_H^2 B_1^2 
     %&=& H(2H-1) n^{2H} \int_0^1 \int_0^1 t^n s^n B_tB_s |t-s|^{2H-2} dsdt\\
     &=& H(2H-1) \left( n^{2H}  \int_0^1 \int_0^1 t^n s^n B_tB_s |t-s|^{2H-2} dsdt -\Gamma(2H-1) B_1^2 \right) \\
     &=& H(2H-1) n^{2H}  \int_0^1 \int_0^1 t^n s^n [B_tB_s  -B_1^2] |t-s|^{2H-2} dsdt \\
     &&+  H(2H-1) \left( n^{2H}  \int_0^1 \int_0^1 t^n s^n |t-s|^{2H-2} dsdt -\Gamma(2H-1)  \right) 
      B_1^2 \\
     &=:& A_{n,1} + A_{n,2} B_1^2.
     \end{eqnarray*}
     The term $A_{n,2}$ is a deterministic term bounded by $Cn^{-1}$,  therefore it does not contribute to the limit. In order to handle the term $A_{n,1}$ we make the decomposition
     \[
     B_sB_t -B_1^2= B_1(B_t-B_1) + (B_s-B_1)( B_t-B_1) + B_1( B_s-B_1).
     \]
We claim that the product $ (B_s-B_1)( B_t-B_1) $ does not contribute to the limit of $n^ H A_{n,1}$.  In fact,
     \begin{eqnarray*}
&& H(2H-1) n^{2H}\int_0^1 \int_0^1 s^n t^n |B_s-B_1| |B_t-B_1| |t-s|^{2H-2} dsdt 
=n^{2H}  \| s^n |B_s-B_1| \|_{\HH}^2 \\
&& \qquad  \le   C_H n^{2H} \| s^n |B_s-B_1| \|_{L^{1/H}}^2  
=C_H  n^{2H}\left( \int_0^1 s^{\frac nH} |B_s-B_1|^{\frac 1H}  ds\right) ^{2H}.
     \end{eqnarray*}
     By Minkowski's inequality,  the expectation of this quantity is estimated as {\color {black} follows}
\beas &&
     C_Hn^{2H}\left\|\int_0^1s^{\frac{n}{H}}|B_s-B_1|^{\frac{1}{H}}ds\right\|_{L^{2H}(\Omega)}^{2H}
     \leq
     C_Hn^{2H}\left(\int_0^1s^{\frac{n}{H}}\left\||B_s-B_1|^{\frac{1}{H}}\right\|_{L^{2H}(\Omega)}ds\right)^{2H}
     \\&\leq&
 C_H  n^{2H}  \left( \int_0^1 s^{\frac nH} (1-s) ds \right)^{2H} \le C'_H n^{-2H}.
\eeas
Therefore, it suffices to consider the term
 \[
 \tilde{A}_{n,1} =2B_1H(2H-1) n^{2H}  \int_0^1 \int_0^1 t^n s^n (B_s-B_1) |t-s|^{2H-2} dsdt 
 =: B_1 A_{n,3},
 \]
 and to show that $n^HA_{n,3}$ converges in law to a Gaussian random variable with mean zero and variance $\sigma^2_{H,1}$ independent of $\{ B_t, t\in [0,1]\}$. This a consequence of the following two facts:
 \begin{itemize}
 \item[(i)]  $E(n^{2H} A_{n,3}^2) \rightarrow \sigma^2_{H,1}$. 
 \item[(ii)]    $E(n^{H}  A_{n,3} B_t ) \rightarrow 0$, for any $t\in [0,1]$.
 \end{itemize}
 The proof of (i) is based on the computation of the limit of the following quantity
      \begin{eqnarray*}
 && 4H^2 (2H-1)^2 n^{6H} \int_{[0,1]^4} s_1^n t_1^n s_2^n t_2^n E[(B_{t_1}-B_1)(B_{t_2} -B_1)] |t_1-s_1| ^{2H-2}
 |t_2-s_2| ^{2H-2} ds_1ds_2 dt_1 dt_2 \\
  &&=2H^2 (2H-1)^2 n^{6H} \int_{[0,1]^4} s_1^n t_1^n s_2^n t_2^n\\
  &&\quad \times  ( |1-t_1|^{2H} + |1-t_2|^{2H} - |t_1-t_2|^{2H} ) |t_1-s_1| ^{2H-2}
 |t_2-s_2| ^{2H-2} ds_1ds_2 dt_1 dt_2 .
 \end{eqnarray*}
 This limit can be evaluated using the change of variables $ s_1^{n+1} =x_1$,  $ s_2^{n+1} =x_2$, $t_1^{n+1} =y_1$
 and $  t_2^{n+1} = y_2$, which leads
 to   the representation (\ref{20170420-1}) of $\sigma_{H,1}^2$.
%       \begin{eqnarray*}
% \sigma_{H,1}^2&=& 2 H^2(2H-1)^2\int_{[0,1]^4} |\log y_1 -\log x_1|^{2H-2}   |\log y_2-\log x_2|^{2H-2}  \\
%&&\times 
% ( |1-\log y_1|^{2H} + |1-\log y_2|^{2H} - |\log y_1 -\log y_2 |^{2H} ) dx_1dx_2 dy_1 dy_2.
%  \end{eqnarray*}
  The proof of (ii) can be done in a similar way.  This concludes the proof of (\ref{eq1}).
  
  For (\ref{eq2}), we can write
\bea  \nn
  \Phi_n&=& H(2H-1) n^{2H} \int_0^1 \int_0^1 t^n B_t \left(\int_s^1 \theta ^n dB_{\theta} \right) |t-s|^{2H-2} dsdt\\\nn
  &=&   H(2H-1) n^{2H} \int_0^1 \int_0^1 t^n (B_t-B_1) \left(\int_s^1 \theta ^n dB_{\theta} \right) |t-s|^{2H-2} dsdt\\ \nn
  &&\quad +  H(2H-1) B_1 n^{2H} \int_0^1 \int_0^1 t^n \left(\int_s^1 \theta ^n dB_{\theta} \right) |t-s|^{2H-2} dsdt\\   
  &{\color {black}=:}&  \Phi_{n,1} +  \Phi_{n,2}.  \label{gh1}
\eea
    We first show that $\Phi_{n,1}$ does not contribute to the limit:
        \begin{eqnarray*}
        |\Phi_{n,1}| &=& n^{2H} \left \langle t^n (B_t-B_1), \int_{\cdot} ^1 \theta^n dB_\theta  \right \rangle_{\HH} \\
       & \le&  n^{2H} \|  t^n (B_t-B_1)\|_{\HH} \left \| \int_{\cdot} ^1 \theta^n dB_\theta  \right\|_{\HH}\\
       &\le&  C_Hn^{2H} \|  t^n (B_t-B_1)\|_{L^{\frac 1H}} \left \| \int_{\cdot} ^1 \theta^n dB_\theta  \right\|_{L^{\frac 1H}}\\
       &=& C_H n^{2H}  \left[\int_0^1 t^{\frac nH} |B_t-B_1|^{\frac 1H} dt \int_0^1 | \int_s^1 \theta ^n dB_\theta |^{\frac 1H} ds
       \right]^H.
  \end{eqnarray*}
  Then, taking expectation and using Minkowski's inequality, we get
  \begin{en-text}
        \begin{eqnarray*}
        E      |{\colorb \Phi}_{n,1}|  
        &\le & 
        {\colorb 
        C_h n^{2H} \left\|\int_0^1t^{\frac{n}{H}}|B_t-B_1|^{\frac{1}{H}}dt\right\|_{L^{2H}(\Omega)}^H
        \left\|\int_0^1\left|\int_0^1\theta^ndB_\theta\right|^{\frac{1}{H}}ds\right\|_{L^{2H}(\Omega)}^H
        }
        \\&\le & 
        C_h n^{2H}   \left( \int_0^1 t^{\frac nH} (1-t) dt \right) ^H \left( \int_0^1  \left| E \left( \int_s^1 \theta^n dB_\theta \right)^2  \right| ^{\frac 1 {2H}} ds \right)^H \\
         &\le & C_H  \left( \int_0^1   \| \theta^n  \mathbf{1}_{[s,1]} \|_{\HH} ^{\frac 1H} ds \right)^H  \\
         &\le & C_H \left( \| \theta^n \|_{L^{\frac 1H}} ^{\frac 1H} \right) ^H\\
         &\le & C_H n^{-H},
        \end{eqnarray*}
\end{en-text}
        \begin{eqnarray*}
        E      [| \Phi_{n,1}| ^2]
        &\le &   
        C_H n^{4H} \left\|\int_0^1t^{\frac{n}{H}}|B_t-B_1|^{\frac{1}{H}}dt\right\|_{L^{4H}(\Omega)}^{2H}
        \left\|\int_0^1\left|\int_s^1\theta^ndB_\theta\right|^{\frac{1}{H}}ds\right\|_{L^{4H}(\Omega)}^{2H}
        \\&\le & 
        C_H n^{4H}   \left( \int_0^1 t^{\frac nH} (1-t) dt \right) ^{2H} 
        \left( \int_0^1  \left\| \int_s^1 \theta^n dB_\theta  \right\|_{L^4(\Omega)} ^{\frac 1 {H}} ds \right)^{2H} \\
         &\le & C_H  \left( \int_0^1   \| \theta^n  \mathbf{1}_{[s,1]} \|_{\HH} ^{\frac 1H} ds \right)^{2H}  \\
         &\le & C_H \left( \| \theta^n \|_{L^{\frac 1H}} ^{\frac 1H} \right) ^{2H}\\
         &\le & C_H n^{-2H},
        \end{eqnarray*}
        and  $n^{1-H}         \| \Phi_{n,1}\|_2 $ converges to zero as $n$ tends to infinity.
        Finally, it suffices to consider the term
        \[
         \Phi_{n,2} = B_1 \tilde{\Phi}_{n,2},
        \]
        where
        \bea
        \tilde{ \Phi}_{n,2} &=&  H(2H-1)n^{2H} \int_0^1 \int_0^1 t^n \left(\int_s^1 \theta ^n dB_{\theta} \right) |t-s|^{2H-2} dsdt. \label{gh2}
        \eea
        We claim that  $n^{1-H}         \tilde{ \Phi}_{n,2}$ converges in law to a Gaussian random variable with zero mean and variance $\sigma^2_{H,2}$ independent of $\{B_t, t\in [0,1]\}$. 
        This is a consequence of the following two facts:
 \begin{itemize}
 \item[(i)]  $E(n^{2-2H} \tilde{ \Phi}_{n,2}^2) \rightarrow \sigma^2_{H,2}$. 
 \item[(ii)]    $E(n^{1-H} \tilde{ \Phi}_{n,2} B_t ) \rightarrow 0$, for any $t\in [0,1]$.
 \end{itemize}
We first show (i):
        \begin{eqnarray*}
E(n^{2-2H} \tilde{ \Phi}_{n,2}^2) &=& H^2(2H-1)^2n^{2+2H}  \int_{[0,1]^4} t_1^n t_2 ^n E\left[ \left(\int_{s_1} ^1 \theta^n dB_\theta \right) \left(\int_{s_2} ^1 \theta^n dB_\theta \right) \right] \\
&&\times |t_1 -s_1|^{2H-2} |t_2-s_2 |^{2H-2} ds_1dt_1 ds_2 dt_2\\
&=& H^3(2H-1)^3n^{2+2H}  \int_{[0,1]^4} t_1^n t_2 ^n  \int_{s_1}^1 \int_{s_2} ^1  \theta_1^n \theta_2^n 
|\theta_1-\theta_2|^{2H-2} d \theta_1 d \theta_2\\
&&\times |t_1 -s_1|^{2H-2} |t_2-s_2 |^{2H-2} ds_1dt_1 ds_2 dt_2.
\end{eqnarray*}
 Using the change of variable $t_1^{n+1} =y_1$, $t_2^{n+1} =y_2$, 
 $ \theta_1^{n+1} =x_1$ and $ \theta_2^{n+1} = x_2$, we can show that this quantity  converges to 
 $\sigma_{H,2}^2$ given in (\ref{20170420-2}).
%\[  
%   \sigma^2_{H,2} =     H^3(2H-1)^3 \int_{[0,1]^2} |1-s_1  |^{2H-2} |1-s_2 |^{2H-2} ds_1 ds_2  \int_{[0,1]^2} | \log x_1 -\log x_2 |^{2H-2} dx_1 dx_2.
%        \]
The proof of (ii) can be done in a similar way.
\qed\halflineskip

In spite of the preceding proposition, the   quasi tangent does not contribute to the  asymptotic expansion derived in the last section. In fact, 
the convergence (\ref{20170430-1}), together with uniform integrability, gives 
$\lim_{n\to\infty}E[\Psi(\sfz)\>{\sf qTan}]=0$, that is, 
${\mathfrak S}^{(2,0)}_0=0$. More strongly,  using the duality relationship between the Skorohod integral and the derivative operator (IBP formula),  we   can show this fact directly for $\Phi_{n,2}$ as follows: 
 \begin{eqnarray*}
&& n^{1-H} 
\Bigg|
E\left[\Psi (\sfz) n^{2H} B_1\int_0^1 \int_0^1 t^n \left(\int_s^1 \theta ^n dB_{\theta} \right) |t-s|^{2H-2} dsdt\right]
\Bigg|
\\
&&=n^{1+H} \Bigg|E\left[ \int_0^1 \int_0^1 t^n  |t-s|^{2H-2}   \left \langle  D_\theta\left(B_1\exp(-\frac 12 \sfz^2 c_H^2B^2_1   ) \right) , {\bf 1}_{[s,1]} ( \theta^n) \right\rangle_{\HH}  
    dsdt  \right]\Bigg|  \\
    && \le Cn^{H-1}.
 \end{eqnarray*}
By (\ref{20170430-1}), ${\sf qTan}$ never converges to zero in probability. 
Thus $D{\sf qTan}$ potentially has some effect at the rate $n^{1-H}$.

 \medskip
 \noindent
\subsubsection{Quasi torsion}\label{170805-6} 
  By (\ref{20170430-2}), we have
     \beas
      {\sf qTor} &=& 
    n^{1-H}  \big\langle D \langle DZ_n, u_n \rangle_{\HH}, u_n \big\rangle_{\HH} 
     \\&=&
    n^{1-H} \big(   \big\langle 
      D\big( \|u_n \|^2_{\HH} - c_H^2 B_1^2\big), u_n \big\rangle_{\HH} +  \langle D\Phi_n, u_n \rangle_{\HH}+\langle DG_\infty, u_n \rangle_{\HH}\big).
     \eeas
    Notice that
    \beas
     \|  \| u_{{\colred n}}\|_{\HH} \|_2  &\le& C_H n^H\| \| t^{{\colred n}} B_t \|_{L^{1/H}} \|_2 \le C'_H,
     \eeas
     which implies that  $ \|  \| u_{{\colred n}}\|_{\HH} \|_p$ is uniformly bounded for any $p\ge 2$. On the other hand, the computations in the  previous section imply
     $\| \|  D \left(  \|u_n \|^2_{\HH} - c_H^2 B_1^2 \right) \|_{\HH}\|_{p}=O(n^{-H})$ for any $p\ge 2$. Therefore,
          $\| n^{1-H}    \big\langle 
      D\big( \|u_n \|^2_{\HH} - c_H^2 B_1^2\big), u_n \big\rangle_{\HH} \|_p=O(n^{1-2H})$ and this term does not contribute to the limit. 

Consider   the term  $\langle D\Phi_n, u_n \rangle_{\HH} $.  Using the decomposition (\ref{gh1}), we can write
        \[
        \langle D\Phi_n, u_n \rangle_{\HH} =\langle D\Phi_{n,1}, u_n \rangle_{\HH}  +\langle D\Phi_{n,2}, u_n \rangle_{\HH}.
        \]
The term $\langle D\Phi_{n,1}, u_n \rangle_{\HH}$ does not contribute to the limit 
 since $\Phi_{n,1}$ is in the second chaos and $\|n^{1-H}\Phi_{n,1}\|_{2}\to0$.
As for  $\langle D\Phi_{n,2}, u_n \rangle_{\HH}$, we can write
\[
\langle D\Phi_{n,2}, u_n \rangle_{\HH}=\langle DB_1, u_n \rangle_{\HH} \tilde{\Phi}_{n,2} +
B_1\langle D\tilde{\Phi}_{n,2}, u_n \rangle_{\HH},
\]
where $\tilde{\Phi}_{n,2}$ is defined in (\ref{gh2}).
%%%%%
The term  $\langle DB_1, u_n \rangle_{\HH} \tilde{\Phi}_{n,2} $ does not contribute to the limit at the rate $n^{1-H}$ 
 in $L^p$, $p\ge 2$ due to the computations in the previous section. On the other hand, for the second term we can write
 \begin{eqnarray*}
 n^{1-H} B_1\langle D\tilde{\Phi}_{n,2}, u_n \rangle_{\HH}  &=&
  n^{1+H} H (2H-1)  B_1  \int_0^1 \int_0^1  t^n |t-s|^{2H-2}   \langle   u_n(\xi) ,   \theta^n {\bf 1}_{[s,1]} (\theta) \rangle_{\HH}  dsdt\\
  &=& n^{1+2H} H^2(2H-1)^2 B_1 \\
  &&\times \int_{[0,1]^4}  t^n |t-s|^{2H-2}B_\xi  \xi^n \theta^n |\xi-\theta|^{2H-2} {\bf 1}_{[s,1]} (\theta) d\xi d\theta dsdt\\
  &=&   C_n B_1^2 + H^2 (2H-1)^2 \Delta_n,
  \end{eqnarray*}
  where
  \[
  C_n=n^{1+2H} H^2(2H-1)^2 \int_{[0,1]^4}  t^n |t-s|^{2H-2}  \xi^n \theta^n |\xi-\theta|^{2H-2} {\bf 1}_{[s,1]} (\theta) d\xi d\theta dsdt,
  \]
  and
  \[
  \Delta_n=  n^{1+2H}  B_1\int_{[0,1]^4}  t^n |t-s|^{2H-2}(B_\xi- B_1)  \xi^n \theta^n |\xi-\theta|^{2H-2} {\bf 1}_{[s,1]} (\theta) d\xi d\theta dsdt.
  \]
 % It is easy to check that $\| \Delta_n \|_p \le C n^{1-2H}$, so this term  does not contribute to the limit. On the other hand,
  With the change of variables  $t^{n+1} =x$, $\theta^{n+1} =y$, $\xi^{n+1} =z$,  we obtain
\beas
  \lim_{n\rightarrow \infty}  C_n &=&
 H^2 (2H-1)\int_{[0,1]^2}   | \log y -\log z |^{2H-2} dydz%=: c^2_{H,3} . 
\\&=&
H^2 (2H-1)\Gamma(2H-1)\>=\>c_H^2H. 
\eeas
\noindent
On the other hand, it is easy to check that $\| \Delta_n \|_p \le C  n^{-1}$, %n^{1-2H}$, 
so this term  does not contribute to the limit. 
  In conclusion,  taking into account  ({\ref{hk3}), the quasi torsion 
  $ {\sf qTor}=n^{1-H}     \langle D \langle DZ_n, u_n \rangle_{\HH}, u_n \rangle_{\HH}$ converges in $L^p$ to  
 $3c_H^2HB_1^2$ for all $p\ge 2$. In other words, ${\mathfrak S}^{(3,0)}=c_H^2HB_1^2$ 
for $H>1/2$, which is discontinuous at $H=1/2$.  In this way, we obtain the expansion
\bea  \label{expan2}
E[f(Z_n)] &=&   E[f( c_H |B_1| \zeta)]+ n^{H-1}E[ c_H^2 H B_1^2 f ^{(3)}( c_H |B_1| \zeta)]
+\rho_n^{(1)}(f),
\eea
for $f\in C^3_b(\bbR)$,  where $\zeta\sim N(0,1)$  is independent of $B_1$,
Again, from  the computations in Sections \ref{170805-5} and \ref{170805-6}, we deduce that conditions (\ref{b1}), (\ref{b2}), (\ref{b32}), (\ref{b33}) and (\ref{b4}) are satisfied for all $p\ge 2$. Thus, taking $\psi_n=1$,  assumption [B] holds and by Theorem \ref{thm2} $\rho_n^{(1)}(f)=o(n^{H-1})$.

\subsection{Fractional Brownian motion. Case $H<\frac 12$}% (2017.04.20)}

\begin{en-text}  
  As a consequence,
   \[
 \langle t^n ,  1\rangle_{\HH} = \frac H{n+2H} + \frac {nc_{n,H}}2.
 \]
\end{en-text}

  Let 
 $ c_{n,H}=\frac {\Gamma(2H+1) \Gamma(n)}{ \Gamma(n+2H+1)} $. 
 We need the following preliminary result.

 \begin{lemme}\label{170722-1}  The norm  $\| t^n \|_{\HH} ^2$ is given by
 \[
  \| t^n \|_{\HH} ^2=  \frac {n^2+ 2nH}{2n+2H}  c_{n,H}.
  \]
%    where  $ c_{n,H}=\frac {\Gamma(2H+1) \Gamma(n)}{ \Gamma(n+2H+1)} $.
  \end{lemme}
  
    \noindent
  {\it Proof:} We can write
 \begin{eqnarray*}
\| t^n \|_{\HH} ^2&=& E\left[  \left(\int_0^1 t^n dB_t \right)^2\right]\\
 &=& E\left[ \left( B_1- \int_0^1 nt^{n-1} B_t dt\right) ^2\right]   \\
 &=& 1- 2 \left( \frac {n+H}{n+2H} - \frac n2 c_{n,H} \right)+ n^2 \int_0^1 \int_0^1 t^{n-1} s^{n-1} E(B_tB_s)dsdt .
 \end{eqnarray*}
 We have
 \begin{eqnarray*}
n^2 \int_0^1 \int_0^1 t^{n-1} s^{n-1} E(B_tB_s)dsdt 
 &=&\frac {n^2}2\int_0^1 \int_0^1 t^{n-1} s^{n-1} (t^{2H} + s^{2H} -|t-s|^{2H})dsdt  \\
 &=& \frac n{n+2H} - \frac { n^2}{2n+2H} c_{n,H}
 \end{eqnarray*}
Therefore
 \[
 \| t^n \|_{\HH} ^2=  \frac {n^2+ 2nH}{2n+2H}  c_{n,H}= \frac n{2(n+H)}  \frac{ \Gamma(2H+1)\Gamma(n) }{\Gamma(n+2H)}.
 \]
\qed

 As a consequence, 
 \beas 
 \lim_{n\rightarrow \infty} n^{2H}  \| t^n \|_{\HH} ^2 &=&\frac 12\Gamma(2H+1) = c_H^2,
 \eeas
 where $c_H$ is the constant introduced in (\ref{e:ch}).

 \medskip
 \noindent
\subsubsection{Quasi tangent}\label{170805-7}
Recall  that $r_n= n^{-H}$ and the  quasi tangent is defined by
 \[
{\sf qTan} = n^{H} G^{(2)}_n = n^{H} \left( \langle DZ_n, u_n \rangle_{\HH} - G_{\infty}\right) .
 \]
 We know that
 \begin{equation} \label{y2}
  \langle DZ_n, u_n \rangle_{\HH}= \|u_n \|_{\HH}^2 +  n^{2H} \langle t^n B_t, \int_t^1 s^n dB_s\rangle_{\HH}.
  \end{equation}
  The inner product in the Hilbert space $\HH$ is more involved than in the case $H>\frac 12$, and it is convenient to rewrite the stochastic integral  $ \int_t^1 s^n dB_s$ using integration by parts: % It is possible to check this formula by computing $E[\delta({\bf 1}_{[t,1]}(\cdot)\cdot^n)G]$. 
  \begin{equation}  \label{y1}
   \int_t^1 s^n dB_s=B_1 - t^n B _t -n \int_t^1 B_s s^{n-1} ds.
   \end{equation}
   Substituting (\ref{y1}) into (\ref{y2}) yields
 \begin{eqnarray*}
  \langle DZ_n, u_n \rangle_{\HH} &=& \|u_n \|_{\HH}^2 +  n^{2H} \langle t^n B_t,  B_1 - t^n B _t -n \int_t^1 B_s s^{n-1} ds\rangle_{\HH}\\
  &=&n^{2H} \langle t^n B_t,  B_1  -n \int_t^1 B_s s^{n-1} ds\rangle_{\HH}\\
  &=& n^{2H}B_1 \langle t^n B_t, 1 \rangle_{\HH} - n^{2H+1}\int_0^1 \langle t^n B_t ,\mathbf{1}_{[0,s]}(t) \rangle_{\HH} s^{n-1} B_s ds.
  \end{eqnarray*}
  Putting $B_t= (B_t-B_1) +B_1$ and $B_tB_s= (B_t-B_1)(B_s-B_1)+B_1(B_t-B_1)+ B_1( B_s-B_1)+B_1^2$, we obtain
  \begin{eqnarray*}
   \langle DZ_n, u_n \rangle_{\HH} &=& 
   n^{2H}B_1^2 \left(  \langle t^n, 1 \rangle_{\HH} - n\int_0^1 \langle t^n ,\mathbf{1}_{[0,s]}(t) \rangle_{\HH} s^{n-1}  ds \right) \\
   &&+ n^{2H} B_1 \langle t^n(B_t-B_1), 1 \rangle_{\HH}  \\
   &&- n^{2H+1}B_1\int_0^1 \langle t^n (B_t-B_1) ,\mathbf{1}_{[0,s]}(t) \rangle_{\HH} s^{n-1}   ds \\
      &&- n^{2H+1}B_1\int_0^1 \langle t^n   ,\mathbf{1}_{[0,s]}(t) \rangle_{\HH} s^{n-1}  (B_s-B_1) ds \\
         &&- n^{2H+1}\int_0^1 \langle t^n (B_t-B_1)  ,\mathbf{1}_{[0,s]}(t) \rangle_{\HH} s^{n-1}  (B_s-B_1) ds. \\
      \end{eqnarray*}
      Actually, we can combine the second and third terms   and last two terms as follows:
       \begin{eqnarray*}
   \langle DZ_n, u_n \rangle_{\HH} &=& 
   n^{2H}B_1^2 \left(  \langle t^n, 1 \rangle_{\HH} - n\int_0^1 \langle t^n ,\mathbf{1}_{[0,s]}(t) \rangle_{\HH} s^{n-1}  ds \right) \\
   &&+ n^{2H} B_1 \langle t^n(B_t-B_1), t^n \rangle_{\HH}  \\
      &&- n^{2H+1} \int_0^1 \langle t^n  B_t  ,\mathbf{1}_{[0,s]}(t) \rangle_{\HH} s^{n-1}  (B_s-B_1) ds \\
         &=:& \sum_{i=1} ^3 A_{i,n}.
      \end{eqnarray*}
          The dominant term in the  limit will be $A_{1,n}$, which can be expressed as
      \[
      A_{1,n}=  n^{2H}B_1^2 \left(  \langle t^n, 1 \rangle_{\HH} - n\int_0^1 \langle t^n ,\mathbf{1}_{[0,s]}(t) \rangle_{\HH} s^{n-1}  ds \right)
      = n^{2H} B_1^2 \| t^n \|_{\HH}^2,
      \]
  {\color {black} and} converges to $G_\infty=H \Gamma(2H) B_1^2$ as $n$ tends to infinity by Lemma \ref{170722-1}. It is not difficult to check, using formulas (\ref{hk1}) and (\ref{hk11}), that  the other two  terms converge to zero in  $L^2$  as $n$ tends to infinity.
  
  We are going to show that the quasi tangent  does not contribute to the asymptotic expansion.  From the preceding computations, we deduce 
  \[
  {\sf qTan} = n^H  B_1^2[ n^{2H}  \|t^n\|^2_{\HH} -  c_H^2] + \sum_{i=2} ^3 n^H A_{i,n}.
  \]
  We examine each term of this expression as follows:
  
  \medskip
  \noindent
  {\bf (i) }  For the first term,   by Lemma \ref{170722-1},  we have
  \[
\tilde{A}_{1,n}:= n^H  B_1^2[ n^{2H}  \|t^n\|^2_{\HH} -  c_H^2]  = B_1^2 n^H \frac 12 \Gamma(2H+1) \left[\frac {n^{2H+1}}{(n+H)}  \frac{ \Gamma(n) }{\Gamma(n+2H)} - \1  \right],
  \]
  which converges to zero.
  
    \medskip
  \noindent
  {\bf (ii) } For the second term, by Lemma \ref{20170616-1} (ii), we have
  \begin{eqnarray}\label{170719-1}
 n^H A_{2.n}&=&n^{3H} B_1 \langle t^n(B_t-B_1), t^n \rangle_{\HH} 
 \nn\\
    &=&
 n^{3H}HB_1\int_0^1t^n(B_t-B_1)(1-t)^{2H-1}dt
 \nn\\&&
+n^{3H}HB_1\int_0^1t^n(B_t-B_1)\int_0^1|t-s|^{2H-1}{\rm sign}(t-s)ns^{n-1}dsdt.
  \end{eqnarray}
We claim that
  \bea\label{170719-2}
  \lim_{n\rightarrow \infty}   n^HE[ \Psi(z) A_{2.n}]=0.
  \eea
  In fact, integrating by parts, the factor $B_t-B_1$ produces a term of the form 
  $|t-1|^{2H}$ due to Lemma \ref{20170616-1} (iii), and then  we have 
  \[
  \int_0^1 t^n(1-t)^{4H-1}dt \simleq  n^{-4H}, 
  \]
hence the first term on the right-hand side of (\ref{170719-1}) converges to $0$. 
For two sequences numbers $a_n$ and $b_n$, 
$a_n\simleq b_n$ means that there exists a positive constant $C$ independent of $n$ such that 
$a_n\leq C b_n$ for all $n\in\bbN$.  
{\color {black} For the second term we apply the integration-by-parts formula} to $B_t-B_1$ 
as well as Lemma \ref{20170616-1} (iii) and Lemma \ref{170720-1} below 
to obtain the bound 
\beas 
\int_0^1t^n(1-t)^{2H}\int_0^1|t-s|^{2H-1}ns^{n-1}dsdt
&=&
O(n^{-4H}). 
\eeas
Therefore the second term on the right-hand side of (\ref{170719-1}) 
converges to $0$, which proves (\ref{170719-2}). }\\
\begin{en-text}\\&=&
\int_0^1t^n(1-t)^{2H}\int_0^t(t-s)^{2H-1}ns^{n-1}dsdt
+\int_0^1t^n(1-t)^{2H}\int_t^1(s-t)^{2H-1}ns^{n-1}dsdt
\\&=&
nB(n,2H)\int_0^1t^{2n+2H-1}(1-t)^{2H}dt
\\&&+\int_0^1t^n(1-t)^{2H}\bigg\{
\frac{(1-t)^{2H}n}{2H}-\int_t^1\frac{(s-t)^{2H}}{2H}n(n-1)s^{n-2}ds
\bigg\}dt
\\&\leq&
nB(n,2H)\int_0^1t^{2n+2H-1}(1-t)^{2H}dt
+\frac{n}{2H}\int_0^1t^n(1-t)^{4H}dt
\\&\leq&
nB(n,2H)B(2n+2H,2H+1)
+\frac{n}{2H}B(n+1,4H+1)
\\&\simleq&
n^{-4H}. 
\eeas
\end{en-text}
%%
%  \[
%  \int_0^1  \int_0^1t^n  |t-s| ^{4H-1} ns^{n-1}  dsdt \le C n^{-4H}.
%  \]
\begin{lemme}\label{170720-1}
Let $\alpha,\beta,\mu\in(-1,\infty)$ and $\nu\in[0,\infty)$. 
Let 
\beas
B(\alpha,\beta,\mu,\nu) &=& 
B(\mu+\nu+\beta+2,\alpha+1)B(\beta+1,\nu+1)
+\frac{1}{\beta+1}B(\mu+1,\alpha+\beta+2). 
\eeas
Then
\bd
\im[(i)] 
$\displaystyle
\int_0^1\int_0^1 t^\mu(1-t)^\alpha |t-s|^\beta s^\nu dsdt
\leq 
B(\alpha,\beta,\mu,\nu). 
$
\im[(ii)] For fixed $\alpha$ and $\beta$, 
it holds that 
\beas
B(\alpha,\beta,\mu,\nu)
&\sim&
\frac{\Gamma(\alpha+1)\Gamma(\beta+1)}{(\mu+\nu)^{\alpha+1}\nu^{\beta+1}}
+\frac{\Gamma(\alpha+\beta+2)}{(\beta+1)\mu^{\alpha+\beta+2}}
\eeas 
as $\mu,\nu\to\infty$. 
\ed
\end{lemme}
\begin{proof}
First, 
\beas 
\int_0^1t^\mu(1-t)^\alpha\int_0^t(t-s)^\beta s^\nu dsdt
&=&
\int_0^1t^{\mu+\nu+\beta+1}(1-t)^\alpha  dt\int_0^1(1-v)^\beta v^\nu dv
\\&=&
B(\mu+\nu+\beta+2,\alpha+1)B(\beta+1,\nu+1). 
\eeas
Next, 
\beas 
\int_0^1t^\mu(1-t)^\alpha\int_t^1(s-t)^\beta s^\nu dsdt
&=&
\int_0^1t^\mu(1-t)^\alpha\bigg\{
\frac{(1-t)^{\beta+1}}{\beta+1}
-\int_t^1\frac{(s-t)^{\beta+1}}{\beta+1}\nu s^{\nu-1}ds\bigg\}dt
\\&\leq&
\frac{1}{\beta+1}\int_0^1t^\mu(1-t)^{\alpha+\beta+1}dt
\\&=&
\frac{1}{\beta+1}B(\mu+1,\alpha+\beta+2). 
\eeas
%Lemma \ref{170720-1} 
{\color {black} Property (i) follows from these inequalities  and  (ii) is obvious. }
\end{proof}
\medskip
\noindent
{\bf (iii) } The third   term also does not produce contribution. 
By Lemma \ref{20170616-1} (i) and 
Lemma \ref{20170616-1} (iii) after integration-by-parts in $B_s-B_1$, 
we estimate $n^HE[\Psi(z)A_{3,n}]$ by
\beas&&
Cn^{3H+1}{\colred \int_0^1}\int_0^{{\colred 1}}  t^n ( t^{2H-1} + |t-s|^{2H-1}) s^{n-1}(1-s)^{2H}ds{\colred dt}
\\&\simleq&
n^{3H+1}\big\{B(2H,0,n-1,n+2H-1)+B(2H,2H-1,n-1,n)\big\}
\\&\simleq&
n^{H-1}+n^{-H}. 
\eeas

In this way, we have proved that ${\sf qTan}$ has no contribution in the limit, that is, 
${\mathfrak S}^{(2,0)}_0=0$.

\medskip
\noindent
\subsubsection{Quasi torsion}\label{170805-8}
The  quasi torsion can be written as
\[
{\sf qTor} = n^H \langle  D\langle DZ_n, u_n \rangle_{\HH} ,u_n \rangle_{\HH} = n^H(\langle DG^{(2)}_n , u_n\rangle_{\HH} + \langle DG_\infty, u_n \rangle_{\HH}).
\]
Let us show that $n^H \langle DG^{(2)}_n , u_n\rangle_{\HH} $ does not contribute   to the  asymptotic expansion.
 First, 
\beas 
\big\|\langle D\tilde{A}_{1,n},u_n\rangle_\HH\big\|_p
&=& 
\big\|\langle D(B_1^2),u_n\rangle_\HH\big\|_p\times o(1)\yeq o(1)
\eeas
for $p\ge 2$. 
We have 
\bea\label{170815-1} 
\big\|\big\langle DB_s,u_n\big\rangle_\HH \big\|_p
&=&
%\big\langle {\bf 1}_{[0,s]}(t),n^Ht^nB_t\big\rangle_\HH 
%\\&=&
\bigg\|n^H\int_0^1HB_tt^n\{t^{2H-1}-|t-s|^{2H-1}{\rm sign}(t-s)\}dt\bigg\|_p
\yeq
O(n^{H-1}). 
\eea
Therefore by (\ref{170719-1}), 
$\big\|\langle n^HDA_{2,n},u_n\big\rangle_\HH \big\|_p=O(n^{4H-2})=o(1)$. 
For the term $A_{3,n}$ we can write for any $p\ge 2$, 
using {\colred Lemma \ref{20170616-1} (i),} (\ref{170815-1}) {\colred and Lemma \ref{170720-1},} 
\beas
\big\|n^H  \langle DA_{3,n}, u_n \rangle_{\HH} \big\|_p
& \le & 
C\sup_{s\in [0,t] }  \big\|\langle DB_s,u_n\rangle_\HH\big\|_p  n^{3H+1}
{\colred \int_0^1}\int_0^{{\colred 1}}  t^n ( t^{2H-1} + |t-s|^{2H-1}) s^{n-1}ds{\colred dt}
\\&=&  O(n^{4H-2})+O(n^{2H-1})=o(1).
\eeas
Thus we have proved that  $\big\|n^H \langle DG^{(2)}_n,u_n\rangle_\HH\big\|_p=o(1)$.  

  Therefore,  taking into account  ({\ref{hk4}), the quasi torsion 
  $ {\sf qTor}$ converges in $L^p$ to  
 $2{\colred H} \Gamma(2H) c_H^2 B_1^2$ for all $p\ge 2$. In other words, ${\mathfrak S}^{(3,0)}=\frac 23  {\colred H}\Gamma(2H) c_H^2B_1^2$ 
for $H<1/2$.

\subsubsection{Asymptotic expansion}
  In this way, we obtain the expansion
\bea  \label{expan2}
E[f(Z_n)] &=&   E\big[f( c_H |B_1| \zeta)\big]
+ n^{-H}E\bigg[ \frac 23 {\colred H}c_H^2 \Gamma(2H) B_1^2 f ^{(3)}( c_H |B_1| \zeta)\bigg]
+\rho_n^{(1)}(f),
\eea
for $f\in C^3_b(\bbR)$,  where $\zeta\sim N(0,1)$  is independent of $B_1$,
Again, from  the computations in Sections \ref{170805-5} and \ref{170805-6}, we deduce that conditions (\ref{b1}), (\ref{b2}), (\ref{b32}), (\ref{b33}) and (\ref{b4}) are satisfied for all $p\ge 2$. Thus, taking $\psi_n=1$,  assumption [B] holds and by Theorem \ref{thm2} $\rho_n^{(1)}(f)=o(n^{-H})$.

%%%%%%%%%%%%%%%%%%%%%%%%%%%%%%%
\section{Quadratic form of a Brownian motion with predictable weights}\label{170813-3}
In this section, we consider a quadratic form of a Brownian motion 
with  predictable weights and 
show that the asymptotic expansion formula for the Skorohod integral 
reproduces the {\color {black}results} obtained in \cite{yoshida2013martingale,yoshida2012asymptotic}. 

\subsection{Quadratic form with random weights and $\mfh$-derivatives}
For a one-dimensional standard Brownian motion $B=\{B_t, t\in[0,1]\}$, let 
\beas 
Z_n&=& \sqrt{n}\sum_{j=1}^na_\tjm\int_\tjm^\tj\int_\tjm^tB_sdB_t
\>=\>  \sqrt{n}\sum_{j=1}^n2^{-1}a_\tjm\big\{(B_\tj-B_\tjm)^2-n^{-1}\big\},
\eeas
where $t_j=j/n$, 
$a_t=a(B_t)$ and  $a$ is an infinitely differentiable function with derivatives of moderate growth ($g$ has moderate growth if
$|g(x)| \le c \exp(\alpha |x|)$ for some constants $c>0$ and $0\le \alpha<2$).
%$a\in C^\infty_p(\bbR)$, the space of 
That is, $Z_n=\delta(u_n)$ with 
\beas 
u_n(t) &=& \sqrt{n}\sum_{j=1}^na_\tjm(B_t-B_\tjm){\bf 1}_{I_j}(t),
\eeas
where $I_j=[\tjm,\tj)$. In this situation $Z_n=M_n$ and  we have $W_n=W_\infty=0$, $N_n=0$, $X_n=X_\infty =0$ and $\psi=1$. 

Let 
\beas 
q_j&=&(B_\tj-B_\tjm)^2-n^{-1}=2\int_\tjm^\tj\int_\tjm^tdB_sdB_t.
\eeas 
We have 
\bea\label{20161117-1}
\langle DZ_n,u_n\rangle_\mfh
&=&
\int_0^1 \bigg\{ \sqrt{n}\sum_{j=1}^n2^{-1} D_ta_\tjm \>q_j+\sqrt{n}\sum_{j=1}^na_\tjm(B_\tj-B_\tjm){\bf 1}_{I_j}(t)\bigg\}
\nn\\&&
\quad\times \sqrt{n}\sum_{k=1}^na_\tkm(B_t-B_\tkm){\bf 1}_{I_k}(t)dt
\eea
and
\bea\label{20161117-2}
D_s\langle DZ_n,u_n\rangle_\mfh
&=&
\int_0^1 \bigg\{ \sqrt{n}\sum_{j=1}^n2^{-1} D_sD_ta_\tjm \>q_j
+\sqrt{n}\sum_{j=1}^nD_ta_\tjm(B_\tj-B_\tjm){\bf 1}_{I_j}(s)\bigg\}
\nn\\&&
\quad+\sqrt{n}\sum_{j=1}^nD_sa_\tjm(B_\tj-B_\tjm){\bf 1}_{I_j}(t)
+\sqrt{n}\sum_{j=1}^na_\tjm{\bf 1}_{I_j}(s){\bf 1}_{I_j}(t)
\bigg\}
\nn\\&&
\quad\times \sqrt{n}\sum_{k=1}^na_\tkm(B_t-B_\tkm){\bf 1}_{I_k}(t)dt
\nn\\&&
+\int_0^1 \bigg\{ \sqrt{n}\sum_{j=1}^n2^{-1} D_ta_\tjm \>q_j+\sqrt{n}\sum_{j=1}^na_\tjm(B_\tj-B_\tjm){\bf 1}_{I_j}(t)\bigg\}
\nn\\&&
\quad\times\bigg\{\sqrt{n}\sum_{k=1}^nD_sa_\tkm(B_t-B_\tkm){\bf 1}_{I_k}(t)
+\sqrt{n}\sum_{k=1}^na_\tkm{\bf 1}_{[\tkm,t)}(s){\bf 1}_{I_k}(t)\bigg\}dt.
\eea

It is known that in this example,
$
G_\infty = \half\int_0^1a_t^2dt
$
and  $r_n =n^{-1/2}$. Then,
\beas 
D_sG_\infty &=& \int_0^1(D_sa_t)a_tdt.
\eeas

\subsection{Quasi torsion}
We shall  {\color {black} study the asymptotic behavior} of the eight terms appearing in the expression 
of $\langle\langle DZ_n,u_n\rangle_\mfh,u_n\rangle_\mfh$ corresponding to (\ref{20161117-2}). 

\bd\im[(i)]  The first term is
\beas 
\cali_1 &=& 
\int_0^1\int_0^1\sqrt{n}\sum_{j=1}^nD_sD_ta_\tjm q_j\times\sqrt{n}\sum_{k=1}^n
a_\tkm(B_t-B_\tkm){\bf 1}_{I_k}(t)dt
\\&&
\times \sqrt{n}\sum_{\ell=1}^na_\tlm(B_s-B_\tlm){\bf 1}_{I_\ell}(s)ds.
\eeas
\ed
We investigate the rate of $E[\Psi(\sfz,\sfx)\cali_1]$. 
The factor $n^{1.5}$ comes from three $\sqrt{n}$. 
It suffices to consider the terms for which $k\vee\ell<j$; otherwise the term vanishes due to $D_sD_ta_\tjm$. 
The number of terms in the sum $\sum_j$ is of order $n^1$. 
The number of terms in the sum $\sum_{k.\ell}$ for $k=\ell$ and $k\vee\ell<j$ is $O(n^1)$, and 
each $B_t-B_\tkm$ ($=B_s-B_\tlm$) or its $\mfh$-derivative contributes  {\color {black} $O(n^{-0.5})$} in $L^p$-norm. 
By the IBP formula for $q_j$ we
get a factor $n^{-2}$.
So that the partial sum in $E[\Psi(\sfz,\sfx)\cali_1]$ for $k=\ell$ is $O(n^{-1.5})$ since 
both $ds$ and $dt$-integrals give $O(n^{-1})$. 
For the partial sum in $E[\Psi(\sfz,\sfx)\cali_1]$ for $k\not=\ell$ is also $O(n^{-1.5})$, 
since the consecutive IBP formulas (i.e., duality) for $B_t-B_\tkm$ and $B_s-B_\tlm$ gives the rate $O(n^{-2})$. 
Thus, we obtain $E[\Psi(\sfz,\sfx)\cali_1]=O(n^{-1.5})$, or $\sqrt{n}E[\Psi(\sfz,\sfx)\cali_1]=O(n^{-1})$. 
This means $\sqrt{n}E[\Psi(\sfz,\sfx)\cali_1]$ is negligible in the expansion. 
Table \ref{tablei1} summarizes how the orders of the partial sums were obtained. 
\begin{table}[htb]
  \begin{center}
  \caption{$E[\Psi(\sfz,\sfx)\cali_1]$}
\begin{tabular}{|c|c|c|c|c|c|c|c||c|} \hline
    $n$ & $\sum_j(k\vee\ell<j)$ & $\sum_{k,\ell}(k=\ell)$ & $ds$ &$dt$&IBP by $q_j$&$B_t-B_\tkm$&$B_s-B_\tlm$&order\\ \hline
    $1.5$  & $1$ & $1$ & $-1$&$-1$&$-2$ &$-0.5$&$-0.5$&$-1.5$\\ \cline{3-9}
     &    \\ \cline{3-9}
     &  & $\sum_{k,\ell}(k\not=\ell)$ & $ds$ &$dt$&IBP by $q_j$&IBP by $B_t-B_\tkm$&IBP by $B_s-B_\tlm$&order \\ \cline{3-9}
      &  & $2$ & $-1$ & $-1$ &$-2$ &$-1$&$-1$&$-1.5$ \\  \hline
  \end{tabular}  \label{tablei1}\end{center}
\end{table}

\bd\im[(ii)]  The second term can be written as
\beas 
\cali_2 &=& 
\int_0^1\int_0^1\sqrt{n}\sum_{j=1}^nD_ta_\tjm (B_\tj-B_\tjm){\bf 1}_{I_j}(s)\times\sqrt{n}\sum_{k=1}^n
a_\tkm(B_t-B_\tkm){\bf 1}_{I_k}(t)dt
\\&&
\times \sqrt{n}\sum_{\ell=1}^na_\tlm(B_s-B_\tlm){\bf 1}_{I_\ell}(s)ds. 
\eeas
\ed
Only terms with $k<j=\ell$  remain due to the product 
${\bf 1}_{I_k}(t)D_ta_\tjm{\bf 1}_{I_j}(s){\bf 1}_{I_\ell}(s)$. 
Table \ref{tablei2} shows 
%the sum of the terms with $j=\ell\not=k$ can remain in the expansion. 
$E[\Psi(\sfz,\sfx)\cali_2]=O(n^{-0.5})$, as explained more precisely below. 
\begin{table}[htb]
  \begin{center}
  \caption{$E[\Psi(\sfz,\sfx)\cali_2]$}
\begin{tabular}{|c|c|c|c|c|c|c||c|} \hline
    $n$ & $\sum_{j,\ell}(j=\ell)$ & $\sum_k(k<\ell)$ & $ds$ &$dt$&IBP by $B_t-B_\tkm$&$(B_s-B_\tlm)(B_\tj-B_\tjm)$&order\\ \hline
    $1.5$  & $1$ & $1$ & $-1$&$-1$&$-1$ &$-1$&$-0.5$\\  \hline
  \end{tabular}  \label{tablei2}\end{center}
\end{table}
\begin{en-text}
\begin{table}[htb]
  \begin{center}
  \caption{$E[\Psi(\sfz,\sfx)\cali_2]$}
\begin{tabular}{|c|c|c|c|c|c|c||c|} \hline
    $n$ & $\sum_{j,\ell}(j=\ell)$ & $\sum_k(k\not=\ell)$ & $ds$ &$dt$&IBP by $B_t-B_\tkm$&$(B_s-B_\tlm)(B_\tj-B_\tjm)$&order\\ \hline
    $1.5$  & $1$ & $1$ & $-1$&$-1$&$-1$ &$-1$&$-0.5$\\ \cline{3-8}
     &    \\ \cline{3-8}
     &  & $\sum_k(k=\ell)$ & $ds$ &$dt$&$B_t-B_\tkm$&$(B_s-B_\tlm)(B_\tj-B_\tjm)$&order \\ \cline{3-8}
      &  & $0$ & $-1$ & $-1$ &$-0.5$ &$-1$&$-1$ \\  \hline
  \end{tabular}  \label{tablei2}\end{center}
\end{table}
\end{en-text}

The contribution of $\sqrt{n}E[\Psi(\sfz,\sfx)\cali_2]$ is evaluated as follows. $A_n\equiv^a B_n$ means $A_n-B_n=o(1)$ as $n\to\infty$. 
By It\^o's formula, 
\bea\label{20161123-1} 
(B_\tj-B_\tjm)(B_s-B_\tjm) &=& 
(B_\tj-B_s)(B_s-B_\tjm)+2\int_\tjm^s\int_\tjm^tdB_rdB_t+(s-\tjm)
\eea
for $s\in I_j$. 
As already mentioned, only the terms with $k<j=\ell$ contribute the result. 
Applying the IBP formula for the first two terms of the right-hand side of (\ref{20161123-1}), we obtain 
\beas 
\sqrt{n}E[\Psi(\sfz,\sfx)\cali_2] 
&\equiv^a&
E\bigg[\Psi(\sfz,\sfx)
\int_0^1\int_0^1n^2\sum_{j=1}^na_\tjm D_ta_\tjm (B_\tj-B_\tjm)(B_s-B_\tjm){\bf 1}_{I_j}(s)
\\&&
\times\sum_{k:k<j}
a_\tkm(B_t-B_\tkm){\bf 1}_{I_k}(t)dtds\bigg]
\\&\equiv^a&
E\bigg[\Psi(\sfz,\sfx)
\int_0^1\int_0^1n^2\sum_{j=1}^na_\tjm D_ta_\tjm (s-\tjm){\bf 1}_{I_j}(s)
\\&&
\times\sum_{k:k<j}
a_\tkm(B_t-B_\tkm){\bf 1}_{I_k}(t)dtds\bigg]. 
\eeas
The IBP formula  for $B_t-B_\tkm=\delta({\bf 1}_{[\tkm,t]})$ yields 
\beas 
\sqrt{n}E[\Psi(\sfz,\sfx)\cali_2] 
&\equiv^a&
\sum_j\sum_{k:k<j}\int_{s\in I_j}\int_{t\in I_k}n^2 (s-\tjm)
E\bigg[\int_{r\in[\tkm,t]} D_r\big\{
\Psi(\sfz,\sfx)(D_ta_\tjm) a_\tjm a_\tkm\big\}dr\bigg]dtds
\\&\equiv^a&
\int_{0}^1\half\int_{0}^s\half
E\bigg[D_t\big\{\Psi(\sfz,\sfx)(D_ta_s) a_s a_t\big\}\bigg]dtds, 
\eeas
where 
\beas 
D_t\big\{\Psi(\sfz,\sfx)(D_ta_s) a_s a_t\big\} &=& 
\lim_{r\up t}D_r\big\{\Psi(\sfz,\sfx)(D_ta_s) a_s a_t\big\}. 
\eeas
Therefore, 
\bea\label{20161124-1}
\sqrt{n}E[\Psi(\sfz,\sfx)\cali_2] 
&\equiv^a&
%\frac{1}{4}\int_{s=0}^1\int_{t=0}^sE\bigg[D_t\big\{\Psi(\sfz,\sfx)(D_ta_s) a_s a_t\big\}\bigg]dtds
%\\&\equiv^a&
\frac{1}{4}\int_{0}^1\int_{0}^s
E\bigg[(D_t\Psi(\sfz,\sfx))(D_ta_s) a_s a_t
+\Psi(\sfz,\sfx) (D_tD_ta_s) a_s a_t
\nn\\&&
+\Psi(\sfz,\sfx)(D_ta_s)^2a_t
\bigg]dtds.
\eea

\bd\im[(iii)]  The third term is given by
\beas 
\cali_3 &=& 
\int_0^1\int_0^1\sqrt{n}\sum_{j=1}^nD_sa_\tjm (B_\tj-B_\tjm){\bf 1}_{I_j}(t)\times\sqrt{n}\sum_{k=1}^n
a_\tkm(B_t-B_\tkm){\bf 1}_{I_k}(t)dt
\\&&
\times \sqrt{n}\sum_{\ell=1}^na_\tlm(B_s-B_\tlm){\bf 1}_{I_\ell}(s)ds.
\eeas
\ed
By symmetry, it is easy to see 
$
\sqrt{n}E[\Psi(\sfz,\sfx)\cali_3] = \sqrt{n}E[\Psi(\sfz,\sfx)\cali_2],
$
and hence the limit is the same as (\ref{20161124-1}).

\bd\im[(iv)]   Consider now the fourth term given by
\beas 
\cali_4 &=& 
\int_0^1\int_0^1\sqrt{n}\sum_{j=1}^na_\tjm {\bf 1}_{I_j}(s){\bf 1}_{I_j}(t)\times\sqrt{n}\sum_{k=1}^na_\tkm(B_t-B_\tkm){\bf 1}_{I_k}(t)dt
\\&&
\times \sqrt{n}\sum_{\ell=1}^na_\tlm(B_s-B_\tlm){\bf 1}_{I_\ell}(s)ds.
\eeas
\ed
Table \ref{tablei4} suggests that $\sqrt{n}E[\Psi(\sfz,\sfx)\cali_4]$ remains.  Indeed,
\begin{table}[htb]
  \begin{center}
  \caption{$E[\Psi(\sfz,\sfx)\cali_4]$}
\begin{tabular}{|c|c|c|c|c||c|} \hline
    $n$ & $\sum_{j,k,\ell}(j=k=\ell)$ &  $ds$ &$dt$&$(B_t-B_\tkm)(B_s-B_\tlm)$&order\\ \hline
    $1.5$  & $1$ &  $-1$&$-1$&$-1$ &$-0.5$
          \\  \hline
  \end{tabular}  \label{tablei4}\end{center}
\end{table}

The contribution of this term is given by
\beas 
\sqrt{n}E[\Psi(\sfz,\sfx)\cali_4] 
&\equiv^a& 
E\bigg[\Psi(\sfz,\sfx)\int_0^1\int_0^1n^2\sum_{j=1}^na_\tjm^3\big((t\wedge s)-\tjm\big){\bf 1}_{I_j}(s){\bf 1}_{I_j}(t)dtds\bigg]
\\&\equiv^a& 
\frac{1}{3}\int_0^1E[\Psi(\sfz,\sfx) a_t^3]dt. 
\eeas

\bd\im[(v)]  The fifth term is
\beas 
\cali_5 &=& 
\int_0^1\int_0^1\sqrt{n}\sum_{j=1}^n2^{-1}D_ta_\tjm q_j
\times\sqrt{n}\sum_{k=1}^nD_sa_\tkm(B_t-B_\tkm){\bf 1}_{I_k}(t)dt
\\&&
\times \sqrt{n}\sum_{\ell=1}^na_\tlm(B_s-B_\tlm){\bf 1}_{I_\ell}(s)ds.
\eeas
\ed
Notice that $\cali_5$ looks like $\cali_1$ but they are slightly different from each other. 
Only the terms satisfying $\ell<k<j$ remain due to $D_ta_\tjm$ and $D_sa_\tkm$. 
\begin{table}[htb]
  \begin{center}
  \caption{$E[\Psi(\sfz,\sfx)\cali_5]$}
\begin{tabular}{|c|c|c|c|c|c|c||c|} \hline
    $n$ & $\sum_{j,k,\ell}(\ell<k<j)$ &  $ds$ &$dt$&IBP by $q_j$&IBP by $B_t-B_\tkm$&IBP by $B_s-B_\tlm$&order\\ \hline
    $1.5$  & $3$ &  $-1$&$-1$&$-2$ &$-1$&$-1$&$-1.5$
          \\  \hline
  \end{tabular}  \label{tablei5}\end{center}
\end{table}
According to Table \ref{tablei5}, we see 
$\sqrt{n}E[\Psi(\sfz,\sfx)\cali_5]=O(n^{-1})$ and is negligible. 

\bd\im[(vi)]  The sixth term is given by
\beas 
\cali_5 &=& 
\int_0^1\int_0^1\sqrt{n}\sum_{j=1}^n2^{-1}D_ta_\tjm q_j
\times\sqrt{n}\sum_{k=1}^na_\tkm {\bf 1}_{[\tkm,t]}(s){\bf 1}_{I_k}(t)dt
\\&&
\times \sqrt{n}\sum_{\ell=1}^na_\tlm(B_s-B_\tlm){\bf 1}_{I_\ell}(s)ds.
\eeas
\ed
Thanks to the product ${\bf 1}_{[\tkm,t]}(s){\bf 1}_{I_k}(t){\bf 1}_{I_\ell}(s)D_ta_\tjm$, only  the terms satisfying $k=\ell<j$  remain.  By table \ref{tablei6} 
below   $\sqrt{n}E[\Psi(\sfz,\sfx)\cali_6]=O(n^{-1})$ and this term  is negligible. 
\begin{table}[htb]
  \begin{center}
  \caption{$E[\Psi(\sfz,\sfx)\cali_6]$}
\begin{tabular}{|c|c|c|c|c|c|c||c|} \hline
    $n$ & $\sum_j$&$\sum_{k,\ell}(\ell=k<j)$ &  $ds$ &$dt$&IBP by $q_j$&IBP by $B_s-B_\tlm$&order\\ \hline
    $1.5$  & $1$ &  $1$&$-1$&$-1$ &$-2$&$-1$&$-1.5$
          \\  \hline
  \end{tabular}  \label{tablei6}\end{center}
\end{table}
\bd\im[(vii)]  Consider the seventh term given by
\beas 
\cali_7 &=& 
\int_0^1\int_0^1\sqrt{n}\sum_{j=1}^na_\tjm (B_\tj-B_\tjm){\bf 1}_{I_j}(t)
\times\sqrt{n}\sum_{k=1}^nD_sa_\tkm(B_t-B_\tkm){\bf 1}_{I_k}(t)dt
\\&&
\times \sqrt{n}\sum_{\ell=1}^na_\tlm(B_s-B_\tlm){\bf 1}_{I_\ell}(s)ds.
\eeas
\ed
Due to the product ${\bf 1}_{I_\ell}(s)D_sa_\tkm{\bf 1}_{I_j}(t){\bf 1}_{I_k}(t)$, 
only the terms satisfying $\ell<k=j$ contribute to the sum. 
Then it turns out that $\cali_7$ is the same as $\cali_2$. 
Therefore $\sqrt{n}E[\Psi(\sfz,\sfx)\cali_7]=\sqrt{n}E[\Psi(\sfz,\sfx)\cali_2]$ and 
the limit is given by (\ref{20161124-1}).

\bd\im[(viii)]  Finally, the last term is
\beas 
\cali_8 &=& 
\int_0^1\int_0^1\sqrt{n}\sum_{j=1}^na_\tjm (B_\tj-B_\tjm){\bf 1}_{I_j}(t)\times\sqrt{n}\sum_{k=1}^na_\tkm{\bf 1}_{[\tkm,t]}(s){\bf 1}_{I_k}(t)dt
\\&&
\times \sqrt{n}\sum_{\ell=1}^na_\tlm(B_s-B_\tlm){\bf 1}_{I_\ell}(s)ds.
\eeas
\ed
{\color {black} It suffices to  consider } the case $j=k=\ell$. 
Table \ref{tablei4} says that $\sqrt{n}E[\Psi(\sfz,\sfx)\cali_8]$  {\color {black} contributes to the limit.}
\begin{table}[htb]
  \begin{center}
  \caption{$E[\Psi(\sfz,\sfx)\cali_8]$}
\begin{tabular}{|c|c|c|c|c||c|} \hline
    $n$ & $\sum_{j,k,\ell}(j=k=\ell)$ &  $ds$ &$dt$&$(B_\tj-B_\tjm)(B_s-B_\tlm)$&order\\ \hline
    $1.5$  & $1$ &  $-1$&$-1$&$-1$ &$-0.5$
          \\  \hline
  \end{tabular}  \label{tablei8}\end{center}
\end{table}

More precisely, following a procedure  quite similar to that of $\cali_4$,  we obtain
\beas 
\sqrt{n}E[\Psi(\sfz,\sfx)\cali_8] 
&\equiv^a& 
E\bigg[\Psi(\sfz,\sfx)\int_0^1\int_0^1n^2\sum_{j=1}^na_\tjm^3\big(s-\tjm\big){\bf 1}_{[\tjm,t]}(s){\bf 1}_{I_j}(t)dtds\bigg]
\\&\equiv^a& 
\frac{1}{6}\int_0^1 E\big[\Psi(\sfz,\sfx)a_t^3\big]dt. 
\eeas

Now, from (i)-(viii) and 
\beas 
\Psi(\sfz,\sfx) &=& \exp\bigg(\frac{1}{4}\int_0^1a_s^2ds(\tti\sfz)^2\bigg),
\eeas
we obtain 
\beas &&
E\big[\Psi(\sfz,\sfx){\mathfrak S}^{(3,0)}{\colorr (\tti\sfz,\tti\sfx)}\big] 
\\&=& 
\lim_{n\to\infty}
\frac{\sqrt{n}}{3}E\bigg[\Psi(\sfz,\sfx)\bigg\langle D\big\langle DM_n,u_n\big\rangle_\mfh,u_n
\bigg\rangle_\mfh\psi_n(\tti\sfz)^3\bigg]
\\&=&
\frac{1}{3}\bigg\{
\frac{3}{4}\int_{s=0}^1\int_{t=0}^s
E\bigg[(D_t\Psi(\sfz,\sfx))(D_ta_s) a_s a_t
+\Psi(\sfz,\sfx) (D_tD_ta_s) a_s a_t
+\Psi(\sfz,\sfx)(D_ta_s)^2a_t
\bigg]dtds(\tti\sfz)^3
\\&&
+\bigg(\frac{1}{3}+\frac{1}{6}\bigg)\int_0^1E[\Psi(\sfz,\sfx) a_t^3]dt(\tti\sfz)^3\bigg\}
\\&=&
\frac{1}{8}E\bigg[\Psi(\sfz,\sfx)\int_0^1 a_t \bigg(\int_t^1(D_ta_s)a_sds\bigg)^2dt\bigg](\tti\sfz)^5
\\&&
+\frac{1}{4}E\bigg[\Psi(\sfz,\sfx)\int_0^1a_t\int_t^1\big\{(D_tD_ta_s)a_s+(D_ta_s)^2\big\}dsdt\bigg](\tti\sfz)^3
\\&&
+\frac{1}{6}E\bigg[\Psi(\sfz,\sfx)\int_0^1a_t^3dt\bigg](\tti\sfz)^3. 
\eeas
It should be remarked that the three terms on the right-hand side of the above equation 
correspond to $\calc_2$, $\calc_3$ and $\calc_1$ of \cite{yoshida2013martingale}, pp. 917-918, respectively. 
We remark that two random symbols with the same adjoint action are considered  equivalent. 

Obviously ${\mathfrak S}^{(2,0)}={\mathfrak S}^{(1,1)}={\mathfrak S}^{(1,0)}={\mathfrak S}^{(0,1)}=0$ 
in the present situation, and moreover we will {\color {black} show that} ${\mathfrak S}^{(2,0)}_0=0$ in Section \ref{20161124-2}. 
Consequently, 
\beas 
{\mathfrak S}^{(3,0)}(\tti\sfz,\tti\sfx)
&=& 
\frac{1}{8}\int_0^1 a_t \bigg(\int_t^1(D_ta_s)a_sds\bigg)^2dt(\tti\sfz)^5
\\&&
+\frac{1}{4}\int_0^1a_t\int_t^1\big\{(D_tD_ta_s)a_s+(D_ta_s)^2\big\}dsdt(\tti\sfz)^3
%\\&&
+\frac{1}{6}\int_0^1a_t^3dt(\tti\sfz)^3
\eeas
and 
\beas 
r_n^{-1}\big({\mathfrak S}_n(\tti\sfz,\tti\sfx)-1\big) &=& {\mathfrak S}^{(3,0)}(\tti\sfz,\tti\sfx).
\eeas
This random symbol is equivalent to the full random symbol $\sigma(\tti\sfz,\tti\sfx)$ of 
\cite{yoshida2013martingale}, p. 918 with $a$ replaced by $a/2$. 
For the quadratic form of a Brownian motion, ${\mathfrak S}^{(3,0)}$ provides both the adapted random symbol and the anticipative random symbol, 
in other words, the quasi torsion includes the tangent as well as the torsion. 
In this way, we found that the quasi torsion reproduces the asymptotic expansion of 
 the quadratic form of a Brownian motion. 

\subsection{Quasi tangent}\label{20161124-2}
For the quasi tangent, we have 
\beas 
\langle DZ_n,u_n\rangle_\mfh-G_\infty
&=&
\int_0^1\bigg\{\sqrt{n}\sum_{j=1}^n2^{-1}D_ta_\tjm q_j+\sqrt{n}\sum_{j=1}^na_\tjm(B_\tj-B_\tjm){\bf 1}_{I_j}(t)\bigg\}
\\&&
\times\sqrt{n}\sum_{k=1}^na_\tkm(B_t-B_\tkm){\bf 1}_{I_k}(t)dt-\half\int_0^1a_t^2dt
\\&=&
{\mathfrak G}_1+{\mathfrak G}_2+{\mathfrak G}_3,
\eeas
where
\beas
{\mathfrak G}_1 
&=&
\int_0^1n\sum_{j=1}^n2^{-1}D_ta_\tjm q_j
\times\sum_{k=1}^na_\tkm(B_t-B_\tkm){\bf 1}_{I_k}(t)dt,
\eeas
\beas 
{\mathfrak G}_2 
&=&
\int_0^1\bigg\{n\sum_{j=1}^na_\tjm^2 (B_t-B_\tjm)^2{\bf 1}_{I_j}(t)-2^{-1}a_t^2\bigg\}dt,
\eeas
and 
\beas 
{\mathfrak G}_3
&=&
\int_0^1n\sum_{j=1}^na_\tjm^2 (B_\tj-B_t)(B_t-B_\tjm){\bf 1}_{I_j}(t)dt.
\eeas

We shall investigate these terms. 

\bd\im[(i)]  For ${\mathfrak G}_1$,
Table \ref{tableg1} shows $\sqrt{n}E[\Psi(\sfz,\sfx){\mathfrak G}_1]=O(n^{-0.5})$ and 
it is negligible in the asymptotic expansion.  \ed
\begin{table}[htb]
  \begin{center}
  \caption{$E[\Psi(\sfz,\sfx){\mathfrak G}_1]$}
\begin{tabular}{|c|c|c|c|c|c||c|} \hline
    $n$ & $\sum_j$ & $\sum_k(k<j)$ &$dt$&IBP by $q_j$&IBP by $B_t-B_\tkm$&order\\ \hline
    $1$  & $1$ & $1$ & $-1$&$-2$ &$-1$&$-1$\\  \hline
  \end{tabular}  \label{tableg1}\end{center}
\end{table}

\bd\im[(ii)] For ${\mathfrak G}_2$, applying It\^o's formula, we have ${\mathfrak G}_2={\mathfrak G}_2'+{\mathfrak G}_2''$, where  \ed
\beas 
{\mathfrak G}_2' 
&=&
\sum_{j=1}^n\int_{I_j} na_\tjm^2 
\int_\tjm^t\int_\tjm^s2dB_rdB_s\>dt
\eeas
and 
\beas 
{\mathfrak G}_2''
&=&
\half\sum_{j=1}^n\int_{I_j}\big(a_\tjm^2-a_t^2\big)
{\bf 1}_{I_j}(t)dt.
\eeas
Table \ref{tableg2prime} shows $\sqrt{n}E[\Psi(\sfz,\sfx){\mathfrak G}_2']=O(n^{-0.5})$ and 
it is negligible in the asymptotic expansion. 
We should remark that $\sqrt{n}{\mathfrak G}_2'$ is $\dotc_n$ of \cite{yoshida2013martingale} 
and it has non-trivial limit distribution though the expectation $\sqrt{n}E[\Psi(\sfz,\sfx){\mathfrak G}_2']$ asymptotically vanishes.  
We see ${\mathfrak G}_2''$ of of $O(n^{-1})$ in $L^2$, and it is also negligible. 
Consequently, ${\mathfrak G}_2$ is negligible in the asymptotic expansion. 
\begin{table}[htb]
  \begin{center}
  \caption{$E[\Psi(\sfz,\sfx){\mathfrak G}_2']$}
\begin{tabular}{|c|c|c|c||c|} \hline
    $n$ & $\sum_j$  &$dt$&IBP by $\int_\tjm^t\int_\tjm^sdB_rdB_s$&order\\ \hline
    $1$  & $1$ & $-1$ &$-2$&$-1$\\  \hline
  \end{tabular}  \label{tableg2prime}\end{center}
\end{table}
 \bd\im[(iii)] Finally, for ${\mathfrak G}_3$
Table \ref{tableg3} shows $\sqrt{n}E[\Psi(\sfz,\sfx){\mathfrak G}_3]=O(n^{-0.5})$ and we can neglect it.  \ed
\begin{table}[htb]
  \begin{center}
  \caption{$E[\Psi(\sfz,\sfx){\mathfrak G}_3]$}
\begin{tabular}{|c|c|c|c|c||c|} \hline
    $n$ & $\sum_j$  &$dt$&IBP by $B_\tj-B_t$&IBP by $B_t-B_\tkm$&order\\ \hline
    $1$  &  $1$ & $-1$&$-1$ &$-1$&$-1$\\  \hline
  \end{tabular}  \label{tableg3}\end{center}
\end{table}

As a consequence of these observations, 
${\mathfrak S}^{(2,0)}_0=0$, %This is what we called ``tangent'' in our discussion. 
%However, all the second-order terms appear in ${\mathfrak S}^{(3,0)}$.
i.e., the quasi tangent has no effect in the asymptotic expansion. 
However, the effect of the tangent already appeared in that of the quasi torsion.

%%%%%%%%%%%%%%%%%%%%%%%%%%%%%%%
 \section{Quadratic form of a fractional Brownian motion with random weights}\label{170813-4}
 % (2017.04.20 Later make sections ``q-Torsion'' and ``q-Tangent'')}}
\subsection{Weighted quadratic variation}
 Let $B=\{B_t, t\in [0,1]\}$ be a fractional Brownian motion with Hurst parameter $H\in (\frac 14, \frac 34)$. We are interested in the following sequence of weighted quadratic variations:
 \[
 Z_n= n^{2H-\frac 12} \sum_{j=1}^n a_{t_{j-1}}  ((\Delta  B_{j,n})^2- n^{-2H}),
 \]
 where $t_j= j/n$, $a_t =a(B_t)$ and $a$ is a function such that $a$ and all its derivatives up to some order $N$ have moderate growth. We use the notation  $\Delta B_{j,n} = B_{j/n} - B_{(j-1)/n}$.
 It is known (see, for instance \cite{nourdin2010central,nourdin2016quantitative}) that for this example the limit variance $G_\infty$ is given by
 \[
 G_\infty= 2c^2_H \int_0^1 a(B_s)^2 ds,
 \]
 where 
 \bea
 c^2_H&=& \sum_{k=-\infty} ^\infty \rho_H(k)^2 \label{ch}
 \eea
 {\colorg with 
 \beas 
 \rho_H(k) &=& 
 \half\big(|k+1|^{2H}-2|k|^{2H}+|k-1|^{2H}\big).
 \eeas
 }
 
 Set $I_j= [t_{j-1}, t_j)$.  Then,
 \[
 a_{t_{j-1}} (\Delta_{j,n} B)^2= \delta(  a_{t_{j-1}}\Delta B_{j,n} \mathbf{1}_{ I_j}) +  a_{t_{j-1}}n^{-2H} + \Delta B_{j,n} \langle D a_{t_{j-1}}, \mathbf{1}_{ I_j} \rangle_{\HH}.
 \]
 Therefore, we obtain the decomposition
 \[
 Z_n = \delta(u_n) + W_n:= M_n +r_nN_n,
 \]
 where
 \[
 u_n(t)=   n^{2H-\frac 12} \sum_{j=1}^n  a_{t_{j-1}}  \Delta B_{j,n} \mathbf{1}_{ I_j}(t)
 \]
 and
 \[
 r_nN_n=    n^{2H-\frac 12} \sum_{j=1}^n  \Delta B_{j,n} \langle D a_{t_{j-1}}, \mathbf{1}_{ I_j}\rangle_{\HH}.
 \]
{\colorb 
       Set
      \bea \label{yu3}
      \Psi(\sfz)= \exp\left(  -\sfz^2c_H^2 \int_0^1 a^2(B_s) ds \right).
      \eea
 }\noindent
In this example, we take $W_n=W_\infty=0$, $X_n=X_\infty=0$ and $\psi=1$.  We are going to study the quasi torsion and the {\colred quasi} tangent of the Skorohod integral $M_n= \delta(u_n)$.  In this example there will be also a contribution to the asymptotic expansion coming from the perturbation term $N_n$. The scaling factor $r_n$ will be taken as $r_n=n^{2H-\frac{3}{2}}$ when $H\in(\half,\frac{3}{4})$, and 
$r_n=n^{\half-2H}$ when $H\in\left(\frac{1}{4},\half\right)$, respectively.  This choice of $r_n$ is motivated by  the rate of convergence
\beas
|E[\varphi(Z_n)- E[\varphi(\zeta G_\infty)]| &\le& C_{f,H} \max_{1\le i \le 5} \| \varphi^{(i)} \|_\infty   r_n^{-1}
\eeas
 obtained in \cite{nourdin2016quantitative} for $\varphi \in  C^5_b(\mathbb{R})$, where $\zeta$ is $N(0,1)$.  
 
 \medskip
 \noindent
\subsection{Quasi torsion}
We recall that
\beas
{\sf qTor} &=& r_n^{-1}\langle D   \langle DM_n, u_n \rangle_{\HH}, u_n \rangle_{\HH}.
\eeas
 Set $q_j= (\Delta B_{j,n})^2- n^{-2H} = I_2(  \mathbf{1}_{ I_j} ^{\otimes 2})$. We have
 \begin{eqnarray*}
 \langle DM_n, u_n \rangle_{\HH} &=& n^{4H-1} \Bigg  \langle    \sum_{j=1}^n  \left[( (D a_{t_{j-1}} )q_j +   2  a_{t_{j-1}}  \Delta B_{j,n}   \mathbf{1}_{ I_j}  \right],
  \sum_{k=1}^n  a_{t_{k-1}} \Delta  B_{k,n}   \mathbf{1}_{ I_k}  \Bigg  \rangle_{\HH}- r_n \langle DN_n, u_n \rangle_{\HH} \\
  &=& \Phi_{n,1} - r_n  \langle DN_n, u_n \rangle_{\HH}.
 \end{eqnarray*}

\medskip
\noindent
{\bf (A)}  We first study the contribution of the term $\langle D \Phi_{n,1} , u_n \rangle_{\HH}$ in the asymptotic expansion. We have
  \begin{eqnarray*}
 && 
  \langle D \Phi_{n,1} , u_n \rangle_{\HH}%\langle D\langle D \Phi_{n,1} , u_n \rangle_{\HH} , u_n \rangle_{\HH} 
 = n^{6H-\frac 32}  \sum_{j,k,\ell=1}^n  \Big[ (D_rD_s  a_{t_{j-1}}) q_j +   2 (D_s  a_{t_{j-1}})  \Delta B_{j,n}  \mathbf{1}_{ I_j}(r)\\
 &&
 \qquad+2 (D_r  a_{t_{j-1}}) \Delta B_{j,n}   \mathbf{1}_{ I_j} (s) + 2 a_{t_{j-1}}  \mathbf{1}_{ I_j} (s)  \mathbf{1}_{ I_j} (r)    \Big] *
 [   a_{t_{k-1}} \Delta  B_{k,n}   \mathbf{1}_{ I_k} (s) ]*[a_{t_{\ell-1}} \Delta  B_{\ell,n}   \mathbf{1}_{ I_\ell} (r) ]  \\
   &&\qquad + n^{6H-\frac 32}  \sum_{j,k,\ell=1}^n   \Big[ (D_s  a_{t_{j-1}})q_j +   2  a_{t_{j-1}}  \Delta B_{j,n}   \mathbf{1}_{ I_j} (s)  \Big]* \Big[ (D_r a_{t_{k-1}}) \Delta B_{k,n}  \mathbf{1}_{ I_k} (s) +
    a_{t_{k-1}}   \mathbf{1}_{ I_k} (r) \mathbf{1}_{ I_k} (s) \Big] \\
 &&\qquad \qquad   *[ a_{t_{\ell-1}} \Delta  B_{\ell,n}   \mathbf{1}_{ I_\ell} (r)] ,
  \end{eqnarray*}
 where the product  $A*B$ means that whenever we found repeated variables in $A$ and $B$, we compute the corresponding inner product in $\HH$. We have a total of eight terms, that we denote by
 \beas
 \langle D \Phi_{n,1} , u_n \rangle_{\HH}&=& \sum_{i=1}^8 \mathcal{I}_i.
 \eeas
 We are interested in the  asymptotic behavior of $r_n^{-1}E[ \Psi(\sfz) \mathcal{I}_i]$ for $i=1,\dots, 8$.
 
 \medskip
 \noindent
 {\bf (i)} The first term is
 \[
 \mathcal{I}_1=n^{6H-\frac 32}  \sum_{j,k,\ell=1}^n a''_{t_{j-1}} a_{t_{k-1}} a_{t_{\ell-1}}  q_j   \Delta  B_{k,n} \Delta  B_{\ell,n}  \alpha_{t_{j-1},k} \alpha_{t_{j-1},\ell},
 \]
 where we have used the notation $\alpha_{t,k} = \langle  \mathbf{1}_{[0,t]}, \mathbf{1}_{I_k} \rangle_{\HH}$.  
 We can make the decomposition
 \[
 \Delta  B_{k,n} \Delta  B_{\ell,n}=  I_2(\mathbf{1}_{I_k} \otimes \mathbf{1}_{I_\ell}) + \beta_{k,\ell},
 \]
 where  $\beta_{k,\ell} =\langle \mathbf{1}_{I_k}, \mathbf{1}_{I_\ell} \rangle_{\HH}$. Integrating by parts shows that the contribution of
 $ I_2(\mathbf{1}_{I_k} \otimes \mathbf{1}_{I_\ell}) $ is of order lower than that of $ \beta_{k,\ell}$. In this way, it suffices to consider the term
  \[
 \mathcal{I}_{1,0}=n^{6H-\frac 32}  \sum_{j,k,\ell=1}^n a''_{t_{j-1}} a_{t_{k-1}} a_{t_{\ell-1}}  q_j   \beta_{k,\ell}  \alpha_{t_{j-1},k} \alpha_{t_{j-1},\ell}{\colred .}
 \]
  In this case, the factors of the above expressions have  the following orders of convergence:
 \begin{itemize}
 \item First factor: $n^{6H-\frac 32} $
 \item  The  IPB formula  for $q_j$ produces a factor $n^{-2(2H\wedge 1)}$, due to part (a) of Lemma \ref{lem1} below.
 \item   The terms   $| \alpha_{t_{j-1},k} \alpha_{t_{j-1},\ell}|$ are bounded by $C_Hn ^{-2(2H\wedge 1)}$ by part (a) of Lemma \ref{lem1} below.
 \item  $\sum_{k,\ell=1}^n  | \beta_{k,\ell}|  \le C_H n^{(1-2H) \vee 0}$ due to part (c)  of Lemma \ref{lem1}.
 \item Finally we get a factor $n$ from the sum in $j$.
 \end{itemize}
 Therefore, the order of this term is $n^{6H -\frac 12 -4(2H\wedge 1) +(1-2H) \vee 0}$. For $H>\frac 12$ this gives $n^{6H-\frac 92}$  and for
 $H<\frac 12$, we obtain the order $n^{\frac 12-4H}$. In both cases, when $H\to \frac 12$, we obtain $n^{-1.5}$ as in the Brownian motion case. 
 We remark that $6H-\frac{9}{2}<2H-\frac{3}{2}$ if $\frac 12 < H<\frac{3}{4}$, and 
$\half-4H<\half-2H$ if $\frac{1}{4} <H< \frac 12$. 
Therefore, this term will not contribute to the limit.
 
 \medskip
 \noindent
 {\bf  (ii)}  The second term is equal to
\[
 \mathcal{I}_2 = 2n^{6H-\frac 32}   \sum_{j,k,\ell=1}^n   a_{t_{k-1}} a'_{t_{j-1}}  a_{t_{\ell-1}}   \Delta B_{j,n}   \Delta B_{k,n}   \Delta B_{\ell,n}   \beta_{j,\ell} \alpha_{t_{j-1},k},
\]
where $\beta_{j,\ell} =\langle \mathbf{1}_{I_j}, \mathbf{1}_{I_\ell} \rangle_{\HH}$.
As before, we can replace the product $ \Delta B_{j,n} \Delta B_{\ell,n} $ by $\beta_{j,\ell}$, and we have to deal with the term
\[
 \mathcal{I}_{2,0} = 2n^{6H-\frac 32}   \sum_{j,k,\ell=1}^n   a_{t_{k-1}} a'_{t_{j-1}}  a_{t_{\ell-1}}     \Delta B_{k,n}      \beta_{j,\ell}^2 \alpha_{t_{j-1},k}.
\]

 We get the following contributions:
 \begin{itemize}
 \item First factor: $n^{6H-\frac 32} $
 \item  The  IPB formula  for $\Delta B_{k,n} $ produces a factor $n^{-(2H\wedge 1)}$, due to part (a) of Lemma \ref{lem1} below.
 \item  $\sum_{k=1}^n  |\alpha_{t_{j-1},k}| \le C_H$ due to part   (b)  of Lemma \ref{lem1}, and $\sum_{j,\ell=1}^n \beta_{j,\ell}^2\le C_Hn^{1-4H}$ by part (d) of Lemma \ref{lem1}.
 \end{itemize}
      Therefore, the order of this term is $n^{2H-\frac 12 -(2H\wedge 1)}= n^{-[(\frac 32-2H)   \wedge\half]}$, which in the Brownian case gives $n^{-\frac 12}$.
      From this result we deduce that this term will not contribute to the asymptotic expansion if $H<\frac 12$. 
     The contribution of  $n^{ \frac 32-2H} E[\Psi(\sfz) \mathcal{I}_2]$ when $H>\frac 12$ is evaluated as follows. 
     
Define a measure $\mu_n$ on $[0,1]^2$ by 
\beas 
\mu_n &=& \sum_{j,\ell=1}^nn^{-1+4H}\beta_{j,\ell}^2\>\delta_{(t_{j-1},t_{\ell-1})}.
\eeas
Then 
\bea  \label{yu2}
\int_{[0,1]^2}\varphi(t,s)\mu_n(dt,ds) &\to& c_H^2\int_{[0,1]}\varphi(t,t)dt
\eea
as $n\to\infty$ for every $\varphi\in C([0,1]^2)$, where 
$c_H^2$ is defined in (\ref{ch}). 
Notice that 
  \bea  \label{yu1}
\sup_{\tau\in [0,1], 1\le k\le n}  \left| \int_0^\tau \left(  n\int_{I_k} |r-s|^{2H-2} dr   - |t_{k-1} -s|^{2H-2}\right)ds \right|  \le C n^{1-2H} \rightarrow 0,
\eea
as $n$ tends to infinity. 
We can write  
\beas 
     n^{\frac 32-2H} E[\Psi(\sfz) \mathcal{I}_2]
     &=&
2  n^{4H}E\left[\Psi(\sfz)
        \sum_{j,k,\ell=1}^n   a_{t_{k-1}} a'_{t_{j-1}}  a_{t_{\ell-1}}   \beta_{j,\ell}^2   \Delta B_{k,n}     \alpha_{t_{j-1},k}\right]  +o(1)
\\&=&
2n\sum_{k=1}^n \int_{[0,1]^2}\mu_n(dt,dt')
E\big[\big\langle D\{\Psi(\sfz)a_{t_{k-1}}a'_ta_{t'}\},{\bf 1}_{I_k}\big\rangle_\HH\big]\alpha_{t,k}+o(1)
\\&=&
2\alpha_H^2n^{-1}\sum_{k=1}^n \int_{[0,1]^2}\mu_n(dt,dt')
E\bigg[\int_0^1\int_{t_{k-1}} ^{t_k} D_{s'}\{\Psi(\sfz)a_{t_{k-1}}a'_ta_{t'} \}|r-s'|^{2H-2}dr ds'\bigg]
\\&&
\times
\int_{[0,1]}{\bf 1}_{[0,t]}(s)|t_{k-1}-s|^{2H-2}ds+o(1)
\eeas
Using  (\ref{yu1}) and (\ref{yu2}), we obtain
\beas 
     n^{\frac 32-2H} E[\Psi(\sfz) \mathcal{I}_2]
     &=&
2\alpha_H^2 \int_{[0,1]^2}\mu_n(dt,dt')
\bigg\{\int_{[0,1]}dr 
E\bigg[\int_{[0,1]}D_{s'}\{\Psi(\sfz)a_ra'_ta_{t'}|r-s'|^{2H-2}ds'\bigg]
\\&&
\times
\int_{[0,1]}{\bf 1}_{[0,t]}(s)|r-s|^{2H-2}ds\bigg\}+o(1)
\\&\to&
2\alpha_H^2c_H^2 \int_{[0,1]}dt
\bigg\{\int_{[0,1]}dr 
E\bigg[\int_{[0,1]}D_{s'}\{\Psi(\sfz)a_ra'_ta_t|r-s'|^{2H-2}ds'\bigg]
\\&&
\times
\int_{[0,1]}{\bf 1}_{[0,t]}(s)|r-s|^{2H-2}ds\bigg\}
\\&=&
 C_{H,3}  \int_{[0,1]^4}   E\left[ D_{s'}  \left(  \Psi(\sfz) a(B_r) D_s (a^2(B_t)) \right) \right] |r-s|^{2H-2} |r-s'|^{2H-2}  dsds'dr dt, 
\eeas
where $\alpha_H = H(2H-1)$ and 
\beas
        C_{H,3} &=&   \alpha_H^2\sum_{j=-\infty} ^{+\infty}   \rho_H(j)^2 .
\eeas

%%
%\beas
%&= &   \sum_{j,k,\ell=1}^n   \rho_H(j-\ell)^2 E\left[ \langle D\left[ \Psi(\sfz)
%   a(B_{t_{k-1}})   \langle D[a^2(B_{t_{j-1}})], \mathbf{1}_{I_k} \rangle_{\HH}  \right]    , \mathbf{1}_{I_k}  \rangle_{\HH}\right]\\
%   &=&\alpha_H^2 \sum_{j,k,\ell=1}^n   \rho_H(j-\ell)^2
% \int_{[0,1]^2}   \int_{I_k} \int_{I_k}
%   E\left[ D_{s'} \left[ \Psi(\sfz)
%   a(B_{t_{k-1}})   D_s[a^2(B_{t_{j-1}})] \right] \right] \\
%   && \times |r-s|^{2H-2} |r'-s'|^{2H-2}  dsds'dr dr',
%\eeas

%     Since the series $\sum_{\ell =-\infty} ^\infty \rho_H(\ell)^2$ is convergent, it suffices to compute the  limit of
%     \[
%     \alpha_H^2 c_H^2 \sum_{j,k=1}^n    
% \int_{[0,1]^2}   \int_{I_k} \int_{I_k}
%   E\left[ D_{s'} \left[ \Psi(\sfz)
%   a(B_{t_{k-1}})   D_s[a^2(B_{t_{j-1}})] \right] \right] 
%  |r-s|^{2H-2} |r'-s'|^{2H-2}  dsds'dr dr',
%  \]
%  where we recall that $c_H^2= \sum_{\ell =-\infty} ^\infty \rho_H(\ell)^2$. 
% As a consequence, the above quantity is equivalent in the limit to
%\beas
%  \alpha_H^2 c_H^2 \sum_{j,k=1}^n    \frac 1{n^2}
% \int_{[0,1]^2}    
%   E\left[ D_{s'} \left[ \Psi(\sfz)
%   a(B_{t_{k-1}})   D_s[a^2(B_{t_{j-1}})] \right] \right] 
%  |t_{k-1}-s|^{2H-2} |t_{k-1}-s'|^{2H-2}  dsds',
%\eeas
%%

\begin{en-text}       
b) The case $H<\frac 12$ is more involved.  Taking into account the definition of $\psi(\sfz)$ given in (\ref{yu3}) and integrating by parts, we obtain
We have 
\beas 
n^{1/2}E[\Psi(\sfz)\cali_2]
&=&
2n^{6H-1}\sum_{j,k,\ell=1}^nE\big[\Psi(\sfz)a_{t_{k-1}}a'_{t_{j-1}}a_{t_{\ell-1}}\Delta B_{k,n}
\beta_{j,\ell}^2\alpha_{t_{j-1},k}\big]+o(1)
\\&=&
2n^{6H-1}\sum_{j,k,\ell=1}^nE\big[
\big\langle D\big\{\Psi(\sfz)a_{t_{k-1}}a'_{t_{j-1}}a_{t_{\ell-1}}\big\},{\bf 1}_{I_k}\big\rangle_\HH\big]
\beta_{j,\ell}^2\alpha_{t_{j-1},k}+o(1)
\\&=&
{\mathfrak G}_{n,1}+{\mathfrak G}_{n,2}+{\mathfrak G}_{n,3}+o(1),
\eeas
where
\beas 
{\mathfrak G}_{n,1} 
&=& 
-\sfz^2c_H^2n^{6H-1}\sum_{j,k,\ell=1}^nE\bigg[\Psi(\sfz)\int_0^1(a^2)'_s\alpha_{s,k}ds\  
a_{t_{k-1}}a'_{t_{j-1}}a_{t_{\ell-1}}\bigg]\beta_{j,\ell}^2\alpha_{t_{j-1},k},
\eeas
\beas 
{\mathfrak G}_{n,2} 
&=& 
2n^{6H-1}\sum_{j,k,\ell=1}^nE\bigg[\Psi(\sfz) 
a'_{t_{k-1}}a'_{t_{j-1}}a_{t_{\ell-1}}\bigg]\alpha_{t_{k-1},k}\beta_{j,\ell}^2\alpha_{t_{j-1},k}
\eeas
and 
\beas 
{\mathfrak G}_{n,3} 
&=& 
2n^{6H-1}\sum_{j,k,\ell=1}^nE\bigg[\Psi(\sfz) 
a_{t_{k-1}}\big(a''_{t_{j-1}}a_{t_{\ell-1}}\alpha_{t_{j-1},k}
+a'_{t_{j-1}}a'_{t_{\ell-1}}\alpha_{t_{\ell-1},k}\big)
\bigg]\beta_{j,\ell}^2\alpha_{t_{j-1},k}.
\eeas

We claim that ${\mathfrak G}_{n,1}\to0$, as $n\rightarrow \infty$. 
In fact,   the inequalities  $\max_{1\le k \le  n} \sup_{s\in [0,1]} | \alpha_{s,k} | \le C_H n^{-2H}$, by Lemma \ref{lem1} (a),  
  $\max{1\le j \le n}\sum_{k=1}^n |\alpha _{ t_{j-1},k}| \le C_H$ by Lemma \ref{lem1} (b) and
  $\sum_{j,\ell=1}^n\beta_{j,\ell}^2 \le C_H n^{1-4H}$  by Lemma \ref{lem1} (d),
  imply that  
\beas
|{\mathfrak G}_{n,1}|
&\simleq&
n^{6H-2}  n^{-2H} n^{1-4H} 
= n^{-1}.
\eeas

Moreover, ${\mathfrak G}_{n,3}\to0$. 
Indeed,
\beas
\max_{1\le j\le n} \sum_{k=1}^n \alpha_{t_{j-1}, k} ^2 
 &=&
\max_{1\le j \le n} \frac 14 n^{-4H}\sum_{k=1}^n   \left( k^{2H} -(k-1)^{2H} + |k-j|^{2H} - | k-j+1|^{2H} \right)^2
\\&\simleq&
n^{-4H}\sum_{k=1}^n k^{2(2H-1)}
\>=\> O(n^{-1})
\eeas
for $H>1/4$, and also the Schwarz inequality gives 
\beas
\max_{1\le j,\ell \le n }\sum_{k=1}^n\alpha_{t_{j-1},k}\alpha_{t_{\ell-1},k}=O(n^{-1})
\eeas
for $H>1/4$. Therefore, 
\beas 
|{\mathfrak G}_{n,3}|
&\simleq&
n^{6H-2}\sum_{j,\ell=1}^n\beta_{j,\ell}^2
\yeq O(n^{2H-1}) 
\eeas
due to Lemma \ref{lem1} (d). 

Since 
\beas 
\alpha_{t_{k-1},k} &=& 
\langle {\bf 1}_{[0,t_{k-1}]},{\bf 1}_{I_k}\rangle_\HH
\yeq E[B_{t_{k-1}}(B_{t_k}-B_{t_{k-1}})]
\\&=& 
\half n^{-2H}\big\{\big(k^{2H}-(k-1)^{2H}\big)-1\big\}, 
\eeas
the contribution of 
the term $k^{2H}-(k-1)^{2H}$ on the right-hand side 
is of order $k^{2H-1}$ when $H<1/2$. Thus we have, using Lemma \ref{lem1} (a), (b) and (d),
\beas &&
\bigg|n^{6H-1}\sum_{j,k,\ell=1}^n E\bigg[\Psi(\sfz) 
a'_{t_{k-1}}a'_{t_{j-1}}a_{t_{\ell-1}}\bigg]n^{-2H}\big(k^{2H}-(k-1)^{2H}\big)\beta_{j,\ell}^2\alpha_{t_{j-1},k}\bigg|
\\&\simleq&
n^{6H-1}\bigg[\sum_{k=1}^nn^{-4H}k^{2(2H-1)}\bigg]^\half
\max_{1\le j \le n}\bigg[\sum_{k=1}^n\alpha_{t_{j-1},k}^2\bigg]^\half\sum_{j,\ell=1}^n\beta_{j,\ell}^2
\\&=&
O(n^{6H-1}\times n^{-\half}\times n^{-H}\times n^{1-4H})
\yeq 
O(n^{H-\frac 12}).
\eeas
By using this and $\alpha_{t,k}=\int_{t_{k-1}}^{t_k}\frac{\partial R_H}{\partial r}(t,r)dr$, we obtain 
\beas 
\lim_{n\to\infty}{\mathfrak G}_{n,2}
&=& 
-\lim_{n\to\infty}\sum_{k=1}^n\int_{[0,1]^2}\mu_n(dt,dt')
E\big[\Psi(\sfz)a'_{t_{k-1}}a'_ta_{t'}\big] \int_{t_{k-1}}^{t_k}\frac{\partial R_H}{\partial r}(t,r)dr
\\&=& 
-\lim_{n\to\infty}\int_{[0,1]^2}\mu_n(dt,dt')
\int_0^1 
E\big[\Psi(\sfz)a'_ra'_ta_{t'}\big] \frac{\partial R_H}{\partial r}(t,r)dr
\\&=&
-c_H^2\int_{[0,1]^2}E\big[\Psi(\sfz)a'_ra'_ta_t\big] \frac{\partial R_H}{\partial r}(t,r)drdt. 
\eeas

After all that, we obtained 
{\color{black}{
\beas
   \lim_{n\rightarrow \infty}   \sqrt{n}  E[\Psi(\sfz) \mathcal{I}_2] 
   &=&
-\frac {c^2_H} 2 E\left[  \Psi(\sfz) \int_{[0,1]^2} a'(B_{t}) (a^2)'(B_r) \frac {\partial R_H }{\partial t} (r,t)  drdt \right].
\eeas
}}
}
\end{en-text}

     \medskip
     \noindent
     {\bf (iii)} By symmetriy, the third term is analogous to the second one and  produces the same contribution.

    \medskip
     \noindent
     {\bf (iv)}    The fourth term is given by
     \[
     \mathcal{I}_4=2n^{6H-\frac 32}  \sum_{j,k,\ell=1}^n  a_{t_{j-1}} a_{t_{k-1}} a_{t_{\ell-1}} \beta_{j,k} \beta_{j,\ell} \Delta B_{k,n} \Delta B_{\ell,n}.
     \]
We can make the decomposition
     \[
      \Delta B_{k,n} \Delta B_{\ell,n}= I_2(\mathbf{1}_{I_k} \otimes \mathbf{1}_{I_\ell}) + \beta_{k,\ell}.
      \]
           By integration by parts the contribution of   $I_2(\mathbf{1}_{I_k} \otimes \mathbf{1}_{I_\ell})$ is of lower order than that of
           $\beta_{k,\ell}$. In this way, it suffices to consider the term
           \[
     \mathcal{I}_{4,0}=2n^{6H-\frac 32}  \sum_{j,k,\ell=1}^n   a_{t_{j-1}} a_{t_{k-1}} a_{t_{\ell-1}} \beta_{j,k} \beta_{j,\ell}  \beta_{k,\ell}.
     \]
  From part (f) of Lemma \ref{lem1}, we see that  this term is bounded in $L^p$ by 
   $n^{2H-\frac 32}$ if $H>\frac 12$ and by $n^{-\frac 12}$ if $H<\frac 12$. This clearly implies that
   \bea\label{e11}
   \lim_{n\rightarrow \infty}  n^{\frac 12-2H}  \mathcal{I}_4 &=& 0
\eea
   if $H<\frac 12$, in $L^p$ for all $p\ge 2$.  On the other hand, 
we claim that, if $H>\frac 12$ we also have the following  convergence is   $L^p$ for all $p\ge 2$, and, as a consequence, this term produces no contribution.
\bea\label{e1}
   \lim_{n\rightarrow \infty}  n^{\frac 32-2H}  \mathcal{I}_4 &=& 0.
\eea
 {\it Proof of {\rm (\ref{e1}):}} We need to show that
\beas
   \lim_{n\rightarrow \infty}       n^{4H}   \sum_{j,k,\ell=1}^n  a_{t_{j-1}} a_{t_{k-1}} a_{t_{\ell-1}}\beta_{j,k} \beta_{j,\ell}  \beta_{k,\ell}
   &=&0
\eeas
in $L^p$. 
 We can write
 \begin{eqnarray*}
 && n^{4H}   \sum_{j,k,\ell=1}^n  a_{t_{j-1}} a_{t_{k-1}} a_{t_{\ell-1}}\beta_{j,k} \beta_{j,\ell}  \beta_{k,\ell}\\
&&\quad  =n^{-2H}   \sum_{j,k,\ell=1}^n   a(B_{\frac{j-1}n})    a(B_{\frac{k-1}n}) a(B_{\frac{\ell-1}n}) \rho_H(j-k) \rho_H(j-\ell) \rho_H(k-\ell)\\
& &\quad = n^{-2H}  \sum_{ {1\le j \le n \atop 1\le j+i \le n} \atop {1\le j+h\le n}}   a(B_{\frac{j-1}n})    a(B_{\frac{j+i-1}n}) a(B_{\frac{j+h-1}n}) \rho_H(i) \rho_H(h) \rho_H(i-h)
\end{eqnarray*}
and it suffices to show that 
\bea\label{170729-1}
\lim_{n\to\infty}
n^{1-2H}\sum_{j,h:1\leq i< h\leq n}\big|\rho_H(i) \rho_H(h) \rho_H(i-h)\big|
&=&0.
\eea
By the inequality $\sup_{i\geq1}|\rho(i)|/i^{2H-2}<\infty$, we have 
\beas 
n^{1-2H}\sum_{j,h:1\leq i< h\leq n}\big|\rho_H(i) \rho_H(h) \rho_H(i-h)\big|
&\simleq&
n^{1-2H}\sum_{j,h:1\leq i< h\leq n}i^{2H-2}h^{2H-2}(h-i)^{2H-2}
\\&=&
n^{1-2H}\sum_{h=1}^n h^{6H-5}\sum_{i=1}^{h-1}(i/h)^{2H-2}(1-i/h)^{2H-2}/h
\\&\simleq&
B(2H-1,2H-1)n^{1-2H+(6H-4)_+}. 
\eeas
The last term is $O(n^{4H-3})$ when $2/3\leq H<3/4$, and 
$O(n^{1-2H})$ when $1/2<H<2/3$. This gives (\ref{170729-1}) and hence (\ref{e1}). 

\begin{en-text}
b) Suppose now that $H<\frac 12$. In that case, from part (f) of Lemma \ref{lem1}, we see that the order of this term is $n^{-1/2}$. 
Then the contribution of $n^{-1/2}\cali_4$ is asymptotically 
 \begin{eqnarray*}
 && 2n^{6H-1}   \sum_{j,k,\ell=1}^n  a_{t_{j-1}} a_{t_{k-1}} a_{t_{\ell-1}}\beta_{j,k} \beta_{j,\ell}  \beta_{k,\ell}\\
%&&\quad  =2n^{-1}   \sum_{j,k,\ell=1}^n   a(B_{\frac{j-1}n})    a(B_{\frac{k-1}n}) a(B_{\frac{\ell-1}n}) \rho_H(j-k) \rho_H(j-\ell) \rho_H(k-\ell)\\
& &\quad = 2n^{-1}  \sum_{ {1\le j \le n \atop 1\le j+i \le n} \atop {1\le j+h\le n}}   a(B_{\frac{j-1}n})    a(B_{\frac{j+i-1}n}) a(B_{\frac{j+h-1}n}) \rho_H(i) \rho_H(h) \rho_H(i-h).
\end{eqnarray*}
First, if we replace $ a(B_{\frac{j+i-1}n})$ by  $a(B_{\frac{j-1}n})$ and $a(B_{\frac{j+h-1}n})$ by  $a(B_{\frac{j-1}n})$, we make  errors of order $(i/n)^H$ and $(h/n)^H$, respectively. These errors produce a zero limit. In fact,
\[
\lim_{n\rightarrow \infty}  n^{-H}  \sum_{|i|, |h| \le n}   |\rho_H(i) \rho_H(h) \rho_H(i-h)|  |i|^H =0.
\]
To show this, set $\rho^{(n)}_H(j)= |\rho_H(j)| \mathbf{1}_{|j|\le n}$. Then,
\begin{eqnarray*}
 \sum_{|i|, |h| \le n}   |\rho_H(i) \rho_H(h) \rho_H(i-h)|  |i|^H &=& \langle  \rho^{(n)}_H |\cdot|^H, \rho_H^{(n)} *\rho_H^{(n)} \rangle_{\ell^2(\mathbb{Z})} \\
 &\le & \| \rho^{(n)}_H |\cdot|^H| \|_{\ell^2(\mathbb{Z})} \| \rho_H^{(n)} *\rho_H^{(n)} \|_{\ell^2(\mathbb{Z})} \\
 &\le & C n^{(3H -\frac 32)_+}  \|\rho_H^{(n)} \| ^2_{\ell^{4/3}(\mathbb{Z})}  \\
 &\le & Cn^{(3H -\frac 32)_+} n^{(4H-\frac 52)_+}
 \\&=& O(1) 
 \end{eqnarray*}
when $H<1/2$. 
Therefore, we have a contribution of the form
 \[
 \lim_{n\rightarrow \infty}  2n^{6H-1}  \sum_{j,k,\ell=1}^n  a_{t_{j-1}} a_{t_{k-1}} a_{t_{\ell-1}} \beta_{j,k} \beta_{j,\ell}  \beta_{k,\ell}
 {\color{black}
 \yeq 2C_{H,2} \int_0^1 a^3(B_t) dt,
 }
\]
 {\color{black}
where
\[
C_{H,2}=  \sum_{i,h=-\infty}^{+\infty} \rho_H(i) \rho_H(h) \rho_H(i-h).
\]
Remark that $C_{H,2}$ is finite because $\sum_{i=-\infty}^\infty|\rho_H(i)|<\infty$ when $H<1/2$ and 
$|\rho_H(i-h)|\leq1$. 
}  
\end{en-text}

         \medskip
         \noindent
         {\bf (v)} The fifth term is
         \[
         \mathcal{I}_5=n^{6H-\frac 32}  \sum_{j,k,\ell=1}^n   a'_{t_{j-1}}a'_{t_{k-1} } a_{t_{\ell-1}}  q_j  \alpha_{t_{j-1},k} \alpha_{t_{k-1},\ell} \Delta B_{k,n} \Delta B_{\ell,n}.
         \]
         As before, we replace $\Delta B_{k,n} \Delta B_{\ell,n}$  by $\beta_{k,\ell}$ and it suffices to study the term
          \[
         \mathcal{I}_{5,0}=n^{6H-\frac 32}  \sum_{j,k,\ell=1}^n   a'_{t_{j-1}}a'_{t_{k-1} } a_{t_{\ell-1}}    q_j \alpha_{t_{j-1},k} \alpha_{t_{k-1},\ell}\beta_{k,\ell}.
                  \]
        We get the following contributions:
        \begin{itemize}
        \item The first factor $n^{6H-\frac 32}$.
        \item Integration by parts for $q_j$ produces $n^{-2(2H\wedge 1)}$.
        \item  $|\alpha_{t_{j-1},k} \alpha_{t_{k-1},\ell}|$ is bounded by $C_H n^{-2(2H\wedge 1)}$, due to Lemma \ref{lem1} (a).
        \item $\sum_{k,\ell} |\beta_{k,\ell}|\le C_H n^{(1-2H)\vee 0}$   due to Lemma \ref{lem1} (c).
        \item A factor $n$ comes from the sum on $j$.
        \end{itemize}
         All together gives the order $n^{6H -\frac 12 -4(2H\wedge 1) +(1-2H)\vee 0}$, which does not produce any contribution.
         
                \medskip
         \noindent
         {\bf (vi)} The  sixth term is
         \[
         \mathcal{I}_6=n^{6H-\frac 32}  \sum_{j,k,\ell=1}^n a'_{t_{j-1}}  a_{t_{k-1}} a_{t_{\ell-1}} q_j \alpha_{t_{j-1},k} \beta_{k,\ell} \Delta B_{\ell,n}.
         \]
                 We get the following contributions:
        \begin{itemize}
        \item The first factor $n^{6H-\frac 32}$.
        \item Integration by parts for $q_j$  and   $\Delta B_{\ell,n}$  produces $n^{-3(2H\wedge 1)}$.
        \item  $|\alpha_{t_{j-1},k}|$  is bounded by $C_H n^{-(2H\wedge 1)}$,  due to Lemma \ref{lem1} (a).
        \item $\sum_{k,\ell} |\beta_{k,\ell}|\le C_H n^{(1-2H)\vee 0}$, due to Lemma \ref{lem1} (c).
        \item A factor $n$ comes from the sum on $j$.
        \end{itemize}
        We get the same order as for the fifth term and no contribution.
        
            \medskip
     \noindent
     {\bf (vii)}  The seventh term is given by
          \[
     \mathcal{I}_7=2n^{6H-\frac 32}  \sum_{j,k,\ell=1}^n  a_{t_{j-1}} a'_{t_{k-1}} a_{t_{\ell-1}}  \Delta B_{j,n} \Delta B_{k,n} \Delta B_{\ell,n}  \alpha_{t_{k-1}, \ell} \beta_{j,k}
     \]
     Replacing  $ \Delta B_{j,n} \Delta B_{k,n}$ by $\beta_{j,k}$ it suffices to consider the term
        \[
     \mathcal{I}_{7,0}=2n^{6H-\frac 32}  \sum_{j,k,\ell=1}^n  a_{t_{j-1}} a'_{t_{k-1}} a_{t_{\ell-1}}    \Delta B_{\ell,n}  \alpha_{t_{k-1}, \ell} \beta_{j,k}^2.
     \]  
       We get the following contributions:
        \begin{itemize}
        \item The first factor $n^{6H-\frac 32}$.
        \item Integration by parts for     $\Delta B_{\ell,n}$  produces $n^{-(2H\wedge 1)}$.
        \item  $\sum_{\ell=1}^n|\alpha_{t_{j-1},\ell}|$  is bounded by $C_H $, by  Lemma \ref{lem1} (b).
        \item $\sum_{j,k} \beta_{j,k}^2\le C_H n^{ 1-4H}$,  by  Lemma \ref{lem1} (d).
           \end{itemize}
As a consequence, this term has the order $n^{2H-\frac 32 -(2H\wedge 1)}$, and its contribution is %analogous to 
{\colorb the same as} 
that of (ii); {\colorb replace indices $j$, $k$ and $\ell$ by $\ell$, $j$ and $k$, respectively.}

    \medskip
     \noindent
     {\bf (viii)}    The eighth term is given by
     \[
     \mathcal{I}_8=    2n^{6H-\frac 32}  \sum_{j,k,\ell=1}^n a_{t_{j-1}} a_{t_{k-1}} a_{t_{\ell-1}} \Delta B _{j,n} \Delta B_{\ell,n} \beta_{j,k} \beta_{k,\ell}.
     \]  
      Replacing $\Delta B _{j,n} \Delta B_{\ell,n} $ by $\beta_{j,\ell}$, is suffices to consider the reduced term
      \[
     \mathcal{I}_{8,0}=    2n^{6H-\frac 32}  \sum_{j,k,\ell=1}^n a_{t_{j-1}} a_{t_{k-1}} a_{t_{\ell-1}}  \beta_{j,\ell} \beta_{j,k} \beta_{k,\ell}.
     \]    
      By Lemma \ref{lem1} part (f), this term is of order $n^{2H-\frac 32 -(2H\wedge 1)}$. Its contribution is  %analogous to 
      {\colorb the same as} that of term (iv).
      
      \medskip
      \noindent
      {\bf (B)}
   One can show by a similar argument that
      \[
      \lim_{n\rightarrow \infty}  \langle D \langle D N_n, u_n \rangle_{\HH}, u_n \rangle_{\HH} =0,
      \]
      in  $L^p$, for all $p\ge 2$.

      \medskip
%     Set
%      \[
%      \Psi(\sfz)= \exp\left(  -\sfz^2c_H^2 \int_0^1 a^2(B_s) ds \right).
%      \]
      In conclusion, we obtain the following results on the random symbol 
            ${\mathfrak S}^{(3,0)}(\tti \sfz)$:

     \medskip
     \noindent
     {\bf Case $H>\frac 12$} We have proved that
     \begin{eqnarray*}
      E [\Psi(\sfz)  \mathfrak{S}^{3} (\tti\sfz)] 
     & =& \lim_{n\rightarrow \infty} \frac 13 n^{\frac 32 -2H} E\left[ \Psi(\sfz) \left \langle D\langle DM_n, u_n \rangle_{\HH}, u_n \right\rangle_{\HH} (\tti\sfz)^3 \right] \\
     &=&   C_{H,3} (\tti\sfz)^3 \int_{[0,1]^4}   E\left[ D_{s'}  \left(  \Psi(\sfz) a(B_r) D_s (a^2(B_t)) \right) \right] |r-s|^{2H-2} |r-s'|^{2H-2}  dsds'dr  dt.
     \end{eqnarray*}
Taking into account that $\alpha_H \int_0^t  |r-s|^{2H-2} ds =\frac {\partial R_H} {\partial r}(t,r)$, we can write the above expression as follows:
          \begin{eqnarray*}
     && E [\Psi(\sfz)  \mathfrak{S}^{3} (\tti\sfz)] \\
          &&=  c_H^4 (\tti\sfz)^5 \int_{[0,1]^3}   E\left[   \Psi(\sfz) (a^2) '(B_\theta) a(B_r)  (a^2)'(B_t)   \right]    \frac {\partial R_H} {\partial r}(\theta,r)\frac {\partial R_H} {\partial r}(t,r)  dt   d\theta dr  \\
           &&\quad +   H  {\colorg c_H^2} (\tti\sfz)^3 \int_{[0,1]^2}   E\left[  \Psi(\sfz)  a'(B_r)  (a^2)'(B_t) \right]  r^{2H-1} \frac {\partial R_H} {\partial r}(t,r)  dt  dr \\
           &&\quad +     {\colorg c_H^2} (\tti\sfz)^3 \int_{[0,1]^2}   E\left[  \Psi(\sfz)  a(B_r)  (a^2)''(B_t) \right]   \left(\frac {\partial R_H} {\partial r}(t,r) \right )^2dt  dr. \\
     \end{eqnarray*}

         \medskip
     \noindent
     {\bf Case $H<\frac 12$}  We have obtained, that in this case, $\mathfrak{S}^{3} (\tti\sfz)=0$.

  %    \begin{eqnarray*}
    %  E [\Psi(\sfz)  \mathfrak{S}^{3} (\tti\sfz)] 
     %& =& \lim_{n\rightarrow \infty} \frac 13  \sqrt{n} E\left[ \Psi(\sfz) \left \langle D\langle DM_n, u_n \rangle_{\HH}, u_n \right\rangle_{\HH} (\tti\sfz)^3 \right] \\
     %&&=  -\frac {c^2_H} 2  (\tti\sfz)^3E\left[  \Psi(\sfz) \int_{[0,1]^2} 
     %(a' (a^2)') (B_r) 
     %{\colorg a'(B_t) (a^2)' (B_r) }
     %\frac {\partial R_H }{\partial t} (r,t) drdt \right]  
     %&&\quad  +\frac 43C_{H,2}(\tti\sfz)^3  \int_0^1 a^3(B_t) dt.
     %\end{eqnarray*}
 
     %
     \medskip
     \noindent
   \subsection{Quasi tangent}
     
     For the quasi tangent we have
     \begin{eqnarray*}
     \langle DM_n, u_n \rangle_{\HH} - G_\infty &=& n^{4H-1} \left  \langle    \sum_{j=1}^n  \left[( D a_{t_{j-1}}) q_j +   2  a_{t_{j-1}}  \Delta B_{j,n}   \mathbf{1}_{ I_j}  \right],
      \sum_{k=1}^n  a_{t_{k-1}} \Delta  B_{k,n}   \mathbf{1}_{ I_k}  \right  \rangle_{\HH}\\
      &&-  2c_H^2 \int_0^1 a^2 (B_s) ds\\
      &=& G_{n,1} + G_{n,2} + G_{n,3},
     \end{eqnarray*}
     where
     \[
     G_{n,1}= n^{4H-1}  \sum_{j,k=1}^n\langle   (D a_{t_{j-1}}) q_j , a_{t_{k-1}} \Delta  B_{k,n}   \mathbf{1}_{ I_k} \rangle_{\HH}
     =n^{4H-1}  \sum_{j,k=1}^n a'_{t_{j-1}}  a_{t_{k-1}} \alpha_{t_{j-1},k} q_j  \Delta  B_{k,n} ,
     \]
     \[
     G_{n,2}= 2n^{4H-1}  \sum_{j,k=1}^n  a_{t_{j-1}} a_{t_{k-1}} [\Delta B_{j,n}\Delta  B_{k,n}  -\beta_{j,k}] \beta_{j,k}
     \]
     and
     \[
      G_{n,3}= 2n^{4H-1}  \sum_{j,k=1}^n  a_{t_{j-1}} a_{t_{k-1}} \beta_{j,k}^2-2c_H^2 \int_0^1 a^2 (B_s) ds.
      \]
     Integrating by parts the terms $q_j$ and $ \Delta  B_{k,n} $ and using that, by Lemma \ref{lem1} (b),  $ \sum_{k=1}^n |\alpha_{t_{j-1},k} |$ is bounded by a constant not depending on $n$, we obtain:
       \[
   |E[ \Psi(\sfz)G_{n,1}] |\le  Cn^{4H-1+1-3(2H\wedge 1)} =  Cn^{4H-3(2H\wedge 1)} .
     \]
     For $H>\frac 12$, we get $4H-3$ which is faster than $2H-\frac 32$ 
     {\colorb because $H<3/4$}, and for $H<\frac 12$, we obtain $-2H$ which is faster than $\frac 12-2H$.
     
     For the term $G_{n,2}$, {\colorb by using integration-by-parts to $\Delta B_{j,n}\Delta  B_{k,n}  -\beta_{j,k}$ 
     and Lemma \ref{lem1} (c),} we get
     \[
        |E[ \Psi(\sfz)G_{n,2}] |\le  Cn^{4H-1-2(2H\wedge 1)+ (1-2H) \vee 0},
     \]
     which produces the same rates as before.
     
     Finally,
          \begin{eqnarray*}
      G_{n,3}&=& 2n^{4H-1}  \sum_{j,k=1}^n  a_{t_{j-1}} a_{t_{k-1}} \beta_{j,k}^2-2c_H^2 \int_0^1 a^2 (B_s) ds \\
      &&=\frac 2n \sum_{j,k=1}^n  a(B_{t_{j-1}}) a(B_{t_{k-1}})\rho_H^2(j-k)-2c_H^2 \int_0^1 a^2 (B_s) ds\\
      &&=\frac 2n \sum_ {1\le j \le n, 1\le j+i \le n}  a(B_{t_{j-1}}) a(B_{t_{j+i-1}})\rho_H^2(i)-2c_H^2 \int_0^1 a^2 (B_s) ds.
      \end{eqnarray*} 
\begin{en-text}
      Replacing  $a(B_{t_{j+i-1}})$ by $a(B_{t_{j-1}})$  in the summation produces an error of the order $(i/n)^H$, which gives a total error of $n^{(5H-3)_+-H}$, which is faster than
      $n^{2H-\frac 32}$ if $H>\frac 12$ or than $n^{-\frac 12}$ if $H<\frac 12$. So, it suffices to consider the term
      \[
      \frac 1n \sum_ {1\le j \le n, 1\le j+i \le n}  a^2(B_{t_{j-1}})  \rho_H^2(i)-c_H^2 \int_0^1 a^2 (B_s) ds,
      \]
      which converges to zero at a rate faster than $n^{2H-\frac 32}$ if $H>\frac 12$ or than $n^{-\frac 12} $ if $H<\frac 12$.
\end{en-text}
Replacing $a(B_{t_{j+i-1}})$ by $a(B_{t_{j-1}})$ with integration-by-parts, we have the error 
\beas&&
E\left[\Psi(\sfz)\left\{
\frac{1}{n} \sum_ {1\le j \le n, 1\le j+i \le n}  a(B_{t_{j-1}}) 
\big(a(B_{t_{j+i-1}})-a(B_{t_{j-1}})\big)\rho_H^2(i)
\right\}\right]
\\&=&
\frac{1}{n}\sum_ {1\le j \le n, 1\le j+i \le n} 
E\left[\left\langle D\left\{\Psi(\sfz)
a(B_{t_{j-1}}) \int_0^1a'\big(B_{t_{j-1}}+\theta(B_{t_{j+i-1}}-B_{t_{j-1}})\big)d\theta\right\},
{\bf 1}_{[t_{j-1},t_{j+i-1}]}
\right\rangle
\right]\rho_H^2(i)
\\&\simleq&
\sum_{i=1}^n (i/n)^{(2H)\wedge1}\rho_H^2(i)
\\&\simleq&
n^{-2H\wedge1+(2H\wedge1+2(2H-2)+1)_+}. 
\eeas
In the case $1/2<H<3/4$, 
this converges to zero at the rate $n^{4H-3}$, that is faster than $n^{2H-3/2}$, and 
in the case $1/4<H<1/2$, 
at the rate $n^{-2H}$, that is faster than $n^{\frac 12 -2H}$. 
Let $\ep\in(0,1)$. Divide the sum $\sum_ {1\le j \le n}$ into 
$\sum_{n^\ep+1\leq j\leq n-n^\ep}$ and the rest. 
Since the convergence $c_H^2=\lim_{I\up\bbZ}\sum_{i\in I}\rho(i)^2$ is monotone, 
the order of the $L^p$-norm of 
\beas 
\frac{1}{n} \sum_ {1\le j \le n, 1\le j+i \le n}  a(B_{t_{j-1}})^2 \rho_H(i)^2
-\frac{1}{n} \sum_ {1\le j \le n} a(B_{t_{j-1}})^2c_H^2
\eeas
is not greater than 
\beas
\left(c_H^2-\sum_{-n^\ep\leq i\leq n^\ep}\rho_H(i)^2\right)+n^{-1+\ep}
&\simleq&
n^{(4H-3)\ep}+n^{-1+\ep}. 
\eeas
Then, in case $H\in(1/2,3/4)$, we can find $\ep\in(1/2,1)$ such that 
$n^{-(3-4H)\ep}+n^{-1+\ep}=o(n^{2H-3/2})$. 
In case $H\in(1/4,1/2)$, it is possible to find $\ep\in(0,1/2)$ such that 
$n^{-(3-4H)\ep}+n^{-1+\ep}=o(n^{\frac 12-2H})$. 
Once again by the Taylor formula and integration-by-parts, we see 
\beas 
E\left[\Psi(\sfz)\left\{
\frac{1}{n} \sum_ {1\le j \le n}  a(B_{t_{j-1}})^2-\int_0^1 a^2 (B_s) ds
\right\}\right]
&=&
O(n^{-(2H)\wedge1}),
\eeas
which is $o(n^{2H-3/2})$ for $H>1/2$, and $o(n^{\frac 12-2H})$ for $H\in(1/4,1/2)$. 
Therefore, $G_{n,3}$ has no contribution. 

In conclusion, the quasi tangent has no effect in the asymptotic expansion, that is, 
${\mathfrak S}^{(2,0)}_0=0$.

%%%%%%%%%%%%%%%%%%%%%%%%%%
%%%%%%%%%%%%%%%%%%%%%%%%%%
%%%%%%%%%%%%%%%%%%%%%%%%%%
%%% ORIGINAL %%%%%%%%%%%%%%%%%
\begin{en-text}
\subsection{Perturbation term}
In this subsection we study the contribution of the term $N_n$ to the asymptotic expansion.
More precisely we will compute the random symbol ${\mathfrak G}^{1} (\tti \sfz):= {\mathfrak G}^{(1,0)} (\tti \sfz,\tti \sfx)$.  The action of this symbol is defined by
\beas
E[ \Psi(\sfz) {\mathfrak G}^{1} (\tti \sfz)] =\lim_{n\rightarrow \infty}  E[ \Psi(z) N_n (\tti \sfz)].
\eeas

\medskip
\noindent
{\bf (i)} Case $H>\frac 12$. We recall that
\beas
N_n &=& n \sum_{j=1}^n  \Delta B_{j,n} \langle D a_{t_{j-1}}, \mathbf{1}_{ I_j}\rangle_{\HH}.
\eeas
Integrating by parts yields
\beas
E[ \Psi(z) N_n (\tti \sfz)]
&=&  n \sum_{j=1}^n E[ \langle D   \langle \Psi(\sfz)D a_{t_{j-1}}, \mathbf{1}_{ I_j}\rangle_{\HH},  {\bf 1}_{I_j} \rangle_{\HH} (\tti\sfz)\\
&=& n \sum_{j=1}^n E[ \Psi(\sfz) a''(B_{t_{j-1}}) ]   \alpha^2_{t_{j-1}, j} (\tti\sfz) \\
&&- nc_H^2 \sum_{j=1}^n\int_0^1    E[   \Psi(\sfz)    (a^2)'(B_s) a_{t_{j-1}}] \alpha_{t_{j-1},j}  \alpha_{s,j} ds (\tti\sfz)\\
&&:=G_{1,n} +G_{2,n}.
\eeas
We know that
\beas
 \alpha^2_{t_{j-1}, j} = \frac 12 n^{-4H} ( j^{2H} -(j-1)^{2H} -1)^2.
 \eeas
Expanding the square $( j^{2H} -(j-1)^{2H} -1)^2$ it turns out that the terms   $-2(j^{2H} -(j-1)^{2H})$ and $1$ do not tribute to the limit. So it suffices to consider the term
\beas
&& \frac 12 n^{1-4H}  \sum_{j=1}^n E[ \Psi(\sfz) a''(B_{t_{j-1}}) ]    ( j^{2H} -(j-1)^{2H})^2
\\ &&= Hn^{1-4H}  \sum_{j=1}^n E[ \Psi(\sfz) a''(B_{t_{j-1}}) ]     j^{2(2H-1)} +o(1) \\
&&=H E\left[ \Psi(\sfz) \int_0^1 a''(B_s) ds \right]  s^{4H-2} ds + o(1).
 \eeas
The term $G_{2,n}$ converges to zero because, by Lemma \ref{lem1} (a) and (b), we can write
\beas
G_{2,n} \simleq n^{1-2H}  \sup_{s\in [0,1]}\sum_{j=1}^n |\alpha_{s,j} |  \le n^{1-2H} C_H.
\eeas
Therefore, we have proved that
\beas
E[ \Psi(\sfz) {\mathfrak G}^{1} (\tti \sfz)]  =(\tti\sfz) H E\left[ \Psi(\sfz) \int_0^1 a''(B_s) ds \right]  s^{4H-2} ds.
\eeas

\medskip
\noindent
{\bf (ii)} Case $H<\frac 12$. In that case,
\beas
N_n &=& n^{4H-1} \sum_{j=1}^n  \Delta B_{j,n} \langle D a_{t_{j-1}}, \mathbf{1}_{ I_j}\rangle_{\HH}.
\eeas
We proceed as before, but in that case the dominating term  $ \alpha^2_{t_{j-1}, j} $ is the constant $1$ and we obtain
\beas
E[ \Psi(\sfz) {\mathfrak G}^{1} (\tti \sfz)]  =(\tti\sfz)  E\left[ \Psi(\sfz) \int_0^1 a''(B_s) ds \right]   ds.
\eeas
\end{en-text}
%%%%%%%%%%%%%%%%%%%%%%%%%%
%%%%%%%%%%%%%%%%%%%%%%%%%%
%%%%%%%%%%%%%%%%%%%%%%%%%%

\subsection{Perturbation term}
In this subsection we study the contribution of the term $N_n$ to the asymptotic expansion.
More precisely we will compute the random symbol ${\mathfrak G}^{1} (\tti \sfz)$.  The action of this symbol is defined by
\beas
E[ \Psi(\sfz) {\mathfrak G}^{1} (\tti \sfz)] =\lim_{n\rightarrow \infty}  E[ \Psi(z) N_n (\tti \sfz)].
\eeas

\medskip
\noindent
{\bf (i)} Case $H>\frac 12$. We recall that
\beas
N_n &=& n \sum_{j=1}^n  \Delta B_{j,n} \langle D a_{t_{j-1}}, \mathbf{1}_{ I_j}\rangle_{\HH}.
\eeas
Integrating by parts yields
\beas
E[ \Psi(z) N_n (\tti \sfz)]
&=&  n \sum_{j=1}^n E[ \langle D   \langle \Psi(\sfz)D a_{t_{j-1}}, \mathbf{1}_{ I_j}\rangle_{\HH},  {\bf 1}_{I_j} \rangle_{\HH} (\tti\sfz)\\
&=& n \sum_{j=1}^n E[ \Psi(\sfz) a''(B_{t_{j-1}}) ]   \alpha^2_{t_{j-1}, j} (\tti\sfz) \\
&&{\colred +}  {\color {black} \frac n 2}c_H^2 \sum_{j=1}^n\int_0^1    E[   \Psi(\sfz)    (a^2)'(B_s) {\colred a'_{t_{j-1}}}] \alpha_{t_{j-1},j}  \alpha_{s,j} ds (\tti\sfz)^{\colred 3}\\
&&:=G_{1,n} +G_{2,n}.
\eeas
We know that
\bea\label{20171230-1}
 \alpha^2_{t_{j-1}, j} = \frac{1}{{\colred 4}} n^{-4H} ( j^{2H} -(j-1)^{2H} -1)^2.
 \eea
Expanding the square $( j^{2H} -(j-1)^{2H} -1)^2$ it turns out that the terms   $-2(j^{2H} -(j-1)^{2H})$ and $1$ do not tribute to the limit. 
So, {\colred for $G_{1,n}$} it suffices to consider the term
\beas
&& \frac{1}{{\colred 4}} n^{1-4H}  \sum_{j=1}^n E[ \Psi(\sfz) a''(B_{t_{j-1}}) ]    ( j^{2H} -(j-1)^{2H})^2
\\ &&= {\color {black} H^2}n^{1-4H}  \sum_{j=1}^n E[ \Psi(\sfz) a''(B_{t_{j-1}}) ]     j^{2(2H-1)} +o(1) \\
&&={\color {black} H^2} E\bigg[ \Psi(\sfz) \int_0^1 a''(B_s) s^{4H-2} {\colred ds  \bigg] }+ o(1).
 \eeas
{\color {black} The term $G_{2,n}$ converges to zero because, by {\colred (\ref{20171230-1}) and}
Lemma \ref{lem1} %(a) and 
(b), we can write
\[
G_{2,n} \simleq n^{1-2H} \sup_{s\in [0,1]} \sum_{j=1} ^\infty |\alpha_{s,j}| \le C_H n^{1-2H}.
\]
}
%{\colred 
%For $G_{2,n}$, we have 
%\beas &&
%n\sum_{j=1}^n\int_0^1E[\Psi(\sfz)(a^2)'_sa'_{t_{j-1}}]\alpha_{t_{j-1},j}\alpha_{s,j}ds
%\\&=&
%H^2(2H-1)\sum_{j=1}^n\int_0^1E[\Psi(\sfz)(a^2)'_sa'_{t_{j-1}}](j/n)^{2H-1}
%\bigg(\int_{t=t_{j-1}}^{t_j}\int_{r=0}^s|t-r|^{2H-2}drdt\bigg)ds
%\\&=&
%H^2(2H-1)\int_0^1\int_0^1E[\Psi(\sfz)(a^2)'_sa'_t]t^{2H-1}
%\bigg(\int_{0}^s|t-r|^{2H-2}dr\bigg)dtds+o(1)
%\\&=&
%H^2\int_0^1\int_0^1E[\Psi(\sfz)(a^2)'_sa'_t]t^{2H-1}
%\big\{t^{2H-1}+|s-t|^{2H-1}\text{sign}(s-t)\big\}dtds+o(1).
%\eeas
%
Therefore, we have proved that
{\color {black}
\[
E[ \Psi(\sfz) {\mathfrak G}^{1} (\tti \sfz)]  =
  H^2 E\left[ \Psi(\sfz)\>  \int_0^1 a''(B_s)  s^{4H-2} ds\right] (\tti\sfz).
\]}

%+
%E\bigg[ \Psi(\sfz)\> 
%(\tti\sfz)^3c_H^2
%H^2\int_0^1\int_0^1(a^2)'_sa'_tt^{2H-1}
%\big\{t^{2H-1}+|s-t|^{2H-1}\text{sign}(s-t)\big\}dtds\bigg]. 
%\eeas
%}

\medskip
\noindent
{\bf (ii)} Case $H<\frac 12$. In that case,
\beas
N_n &=& n^{4H-1} \sum_{j=1}^n  \Delta B_{j,n} \langle D a_{t_{j-1}}, \mathbf{1}_{ I_j}\rangle_{\HH}.
\eeas
We proceed as before, but in that case the dominating term 
{\colred in} $ \alpha^2_{t_{j-1}, j} $ {\colred in $G_{1,n}$} is the constant $1$ and we obtain
\beas
E[ \Psi(\sfz) {\mathfrak G}^{1} (\tti \sfz)]  
&=&
E\left[ \Psi(\sfz) \>{\colred \frac{1}{4}}(\tti\sfz)  \int_0^1 a''(B_s) ds \right]   ds.
\eeas
%{\colred Indeed, 
%the term $G_{2,n}$ converges to zero because, by Lemma \ref{lem1} (a) and (b), we can write
%\beas
%|G_{2,n}| &\simleq& n^{2H-1}  \sup_{s\in [0,1]}\sum_{j=1}^n |\alpha_{s,j} |  \le n^{2H-1} C_H.
%\eeas
%}

We can also show that
\beas
\lim_{n\rightarrow \infty}  E[ \Psi(z)  \langle DN_n ,u_n \rangle_{\HH} ]=0.
\eeas

\begin{rem}\rm 
The functional $\Psi(\sfz)$ should be replaced by $\Psi(\sfz)\psi_n$ when we need a truncation $\psi_n$. 
However, the above arguments are essentially unchanged because 
$\|1-\psi_n\|_{{\ell},p}$ would converge to zero much faster than the total error we found. 
\end{rem}

The following lemma is in Nourdin, Nualart and Peccati \cite{nourdin2016quantitative}.

     \bigskip
     \begin{lemme} \label{lem1}
\label{lem1} Let $0<H<1$ and $n\geq 1$. We have, for some constant $C_H$,
\begin{itemize}
\item[(a)] $\left|  \alpha_{t,k} \right| \leqslant n^{-(2H\wedge 1)}$ for any $t\in
\lbrack 0,1]$ and $k=1,\dots, n$.
\item[(b)] $\sup_{t\in \lbrack 0,1]}\sum_{k=1}^{n}\left|  \alpha_{t,k}\right| \leq C_H$.
\item[(c)] $\sum_{k,j=1}^{n}\left|  \beta_{j,k}\right| \leq C_H n^{(1-2H)\vee 0}$.
\item[(d)]   If $H<\frac 34$, then  $\sum_{k,j=1}^{n} \beta_{j,k} ^{2}\leq C_Hn^{1-4H}$.
 \item[(e)]    $\sum_{k,j=1}^{n}\left|  \beta_{k,l} \beta_{j,l} \right| \leq C_{H}n^{-(4H \wedge 2)}$  for any $l=1,\dots, n$.
 \item[(f)]   If $H<\frac 34$, then $\sum_{k,j=1}^{n}\left|  \beta_{k,l} \beta_{j,l}  \beta_{j,k}\right| \leq C_H n^{-4H -(2H\wedge 1)}$ for any $l=1,\dots, n$.
\end{itemize}
\end{lemme}
\onelineskip

%%%%%%%%%%%%%%%%%%%%%%%%%%%%%%%
\section{Asymptotic expansion for measurable functions}\label{170813-5}
{\colorb

Let $\ell=\check{\sfd}+8$ and denote by 
$\beta_x$ the maximum degree in $x$ of ${\colred {\mathfrak S}.}$ %[Estimate $\sum_{i=5}^{12}R^{(i)}_n$]. 
We denote by $\sigma_F$ the Malliavin covariance matrix of a multivariate functional $F$ 
and write $\Delta_F=\det\sigma_F$. 
{\colorr Let $\sfd_2=(\ell+\beta_x-7)\vee\big(2[(\sfd_1+2)/2]+2[(\beta_x+1)/2]\big)$, 
where $[x]$ is the maximum integer not larger than $x$.}
We consider the following condition. 
%%%
\bd\im[[C\!\!]] 
{\bf (i)}  
$u_n\in\bbD^{{\colorr \ell+1},\infty}(\mfh\otimes\bbR^\sfd)$, 
%$G_\infty\in\bbD^{\ell\vee{\colorr (\ell+\beta_x-7)\vee \sfd_2},\infty}(\bbR^\sfd\otimes_+\bbR^\sfd)$, 
$G_\infty\in\bbD^{\ell{\colorr\vee \sfd_2},\infty}(\bbR^\sfd\otimes_+\bbR^\sfd)$, 
$W_n,N_n\in\bbD^{{\colorr \ell},\infty}(\bbR^\sfd)$, 
%$W_\infty\in\bbD^{{\colorr \ell\vee(\ell+\beta_x-7)\vee \sfd_2},
$W_\infty\in\bbD^{{\colorr \ell\vee \sfd_2},
\infty}(\bbR^\sfd)$, 
$X_n\in\bbD^{{\colorr \ell},\infty}(\bbR^{\sfd_1})$, 
%$X_\infty\in\bbD^{{\colorr \ell\vee(\ell+\beta_x-6)\vee(\sfd_2+1)},\infty}(\bbR^{\sfd_1})$. 
$X_\infty\in\bbD^{{\colorr \ell\vee(\sfd_2+1)},\infty}(\bbR^{\sfd_1})$. 
%{\colorr [Construct  $\psi_n\in\bbD^{\ell-2,\infty}(\bbR)$. ]}
%

\bd
\im[\hspace{-2mm}(ii)] For every $p>1$, the following estimates hold: 
\bea\label{c1}
%\|u_n\|_{{\colorr \ell-2},p} &=& O(1)
\|u_n\|_{{\colorr \ell},p} &=& O(1) % $\ell$ for bouindedness of $\ell-2$-derivative of $\psi_n$
\eea
\bea\label{c22}
\|G_n^{(2)}\|_{{\colorr \ell-2},p} &=& O(r_n)%\|G_n^{(2)}\|_{\ell-1,p} &=& O(r_n)%\qquad(k=2,3)
\eea
\bea\label{c23}
\|G_n^{(3)}\|_{{\colorr \ell-2},p} &=& O(r_n)%\qquad(k=2,3)
\eea
\bea\label{c33}
\|\big\langle DG^{(3)}_n,u_n\big\rangle_\mfh\|_{{\colorr \ell-1},p}
&=& o(r_n)%\qquad(k=2,3)
\eea
\bea\label{c42}
\bigg\|\bigg\langle D\bigg(\big\langle DG^{(2)}_n,u_n\big\rangle_\mfh\bigg),u_n\bigg\rangle_\mfh\bigg\|_{\ell-3,p}
&=& o(r_n)%\qquad(k=2,3)
\eea
\bea\label{c5}
%W_n,X_n\text{ for }G_n^{(1)}
\sum_{{\sf A}=W_\infty,X_\infty}
\big\|\langle D{\sf A},u_n\rangle_\mfh\big\|_{\ell-1,p} &=& O(r_n)
\eea
%\bea\label{b6}
%\big\|\langle DX_\infty[\sfx],u_n[\sfz]\rangle_\mfh\big\|_p &=& O(r_n)
%\eea
\bea\label{c7}
\sum_{{\sf A}=W_\infty,X_\infty}
\|\big\langle D\langle D{\sf A},u_n\rangle_\mfh,u_n\big\rangle_\mfh\|_{\ell-2,p}
&=& o(r_n)
\eea
%\bea\label{b8}
%\|\big\langle D\langle DX_\infty[\sfx],u_n[\sfz]\rangle_\mfh,u_n[\sfz]\big\rangle_\mfh\|_{p/3}
%&=& o(r_n)
%\eea
\bea\label{c10}
\|\dotw_n\|_{\ell-1,p}+\|N_n\|_{\ell-1,p}+\|\dotx_n\|_{\ell-1,p}&=&O(1)
\eea
\bea\label{c10.1}
\sum_{{\sf B}=\dotw_n,N_n,\dotx_n}
\bigg\| \big\langle D\langle D{\sf B},u_n\rangle_\HH ,u_n\big\rangle_\HH\bigg\|_{\ell-2,p}&=& o(1)
\eea
%
%
%!TEX encoding = UTF-8 Unicode
\im[\hspace{-2mm}(iii)] 
For each pair $({\mathfrak T}_n,{\mathfrak T})=({\mathfrak S}^{(3,0)}_n,{\mathfrak S}^{(3,0)})$, 
$({\mathfrak S}^{(2,0)}_{0,n},{\mathfrak S}^{(2,0)}_0)$, 
$({\mathfrak S}^{(2,0)}_n,{\mathfrak S}^{(2,0)})$, 
$({\mathfrak S}^{(1,1)}_n,{\mathfrak S}^{(1,1)})$, 
$({\mathfrak S}^{(1,0)}_n,{\mathfrak S}^{(1,0)})$, 
$({\mathfrak S}^{(0,1)}_n,{\mathfrak S}^{(0,1)})$, 
{\colorg 
$({\mathfrak S}^{(2,0)}_{1,n},{\mathfrak S}^{(2,0)}_1)$ and 
$({\mathfrak S}^{(1,1)}_{1,n},{\mathfrak S}^{(1,1)}_1)$,} 
the following conditions are satisfied. 
\bd
\im[(a)] ${\mathfrak T}$ is polynomial random symbol the coefficients of which are in 
$\bbD^{\csfd+\beta_x+1,1+}=\bigcup_{p>1}\bbD^{\csfd+\beta_x+1,p}$. 
\im[(b)] 
For some $p>1$, there exists a polynomial random symbol $\bar{\mathfrak T}_n$ that has $L^p$ coefficients and  
the same degree as ${\mathfrak T}$, 
\beas 
E\big[\Psi(\sfz,\sfx){\mathfrak T}_n({\colred\tti}\sfz,{\colred\tti}\sfx)\big]
&=& E\big[\Psi(\sfz,\sfx)\bar{\mathfrak T}_n({\colred\tti}\sfz,{\colred\tti}\sfx)\big]
\eeas
and 
$\bar{\mathfrak T}_n\to{\mathfrak T}$ in $L^p$. 
\ed

\begin{en-text}
For every $\sfz\in\bbR^\sfd$ and $\sfx\in\bbR^{\sfd_1}$, 
it holds that 
\beas 
\lim_{n\to\infty}
{\colorg 
E\big[\Psi(\sfz,\sfx){\mathfrak T}_n(\tti\sfz,\tti\sfx)%\psi_n
\big]}
&=&
E\big[\Psi(\sfz,\sfx){\mathfrak T}(\tti\sfz,\tti\sfx)\big]
\eeas
for $({\mathfrak T}_n,{\mathfrak T})=({\mathfrak S}^{(3,0)}_n,{\mathfrak S}^{(3,0)})$, 
$({\mathfrak S}^{(2,0)}_{0,n},{\mathfrak S}^{(2,0)}_0)$, 
$({\mathfrak S}^{(2,0)}_n,{\mathfrak S}^{(2,0)})$, 
$({\mathfrak S}^{(1,1)}_n,{\mathfrak S}^{(1,1)})$, 
$({\mathfrak S}^{(1,0)}_n,{\mathfrak S}^{(1,0)})$, 
$({\mathfrak S}^{(0,1)}_n,{\mathfrak S}^{(0,1)})$, 
{\colorg 
$({\mathfrak S}^{(2,0)}_{1,n},{\mathfrak S}^{(2,0)}_1)$ and 
$({\mathfrak S}^{(1,1)}_{1,n},{\mathfrak S}^{(1,1)}_1)$.}
Moreover, the coefficients of their random symbols are in $\bbD^{\csfd+\beta_x+1,\infty}$. %$\bbD^{\ell+\beta_x-7,\infty}$. 
\end{en-text}
%. %

\im[\hspace{-2mm}(iv)] 
\hspace{0.5mm} {\bf (a)} $G_\infty^{-1}\in L^{\infty-}$. 
\bd\im[(b)]
There exist $c\in(-1,0)\cup(0,1)$ and $\kappa>0$ such that 
\beas 
%\sup_{\theta\in(\theta_0,1]}
P\big[\Delta_{(c M_n+W_\infty,X_\infty)}<s_n\big] &=& O(r_n^{1+\kappa})
\eeas
%Moreover, 
%\beas 
%P\big[\Delta_{X_\infty}<s_n\big] &=& O(r_n^{1+\kappa}){\colorr remove}
%\eeas
for some positive random variables $s_n\in\bbD^{\ell-2,\infty}$ satisfying 
$\sup_{n\in\bbN}(\|s_n^{-1}\|_p+\|s_n\|_{\ell-2,p})<\infty$ for every $p>1$. 
\im[(c)] There exists $\kappa_1>0$ such that 
\beas 
\sum_{{\sf A}=DW_\infty,DX_\infty}\big\|\langle DM_n,{\sf A}\rangle_\HH\big\|_p &=& O(r_n^{\kappa_1})
\eeas
for every $p>1$. 
\ed
\ed
\ed
\onelineskip
%%%
\begin{rem}\rm 
(i) The index $\ell+1$ of $u_n$ comes from (\ref{c42}). 
(ii) Intuitively, %the uniform non-degeneracy condition for $M_n+\theta^{-1}W_\infty$ suggests 
[C] (iv) (c) is 
a kind of orthogonality between $M_n$ and $(W_\infty,X_\infty)$. 
It is natural because $M_n$ converges stably in most statistical problems. 
We will give a slightly different formulation of the problem later. 
(iii) The degree of ${\mathfrak T}_n$ and ${\mathfrak T}$ may be different. 
That $\bar{\mathfrak T}_n\to{\mathfrak T}$ in $L^p$ means 
$L^p$ convergence of all random coefficients. 
(iv) [C] (iv) (b) ensures non-degeneracy of $X_\infty$, that is, 
$\Delta_{X_\infty}^{-1}\in L^{\infty-}$. 
\end{rem}
\halflineskip

\begin{rem}\rm
Condition [C] (iii) is a sufficient condition for coming results. We can replace [C] (iii) by 
\bd
\im[[C\!\!]] (iii)$^\flat$ 
For pairs of polynomial random symbols $({\mathfrak T}_n,{\mathfrak T})=({\mathfrak S}^{(3,0)}_n,{\mathfrak S}^{(3,0)})$, 
$({\mathfrak S}^{(2,0)}_{0,n},{\mathfrak S}^{(2,0)}_0)$, 
$({\mathfrak S}^{(2,0)}_n,{\mathfrak S}^{(2,0)})$, 
$({\mathfrak S}^{(1,1)}_n,{\mathfrak S}^{(1,1)})$, 
$({\mathfrak S}^{(1,0)}_n,{\mathfrak S}^{(1,0)})$, 
$({\mathfrak S}^{(0,1)}_n,{\mathfrak S}^{(0,1)})$, 
{\colorg 
$({\mathfrak S}^{(2,0)}_{1,n},{\mathfrak S}^{(2,0)}_1)$ and 
$({\mathfrak S}^{(1,1)}_{1,n},{\mathfrak S}^{(1,1)}_1)$,}
the coefficients of ${\mathfrak T}$ are in $\bbD^{\csfd+\beta_x+1,1+}$, %$\bbD^{\ell+\beta_x-7,\infty}$.  
and  
\beas 
\lim_{n\to\infty}
{\colorg 
\partialbs^\alpha E\big[\Psi(\sfz,\sfx){\mathfrak T}_n(\tti\sfz,\tti\sfx)%\psi_n
\big]}
&=&
\partialbs^\alpha E\big[\Psi(\sfz,\sfx){\mathfrak T}(\tti\sfz,\tti\sfx)\big]
\eeas
for 
every $(\sfz,\sfx)\in\bbR^\csfd$ and $\alpha\in\bbZ_+^\csfd$. 
\ed
\end{rem}
\halflineskip

Define $\xi_n$ by 
{\colred 
\bea\label{170810-1} 
\xi_n &=& \frac{3s_n}{2s_n+12\Delta_n}
+\frac{e_n}{s_n^2}
+\frac{f_n}{\Delta_{X_\infty}^2}
\eea
}
\begin{en-text}
\bea\label{170810-1} 
\xi_n &=& \frac{3s_n}{2s_n+12\Delta_n}+r_n|\langle D\dotw_n,D\dotw_n\rangle_\HH|^2
+r_n|\langle DN_n,DN_n\rangle_\HH|^2
+r_n|\langle D\dotx_n,D\dotx_n\rangle_\HH|^2
\nn\\&&
+r_n^{-\kappa_1}|\langle DM_n,DW_\infty\rangle_\HH|^2
+r_n^{-\kappa_1}|\langle DM_n,DX_\infty\rangle_\HH|^2
+\frac{e_n}{s_n^2}+\frac{3s_n}{2s_n+12\Delta_{X_\infty}}
\eea
\end{en-text}
for $\Delta_n=\Delta_{(cM_n+W_\infty,X_\infty)}$. 
The functionals $e_n$ {\colred and $f_n$} will be specified later. 
%where $c_0$ is a positive constant. 
Let $\psi\in C^\infty(\bbR;[0,1])$ such that $\psi(x)=1$ for $|x|\leq1/2$ and 
$\psi(x)=0$ for $|x|\geq1$. 
Let $\psi_n=\psi(\xi_n)$. Then $\sup_{n\in\bbN}\|\psi_n\|_{\ell-2,p}<\infty$ for every $p>1$.

Denote by $\phi(z;\mu,\Sigma)$ the density function of the normal distribution 
with mean vector $\mu$ and covariance matrix $\Sigma$. 
{\colred We write ${\mathfrak S}_n=1+r_n{\mathfrak S}$. }
Define the function $p_n(z,x)$ by 
\beas 
p_n(z,x) 
&=& 
%E\bigg[\phi(z;W_\infty,G_\infty)\delta_x(X_\infty)\bigg]
%\\&&
%+
E\bigg[{\mathfrak S}_n(\partial_z,\partial_x)^*\bigg\{
\phi(z;W_\infty,G_\infty)\delta_x(X_\infty)\bigg\}\bigg],
\eeas
where $\delta_x(X_\infty)$ is Watanabe's delta function, i.e., the pull-back of 
the delta function $\delta_x$ by $X_\infty$. 
See \cite{IkedaWatanabe1989} for the notion of generalized Wiener functionals and Watanabe's delta function. 
The operation of the adjoint $\varsigma(\partial_z,\partial_x)^*$ 
for a random polynomial symbol $\varsigma(\tti\sfz,\tti\sfx)=\sum_\alpha c_\alpha(\tti\sfz,\tti\sfx)^\alpha$ 
is defined by 
\beas
E\bigg[\varsigma(\partial_z,\partial_x)^*\bigg\{
\phi(z;W_\infty,G_\infty)\delta_x(X_\infty)\bigg\}\bigg]
&=&
\sum_\alpha (-\partial_z,-\partial_x)^\alpha E\bigg[c_\alpha\phi(z;W_\infty,G_\infty)\delta_x(X_\infty)\bigg]. 
\eeas
The function $p_n(z,x)$ is well defined under [C]. %\koko
%Different random symbols can produce the same function $p_n$. 

Given positive numbers $M$ and $\gamma$, denote by $\cale(M,\gamma)$ 
the set of measurable functions 
$f:\bbR^{\check{\sfd}}\to\bbR$ satisfying 
$|f(z,x)|\leq M(1+|z|+|x|)^\gamma$ for all $(z,x)\in\bbR^{\check{\sfd}}$. 
We intend to approximate the joint distribution of $(Z_n,X_n)$ by the density function $p_n(z,x)$. 
The error of the approximation is evaluated by 
the supremum of 
\beas 
\Delta_n(f) 
&=&
\bigg| E\big[f(Z_n,X_n)\big]-\int_{\bbR^{\check{\sfd}}} f(z,x)p_n(z,x)dzdx\bigg|
\eeas
in $f\in\cale(M,\gamma)$.

For $\check{Z}_n=(Z_n,X_n)$, we write $\check{Z}_n^\alpha=Z_n^{\alpha_1}X_n^{\alpha_2}$ for 
$\alpha=(\alpha_1,\alpha_2)\in\bbZ_+^\sfd\times\bbZ_+^{\sfd_1}=\bbZ_+^{\check{\sfd}}$. 
Define $\hat{g}^\alpha_n(\sfz,\sfx)$ by 
\beas 
\hat{g}^\alpha_n(\sfz,\sfx)
&=&
E\big[\psi_n\check{Z}^\alpha_n\exp\big(Z_n[\tti\sfz]+X_n[\tti\sfx]\big)\big]
\eeas
for $\sfz\in\bbR^\sfd$ and $\sfx\in\bbR^{\sfd_1}$. 
Define $g_n^\alpha(z,x)$ by 
\beas
g_n^\alpha(z,x) 
&=& 
\frac{1}{(2\pi)^{\check{\sfd}}}\int_{\bbR^{\check{\sfd}}}
\exp\big(-z[\tti\sfz]-x[\tti\sfx]\big)\>\hat{g}^\alpha_n(\sfz,\sfx)d\sfz d\sfx
\eeas
if the integral exists. 

In the notation of Section \ref{170805-10}, 
\beas 
\hat{g}^\alpha_n(\sfz,\sfx) &=& E\big[e^{\lambda_n(1;\sfz,\sfx)} \psi_n\check{Z}^\alpha_n\big]
\yeq
\varphi_n\big(1,\sfz,\sfx;\psi_n\check{Z}^\alpha_n\big)
\yeq \varphi_n\big(1;\psi_n\check{Z}^\alpha_n\big).
\eeas
Moreover, 
\beas 
\hat{g}^\alpha_n(\sfz,\sfx) &=& \partialbs^\alpha\hat{g}^0_n(\sfz,\sfx)
\yeq \partialbs^\alpha\varphi_n(1;\psi_n)
\eeas
for $\alpha\in\bbZ_+^{\check{\sfd}}$, where 
$\partialbs^\alpha=\partialbs_\sfz^{\alpha_1}\partialbs_\sfx^{\alpha_2}$.

Let 
\beas 
h^0_n(z,x) 
&=&
E\big[\psi_n\phi(z;W_\infty,G_\infty)\delta_x(X_\infty)\big]
\\&&
+{\colred r_n}E\big[
%\big({\mathfrak S}_n(\partial_z,\partial_x)-1\big)^*
{\colred{ \mathfrak S}(\partial_z,\partial_x)^*}
\big\{\phi(z;W_\infty,G_\infty)\delta_x(X_\infty)\big\}\big]
\eeas
and let 
\beas 
h^\alpha_n(z,x) 
&=&
(z,x)^\alpha h^0_n(z,x)
\eeas
for $\alpha\in\bbZ_+^{\check{\sfd}}$. 
Let 
\beas 
\hat{h}^\alpha_n(\sfz,\sfx) 
&=& 
\int_{\bbR^{\check{\sfd}}}e^{z[\tti\sfz]+x[\tti\sfx]}h^\alpha(z,x)dzdx.
\eeas
Then 
\beas
\hat{h}^\alpha_n(\sfz,\sfx) 
&=&
\partialbs^\alpha E\big[\Psi(\sfz,\sfx)\psi_n\big]
+{\colred r_n}
\partialbs^\alpha E\big[\Psi(\sfz,\sfx)
{\colred {\mathfrak S}(\sfz,\sfx)}
%({\mathfrak S}_n(\sfz,\sfx)-1)
\big].
\eeas

Let 
$ \Lambda_n(d)=\{u\in\bbR^d;|u|\leq {\colorr r_n^{-q}}\}$, {\colorr where $q\in(0,1/2)$. 
Let $\theta=\half-q$. }

\begin{lemme}\label{170810-5}
Suppose that $[C]$ is fulfilled. Then 
\bd
\im[(a)] For each $(\sfz,\sfx)\in\bbR^\csfd$ and $\alpha\in\bbZ_+^{\check{\sfd}}$, 
\bea\label{170811-41} 
\hat{g}^\alpha_n(\sfz,\sfx)-\hat{h}^\alpha_n(\sfz,\sfx)
&=& o(r_n). 
\eea
\im[(b)] %{\colorr There exists $\ep>0$ such that,}
For every $\alpha\in\bbZ_+^{\check{\sfd}}$, 
\bea\label{170810-7}
\sup_n\sup_{(\sfz,\sfx)\in\Lambda_n(\csfd)}|(\sfz,\sfx)|^{\check{\sfd}+1}
r_n^{-1}\big|\hat{g}^\alpha_n(\sfz,\sfx)-\hat{h}^\alpha_n(\sfz,\sfx)\big|
&<& \infty.
\eea
\ed
\end{lemme}
\proof 
First we estimate 
\beas 
\varphi_n(\theta,\sfz,\sfx;\Xi)
&=& 
E\big[e^{\lambda_n(\theta,\sfz,\sfx)}\Xi]
\eeas
for $\Xi\in\bbD^{{\sf k},\infty}=\cap_{p>1}\bbD^{{\sf k},p}$, ${\sf k}\leq\csfd+6$. %for a positive integer ${\ell}$. 
%Denote by $\sigma_{F}$ the Malliavin covariance matrix of a vector-valued functional $F$, 
%and let $\Delta_F=\det\sigma_{F}$. 
Let 
\beas 
\check{M}_n(\theta) &=& \big(M_n+\theta^{-1}W_n(\theta)+r_nN_n,X_n(\theta)\big)
\eeas
for $\theta\in(0,1]$. 
%Let $\ep$ be a positive number. 
Suppose that there exists $\theta_0\in[0,1)$ such that 
\bea\label{170806-5} 
%\sup_{\theta\in\big(\sqrt{1-r_n^\ep},1\big]}
\sup_{\theta\in(\theta_0,1]}
E\bigg[\Delta_{\check{M}_n(\theta)}^{-p}
{\bf 1}_{\{\sum_{j=1}^{\sf k}\|D^j\Xi\|_{\mfh^{\otimes j}}>0\}}
\bigg] &<& \infty
\eea
and that 
\bea\label{170806-6}
%\sup_{\theta\in\big(0,\sqrt{1-r_n^\ep}\big]}
\sup_{\theta\in(0,\theta_0]}
E\bigg[\Delta_{X_n(\theta)}^{-p}
{\bf 1}_{\{\sum_{j=1}^{\sf k}\|D^j\Xi\|_{\mfh^{\otimes j}}>0\}}\bigg] 
&<& \infty
\eea
for every $p>1$. 

By definition, 
\beas 
\varphi_n(\theta,\sfz,\sfx;\Xi) &=& E\big[e^{\check{M}_n{\colred (\theta)}[\theta\tti\sfz,\tti\sfx]}
e^{2^{-1}(1-\theta^2)G_\infty[(\tti\sfz)^{\otimes2}]}
\>\Xi]. 
\eeas
Use 
$M_n, W_n, W_\infty, N_n, N_\infty\in\bbD^{{{\sf k}+1},\infty}(\bbR^{\sfd})$, 
$X_n, X_\infty\in\bbD^{{{\sf k}+1},\infty}(\bbR^{\sfd_1})$, and 
$\det G_\infty^{-1}\in L^{\infty-}$,  
then 
with the IBP-formula ${\sf k}$-times with respect to $\check{M}_n(\theta)$, 
%where ${\sf k}$ is  number not greater than $\ell$, 
we obtain 
\beas 
\big|\varphi_n(\theta,\sfz,\sfx;\Xi)\big|
&\leq& 
\big|(\theta\sfz,\sfx)\big|^{-{\sf k}}
E\bigg[\exp\bigg(-\half(1-\theta^2)G_\infty[z^{\otimes2}]\bigg)\times |A_n(\theta,\sfz;\Xi)|\bigg]
%\>\leq\>
%\big|(\theta\sfz,\sfx)\big|^{-{\sf k}}E\big[ |A(\theta,\sfz;\Xi)|\big]
\eeas
for $\theta\in(\theta_0,1]$, 
the functional 
$A_n(\theta,\sfz;\Xi)$ is linear in $\Xi$, and 
the expectation on the right-hand side is dominated by a polynomial in 
\beas 
\|G_\infty^{-1}\|_{{\colorg p}},\>
\|G_\infty\|_{{\sf k},p},\>\|\check{M}_n(\theta)\|_{{\sf k}+1,p},\>
\|\Delta_{\check{M}_n(\theta)}^{-1}
{\bf 1}_{\{\sum_{j=1}^{\sf k}\|D^j\Xi\|_{\HH^{\otimes j}}>0\}}
\|_p,\>\|\Xi\|_{{\sf k},p}
\eeas
for some $p>1$ 
uniformly in $\theta\in(0,1]$, $(\sfz,\sfx)\in\bbR^{\check{\sfd}}$ and $n\in\bbN$. 
For it, we may add an independent Gaussian variable to $\check{M}_n(\theta)$ and shrink its variance after 
integration-by-parts. 
Here we remark that 
\beas 
\sup_{\theta,\sfz}\bigg[
\exp\bigg(-\half(1-\theta^2)G_\infty[z^{\otimes2}]\bigg)\times
\|(1-\theta^2)D^{j_1}G_\infty[z^{\otimes2}]\|_{\mfh^{\otimes j_1}}\cdots
\|(1-\theta^2)D^{j_m}G_\infty[z^{\otimes2}]\|_{\mfh^{\otimes j_m}}\bigg]
&\in& L^{\infty-},
\eeas
which is a consequence of $L^p$ integrability of $G_\infty^{-1}$ for sufficiently large $p$. 
[This estimate is possible only when $(\sfz,\sfx)$ appears with factor $1-\theta^2$. 
Otherwise, even though the non-degeneracy of $G_\infty$ is used, the factor $(1-\theta^2)^{-1}$ 
would appear and the estimation failed for $\theta$ near $1$. ] 
Therefore
\beas
\sup_{\theta\in(\theta_0,1]}\big|\varphi_n(\theta,\sfz,\sfx;\Xi)\big|
&\simleq& 
\big|(\sfz,\sfx)\big|^{-{\sf k}}
\eeas

For $\theta\in(0,\theta_0)$, %$\theta\leq\sqrt{1-r_n^\ep}$, 
we use nondegeneracy of $X_n(\theta)$. 
Applying ingeration-by-parts with respect to $X_n(\theta)$ to 
\beas 
\varphi_n(\theta,\sfz,\sfx;\Xi) &=& 
E\bigg[e^{X_n(\theta)[\tti\sfx]}
\exp\bigg(2^{-1}(1-\theta^2)G_\infty[(\tti\sfz)^{\otimes2}]
+{\theta M_n[\tti\sfz]+W_n(\theta)[\tti\sfz]+\theta r_nN[\tti\sfz]}\bigg)\>\Xi\bigg], 
\eeas
we obain 
\beas 
(\tti\sfx)^{\alpha_2}\varphi_n(\theta,\sfz,\sfx;\Xi) &=& 
E\big[
e^{\lambda_n(\theta;\sfz,\sfx)}B_{n,\alpha_2}(\theta,\sfz;\Xi)\big]
\eeas
for some function $B_{n,\alpha_2}(\theta,\sfz;\Xi)$ for 
$\theta\in(0,\theta_0)$ and 
$\alpha_2\in\bbZ_+^{\sfd_1}$ with $|\alpha_2|={\sf k}$. 
For every $L>0$, 
the $L^1$-norm of the functional $B_{n,\alpha_2}(\theta,\sfz;\Xi)$ is dominated by 
a polynomial of 
\beas &&
\|X_n(\theta)\|_{{\sf k}+1,p}, \quad
\|\Delta_{X_n(\theta)}^{-1}{\bf 1}_{\{\sum_{j=1}^{\sf k}\|D^j\Xi\|_{\HH^{\otimes j}}>0\}}\|_p,\quad
%\exp\bigg(2^{-1}(1-\theta^2)G_\infty[(\tti\sfz)^{\otimes2}]\bigg),\quad
\|G_\infty^{-1}\|_p,\quad
{\colred \|G_\infty\|_{{\sf k},p},\quad}
\\&&
\|M_n\|_{{\sf k},p},\quad
\|W_n(\theta)\|_{{\sf k},p},\quad
\|N_n\|_{{\sf k},p},\quad
\|\Xi\|_{{\sf k},p},\quad
(1+|\sfz|)^{-L}
\eeas
uniformly in $\theta\in(0,\theta_0)$ and $n\in\bbN$, if we take a sufficiently large $p$. 
Therefore we have 
\beas
\big|\varphi_n(\theta,\sfz,\sfx;\Xi) \big|
&\simleq& 
|(\sfz,\sfx)|^{-{\sf k}}
\eeas
uniformly in $\theta\in(0,\theta_0)$, $(\sfz,\sfx)\in\bbR^{\check{\sfd}}$ and $n\in\bbN$. 
Consequently, we obtained 
\bea\label{170806-8} 
\sup_{\theta\in(0,1]}\big|\varphi_n(\theta,\sfz,\sfx;\Xi)\big|
&\simleq& 
\big|(\sfz,\sfx)\big|^{-{\sf k}}
\eea
uniformly in $(\sfz,\sfx)\in\bbR^{\check{\sfd}}$ and $n\in\bbN$, 
under the assumptions (\ref{170806-5}) and (\ref{170806-6}).

For $\theta\geq|c|$, 
\beas 
\Delta_{\check{M}_n(\theta)} &=& 
\Delta_{(M_n+\theta^{-1}W_\infty,X_\infty)}
+r_nd_n^*(\theta),
\eeas
where $\{d_n^*(\theta);\theta\in[|c|,1],n\in\bbN\}$ is a family of functionals bounded in $\bbD^{\csfd+6,\infty}$. 
Moreover, 
\beas&&
\Delta_{(M_n+\theta^{-1}W_\infty,X_\infty)}
\\&=& 
\theta^{-2d}\det
\left[\begin{array}{cc}
\langle \theta DM_n,\theta DM_n\rangle_\HH+\langle DW_\infty,DW_\infty\rangle_\HH & \langle DX_\infty,DW_\infty\rangle_\HH\\
\langle DW_\infty,DX_\infty\rangle_\HH& \langle DX_\infty,DX_\infty\rangle_\HH
\end{array}
\right]
+r_n^{(1\wedge\kappa_1)/2}{\colred \dot{d}_n(\theta)}
\\&\geq& 
\det
\left[\begin{array}{cc}
\langle c DM_n,c DM_n\rangle_\HH+\langle DW_\infty,DW_\infty\rangle_\HH & \langle DX_\infty,DW_\infty\rangle_\HH\\
\langle DW_\infty,DX_\infty\rangle_\HH& \langle DX_\infty,DX_\infty\rangle_\HH
\end{array}
\right]
+r_n^{(1\wedge\kappa_1)/2}{\colred \dot{d}_n(\theta)}
\\&=&
\Delta_{(cM_n+W_\infty,X_\infty)}+r_n^{(1\wedge\kappa_1)/2}\tilde{d}_n(\theta)
\eeas
for $\theta\in[|c|,1]$, where {\colred $\dot{d}_n(\theta)$ and} $\tilde{d}_n(\theta)$ are 
functionals in $\bbD^{\csfd+6,\infty}$ such that 
\beas 
\sup_{\theta\in[0,1],n\in\bbN}r_n^{-(1\wedge\kappa_1)/2}
{\colred \bigg\{\|\dot{d}_n(\theta)\|_{\csfd+6,p}+}\|\tilde{d}_n(\theta)\|_{\csfd+6,p}
{\colred \bigg\}}
&<&\infty
\eeas 
for every $p>1$. 
Consequently, 
\bea\label{170810-2} 
\Delta_{\check{M}_n(\theta)} &\geq& 
\Delta_{(cM_n+W_\infty,X_\infty)}+r_n^{(1\wedge\kappa_1)/2}d^{**}_n(\theta)
\eea
for $\theta\in[|c|,1]$ and $n\in\bbN$. 
The functional $d^{**}_n(\theta)$ is defined by $d_n^*(\theta)$ and $\tilde{d}_n(\theta)$. 
We define $e_n$ as the sum of squares of the coefficient of the polynomial $d^{**}_n(\theta)$ in $\theta$ and $\theta^{-1}$. 
Then $\|e_n\|_{\csfd+6,p}=O(r_n^{1\wedge\kappa_1})$ for every $p>1$. 
{\colred 
On the other hand, we have an expansion
\beas 
\Delta_{X_n(\theta)} 
&=& 
\Delta_{X_\infty}\big(1+r_n^{1/2}\Delta_{X_\infty}^{-1}\hat{d}_n(\theta)\big)
\eeas
with a functional $\hat{d}_n(\theta)$ such that 
all coefficients of this polynomial in $\theta$ are of $O(r_n^{1/2})$ in $L^{\infty-}$. 
Let $f_n$ be the sum of squares of the coefficients of $\hat{d}_n(\theta)$.
}
By [C] (iv) (b) and the definition of $\psi_n$, 
considering the event $\{\xi_n>1/2\}\supset\{\psi_n<1\}$, we have 
\bea\label{r1}
\|1-\psi_n\|_{\csfd+6,1+2^{-1}\kappa}=O(r_n^{p_1}) 
\eea
for $p_1=(1+\kappa)/(1+2^{-1}\kappa)>1$. 
If $\xi_n\leq1$, then $\inf_{\theta\in[|c|,1]}\Delta_{\check{M}_n(\theta)}\geq s_n/13$ 
and {\colred $\inf_{\theta\in[0,1]}\Delta_{X_n(\theta)}\geq \Delta_{X_\infty}/2$} for large $n$. 
Thus the conditions (\ref{170806-5}) and (\ref{170806-6}) are ensured 
and hence the estimate (\ref{170806-8}) is available 
for various functionals $\Xi$ having a factor related to $\psi_n$, as we will see below. 
%On the other hand, the probability of the event $\{\xi_n\leq1/2\}\subset\{\psi_n=1\}$ is of $1-o(r_n)$.  
Condition (\ref{c1}) gives $L^{\infty-}$-boundedness of $\csfd+6$ derivatives of $\sigma_{M_n}$. 

Condition (\ref{c22}) with (\ref{c1}) implies 
\bea\label{c32}
\|\big\langle DG^{(2)}_n,u_n\big\rangle_\mfh\|_{{\colorr \csfd+5},p}%{{\colorr \ell-3},p}
&=& O(r_n)%{\colorr dasu}%\qquad(k=2,3)
\eea
[This is the only place where the $L^p$ boundedness of $D^{\ell-2}G_n^{(2)}$ is required. 
That is, we will only need that $\|G^{(2)}_n\|_{\ell-3,p}=O(r_n)$ and (\ref{c32}) 
in what follows. ] 
Condition (\ref{c33}) implies 
\bea\label{c43}
\bigg\|\bigg\langle D\bigg(\big\langle DG^{(3)}_n,u_n\big\rangle_\mfh\bigg),u_n\bigg\rangle_\mfh\bigg\|_{\csfd+6,p}
&=& o(r_n)%{\colorr deru}%\qquad(k=2,3)
\eea
for every $p>1$ under (\ref{c1}). 
Condition (\ref{c7}) implies 
\bea\label{c9}
\sum_{{\sf A}=W_\infty,X_\infty}
\bigg\|\bigg\langle D\bigg(\big\langle D\langle D{\sf A},u_n\rangle_\mfh,u_n\big\rangle_\mfh\bigg),u_n\bigg\rangle_\mfh\bigg\|_{\csfd+5,p}
&=& o(r_n)%{\colorr deru}
\eea
under (\ref{c1}). 
Moreover, 
Condition (\ref{c10.1}) implies 
\bea\label{c10.2}
\sum_{{\sf B}=\dotw_n,N_n,\dotx_n}
\bigg\| \bigg\langle D\big\langle D\langle D{\sf B},u_n\rangle_\HH ,u_n\big\rangle_\HH
 ,u_n\bigg\rangle_\HH\bigg\|_{\csfd+5,p}&=& o(1)%{\colorr deru}
\eea
under (\ref{c1}). 

The estimate 
$\|\hat{G}^{(1)}_n(\theta)\|_{\csfd+5,p} = O(r_n)$ for every $p>1$ 
follows from (\ref{c5}), (\ref{c10}) and (\ref{c1}). 
% and the assumptions that $W_\infty\in \bbD^{\csfd+6,\infty}(\bbR^\sfd)$ 
%and $X_\infty\in \bbD^{\csfd+6,\infty}(\bbR^{\sfd_1})$. 
The estimate 
$ \|\check{G}^{(1)}_n\|_{\csfd+5,p} = O(r_n)$ for every $p>1$ 
follows from (\ref{c10}), therefore 
\bea\label{cesta} 
\|G^{(1)}_n(\theta)\|_{\csfd+5,p} &=& O(r_n)
\eea
for every $p>1$. 
We obtain 
\bea\label{cestb1} 
\big\|\big\langle D\hat{G}^{(1)}_n(\theta),u_n\big\rangle_\HH\big\|_{\csfd+5,p} &=& o(r_n) 
\eea
for every $p>1$ 
from (\ref{c7}) and (\ref{c10.1}). 
Estimate 
\bea\label{cestb2} 
\big\|\big\langle D\check{G}^{(1)}_n,u_n\big\rangle_\HH\big\|_{\csfd+5,p} &=& O(r_n)
\eea
for every $p>1$ follows from (\ref{c10}). 

We have 
\beas
\bigg\|\bigg\langle D\big\langle D\hat{G}^{(1)}_n(\theta),u_n\big\rangle_\HH,u_n\bigg\rangle_\HH\bigg\|_{\csfd+5,p} &=& o(r_n) 
\eeas
by (\ref{c9}) and (\ref{c10.2}). 
Follows the estimate 
\beas
\bigg\|\bigg\langle D\big\langle D\check{G}^{(1)}_n,u_n\big\rangle_\HH,u_n\bigg\rangle_\HH\bigg\|_{\csfd+5,p} &=& o(r_n) 
\eeas
from (\ref{c10.1}), so that 
\bea\label{cestc}
\bigg\|\bigg\langle D\big\langle DG^{(1)}_n(\theta),u_n\big\rangle_\HH,u_n\bigg\rangle_\HH\bigg\|_{\csfd+5,p} &=& o(r_n) 
\eea

Since $\varphi_n(0;\psi_n)=E[\Psi(\sfz,\sfx)\psi_n]$, Proposition \ref{170806-1} gives 
\beas 
\hat{g}^\alpha_n(\sfz,\sfx)-\hat{h}^\alpha_n(\sfz,\sfx)
&=&
\partialbs^\alpha R_n(\sfz,\sfx)\yeq \sum_{i=3}^{12}\partialbs^\alpha R^{(i)}_n(\sfz,\sfx).
\eeas
We shall show 
\bea\label{170806-2}
\sup_n\sup_{(\sfz,\sfx)\in\Lambda(\csfd)}
|(\sfz,\sfx)|^{\csfd+1}r_n^{-1}\big|\partialbs^\alpha R^{(i)}_n(\sfz,\sfx)\big|
&<& \infty
\eea
for $i=3,...,12$. 

We remind the representation of $R^{(3)}_n(\sfz,\sfx)=\rho^{(3)}_n(f)$ 
with $f(z,x)=\exp(z[\tti\sfz]+x[\tti\sfx])$. 
There appear 24 terms in this expression and we name them $R[i]$ ($i=1,...,24$). 
We will repeatedly use the inequality (\ref{170806-8}) based on 
integration-by-parts (IBP) to estimate $\partialbs^\alpha R[i]$. 
It should be noted that the factors $(1-\theta^2)G_\infty[\tti\sfz,\cdot]$ comes out but 
they are controlled by $\exp\big(2^{-1}(1-\theta^2)G_\infty[(\tti\sfz)^{\otimes2}]\big)$ if non-degeneracy of $G_\infty$ is used, 
as already mentioned. 
%In this step, boundedness of $\bbD_{\ell,p}$-norms of $W_n$, $N_n$ and $X_n$ in $\bbD_{\ell+1,p}$ is used\footnote{$\psi_n$ requires $(\ell+1)$-th derivatives. } 
%but it is a consequence of (\ref{c10}). 
Estimates of $R[i]$ are as follows. 
\bi
\im $R[1]$. The estimate 
$\sup_{(\sfz,\sfx)\in\Lambda_n(\csfd)}|(\sfz,\sfx)|^{\csfd+1}|\partialbs^\alpha R[1]|=o(r_n)$ follows from $\csfd+2$ times IBP, (\ref{r1}) and (\ref{c1}). 

\im $R[9]$. 
There are three components $R[9,j]$ $(j=1,2,3)$ of $R[9]$ 
corresponding to the decomposition 
\bea\label{170808-1}
\bigg\langle D\big\langle D(G(\sfz)\psi_n),u_n[\tti\sfz]\big\rangle_\HH,u_n[\tti\sfz]\bigg\rangle_\HH
&=&
\bigg\langle D\big\langle DG(\sfz),u_n[\tti\sfz]\big\rangle_\HH,u_n[\tti\sfz]\bigg\rangle_\HH\psi_n
\nn\\&&
+2\big\langle DG(\sfz),u_n[\tti\sfz]\big\rangle_\HH\big\langle D\psi_n,u_n[\tti\sfz]\big\rangle_\HH
\nn\\&&
+G(\sfz)\bigg\langle D\big\langle D\psi_n,u_n[\tti\sfz]\big\rangle_\HH,u_n[\tti\sfz]\bigg\rangle_\HH
\eea
for $G=G^{(3)}_n$. 
The estimate 
$\sup_{(\sfz,\sfx)\in\Lambda_n(\csfd)}|(\sfz,\sfx)|^{\csfd+1}|\partialbs^\alpha R[9,1]|=o(r_n)$ follows from 
{\coloro $\csfd+6$} IBP, (\ref{c43}) and (\ref{r1}). 
Since $\|\big\langle DG^{(3)}_n(\sfz),u_n[\tti\sfz]\big\rangle_\HH\|_{\csfd+4,p}\simleq o(r_n)|\sfz|^4$ 
by (\ref{c33}) and $|(\sfz,\sfx)|\leq r_n^{-q}\leq r_n^{-1/2}$, we may deal with $|(\sfz,\sfx)|^3$ for $R[9,2]$. 
Apply $\csfd+4$ IBP , (\ref{r1}) and (\ref{c1}) 
to obtain $\sup_{(\sfz,\sfx)\in\Lambda_n(\csfd)}|(\sfz,\sfx)|^{\csfd+1}|\partialbs^\alpha R[9,2]|=O(r_n^{p_1})$. 
Similarly, we use $\|G^{(3)}_n(\sfz)\|_{\csfd+4,p}\simleq O(r_n)|\sfz|^3$ by (\ref{c23}), 
$\csfd+4$ IBP, (\ref{r1}) and (\ref{c1}) to obtain $\sup_{(\sfz,\sfx)\in\Lambda_n(\csfd)}|(\sfz,\sfx)|^{\csfd+1}|\partialbs^\alpha R[9,3]|=O(r_n^{p_1})$. 
Thus, $\sup_{(\sfz,\sfx)\in\Lambda_n(\csfd)}|(\sfz,\sfx)|^{\csfd+1}|\partialbs^\alpha R[9]|=o(r_n)$. 

\im $R[2]$. 
We take a way similar to $R[9]$ to show $\sup_{(\sfz,\sfx)\in\Lambda_n(\csfd)}|(\sfz,\sfx)|^{\csfd+1}|\partialbs^\alpha R[2]|=o(r_n)$. 
$R[2,j]$ $(j=1,2,3)$ are defined by (\ref{170808-1}) for $G=G^{(2)}_n$. 
Apply $\csfd+5$ IBP to $R[2,1]$ with (\ref{c42}) and (\ref{r1}). 
$\csfd+3$ IBP to $R[2,2]$ with (\ref{c32}), (\ref{r1}) and (\ref{c1}). 
$\csfd+3$ IBP to $R[2,3]$ with (\ref{c22}), (\ref{r1}) and (\ref{c1}). 

\im $R[16]$. In the same way as for $R[2]$, we can show 
$\sup_{(\sfz,\sfx)\in\Lambda_n(\csfd)}|(\sfz,\sfx)|^{\csfd+1}|\partialbs^\alpha R[16]|=o(r_n)$. 
In this case, decomposing $R[16]$ into $R[16,j]$ $(j=1,2,3)$ by (\ref{170808-1}) for $G=G^{(1)}_n$, 
we apply $\csfd+5$ IBP to $R[16,1]$ with (\ref{cestc}) and (\ref{r1}). 
$\csfd+3$ IBP to $R[16,2]$ with (\ref{cestb1}), (\ref{cestb2}), (\ref{r1}) and (\ref{c1}). 
$\csfd+3$ IBP to $R[	16,3]$ with (\ref{cestc}), (\ref{r1}) and (\ref{c1}). 

\im $R[23]$. 
There are two terms $R[23,i]$ ($i=1,2$) for the decomposition 
\bea\label{170808-5} 
F\big\langle D\big(G\psi_n\big),u_n[\tti\sfz]\big\rangle_\HH
&=&
F\big\langle (DG),u_n[\tti\sfz]\big\rangle_\HH\psi_n
+\ FG\big\langle D\psi_n,u_n[\tti\sfz]\big\rangle_\HH
\eea
%\bea\label{170808-5} 
%\big\langle D\big(G\psi_n\big),u_n[\tti\sfz]\big\rangle_\HH
%&=&
%\big\langle (DG),u_n[\tti\sfz]\big\rangle_\HH\psi_n
%+\ G\big\langle D\psi_n,u_n[\tti\sfz]\big\rangle_\HH
%\eea
for $F=1$ and $G=G^{(3)}_n(\sfz)$. 
Apply $\csfd+5$ IBP, (\ref{c33}) and (\ref{r1}) to $R[23,1]$, and 
$\csfd+3$ IBP, (\ref{c23}) and (\ref{r1}) to $R[23,2]$. 
Then we obtain $\sup_{(\sfz,\sfx)\in\Lambda_n(\csfd)}|(\sfz,\sfx)|^{\csfd+1}|\partialbs^\alpha R[23]|=o(r_n)$. 

\im $R[24]$. There are two terms $R[24,i]$ ($i=1,2$) 
according to (\ref{170808-5}) for $F=1$ and $G=\hat{G}^{(1)}_n(\theta;\sfz,\sfx)$. 
To $R[24,1]$, use $\csfd+4$ IBP with (\ref{cestb1}) and (\ref{r1}). 
To $R[24,2]$, $\csfd+2$ IBP with (\ref{cesta}) and (\ref{r1}). 
Then $\sup_{(\sfz,\sfx)\in\Lambda_n(\csfd)}|(\sfz,\sfx)|^{\csfd+1}|\partialbs^\alpha R[24]|=o(r_n)$. 

\im $R[11]$. By the decomposition (\ref{170808-5}) 
%\bea\label{170808-5} 
%F\big\langle D\big(G\psi_n\big),u_n[\tti\sfz]\big\rangle_\HH
%&=&
%F\big\langle (DG),u_n[\tti\sfz]\big\rangle_\HH\psi_n
%+\ FG\big\langle D\psi_n,u_n[\tti\sfz]\big\rangle_\HH
%\eea
for $F=G=G^{(3)}_n(\sfz)$, 
we have two terms $R[11,i]$ ($i=1,2$) as the components of $R[11]$. 
Apply $\csfd+6$ IBP with (\ref{c33}) and (\ref{r1}) to $R[11,1]$. 
Apply $\csfd+4$ IBP with (\ref{c23}) and (\ref{r1}) to $R[11,2]$. 
Then we have $\sup_{(\sfz,\sfx)\in\Lambda_n(\csfd)}|(\sfz,\sfx)|^{\csfd+1}|\partialbs^\alpha R[11]|=o(r_n)$. 

\im $R[10]$. This case is similar to $R[11]$. 
There appear two terms $R[10,i]$ ($i=1,2$) by (\ref{170808-5}) 
for $F=G^{(2)}_n(\sfz)$ and $G=G^{(3)}_n(\sfz)$.  
Then we obtain 
$\sup_{(\sfz,\sfx)\in\Lambda_n(\csfd)}|(\sfz,\sfx)|^{\csfd+1}|\partialbs^\alpha R[10]|=o(r_n)$ 
by applying $\csfd+5$ IBP with (\ref{c33}), (\ref{c22}) and (\ref{r1}) to $R[10,1]$, and 
by $\csfd+3$ IBP with (\ref{c22}), (\ref{c23}) and (\ref{r1}) to $R[10,2]$. 

\im $R[12]$. This case is similar to $R[10]$. 
There appear two terms $R[10,i]$ ($i=1,2$) by (\ref{170808-5}) 
for $F=G^{(1)}_n(\theta;\sfz,\sfx)$ and $G=G^{(3)}_n(\sfz)$.  
Then 
$\sup_{(\sfz,\sfx)\in\Lambda_n(\csfd)}|(\sfz,\sfx)|^{\csfd+1}|\partialbs^\alpha R[12]|=o(r_n)$. 
For that, apply $\csfd+5$ IBP with (\ref{c33}), (\ref{cesta}) and (\ref{r1}) to $R[12,1]$, and 
by $\csfd+3$ IBP with (\ref{cesta}), (\ref{c23}) and (\ref{r1}) to $R[12,2]$. 

\im $R[4]$. %This is similar to the case $R[11]$
There appear two terms $R[4,i]$ ($i=1,2$) by (\ref{170808-5}) 
for $F=G^{(3)}_n(\sfz)$ and $G=G^{(2)}_n(\sfz)$. 
We obtain 
 $\sup_{(\sfz,\sfx)\in\Lambda_n(\csfd)}|(\sfz,\sfx)|^{\csfd+1}|\partialbs^\alpha R[4,1]|=O(r_n^{1+2\theta})$ 
 by applying $\csfd+5$ IBP with (\ref{c32}) and (\ref{r1}). 
 In this case, the factor $|(\sfz,\sfx)|^2$ is evaluated by $r_n^{-2q}$. 
Apply $\csfd+3$ IBP with (\ref{c22}), (\ref{c23}) and (\ref{r1}) to $R[4,2]$ to obtain 
$\sup_{(\sfz,\sfx)\in\Lambda_n(\csfd)}|(\sfz,\sfx)|^{\csfd+1}|\partialbs^\alpha R[4,2]|=O(r_n^{p_1})$, therefor 
$\sup_{(\sfz,\sfx)\in\Lambda_n(\csfd)}|(\sfz,\sfx)|^{\csfd+1}|\partialbs^\alpha R[4]|=O(r_n^{p_2})$, 
where $p_2=p_1\wedge(1+2\theta)$. 

\im $R[3]$. This is similar to the case $R[4]$. 
There appear two terms $R[3,i]$ ($i=1,2$) by (\ref{170808-5}) 
for $F=G=G^{(2)}_n(\sfz)$. 
We obtain 
$\sup_{(\sfz,\sfx)\in\Lambda_n(\csfd)}|(\sfz,\sfx)|^{\csfd+1}|\partialbs^\alpha R[3]|=O(r_n^{p_2})$ 
by applying $\csfd+4$ IBP with (\ref{c32}) and (\ref{r1}) to $R[3,1]$, and 
by $\csfd+2$ IBP with (\ref{c22}) and (\ref{r1}) to $R[3,2]$. 

\im $R[5]$. Similar to the case $R[3]$. 
There appear two terms $R[5,i]$ ($i=1,2$) by (\ref{170808-5}) 
for $F=G^{(1)}_n(\theta;\sfz,\sfx)$ and $G=G^{(2)}_n(\sfz)$. 
We obtain 
$\sup_{(\sfz,\sfx)\in\Lambda_n(\csfd)}|(\sfz,\sfx)|^{\csfd+1}|\partialbs^\alpha R[5]|=O(r_n^{p_2})$ 
by applying $\csfd+4$ IBP with (\ref{c32}), (\ref{cesta}) and (\ref{r1}) to $R[3,1]$, and 
by $\csfd+2$ IBP with (\ref{c22}), (\ref{cesta}) and (\ref{r1}) to $R[3,2]$. 

\im $R[18]$. Similar to $R[4]$. 
There are two terms $R[18,i]$ ($i=1,2$) by (\ref{170808-5}) 
for $F=G^{(3)}_n(\sfz)$ and $G=G^{(1)}_n(\theta;\sfz,\sfx)$. 
Then 
$\sup_{(\sfz,\sfx)\in\Lambda_n(\csfd)}|(\sfz,\sfx)|^{\csfd+1}|\partialbs^\alpha R[18]|=O(r_n^{p_2})$ follows from 
$\csfd+5$ IBP with (\ref{cestb1}), (\ref{cestb2}), (\ref{c23}) and (\ref{r1}) to $R[18,1]$, and also 
$\csfd+3$ IBP with (\ref{cesta}), (\ref{c23}) and (\ref{r1}) to $R[18,2]$. 

\im $R[17]$. Similar to $R[18]$. 
There are two terms $R[17,i]$ ($i=1,2$) by (\ref{170808-5}) 
for $F=G^{(2)}_n(\sfz)$ and $G=G^{(1)}_n(\theta;\sfz,\sfx)$. 
Then 
$\sup_{(\sfz,\sfx)\in\Lambda_n(\csfd)}|(\sfz,\sfx)|^{\csfd+1}|\partialbs^\alpha R[17]|=O(r_n^{p_2})$ follows from 
$\csfd+4$ IBP with (\ref{cestb1}), (\ref{cestb2}), (\ref{c22}) and (\ref{r1}) to $R[17,1]$, as well as  
$\csfd+2$ IBP with (\ref{cesta}), (\ref{c22}) and (\ref{r1}) to $R[17,2]$. 

\im $R[19]$. Similar to $R[17]$. 
Two terms $R[19,i]$ ($i=1,2$) by (\ref{170808-5}) 
for $F=G^{(1)}_n(\theta';\sfz,\sfx)$ and $G=G^{(1)}_n(\theta;\sfz,\sfx)$. 
Then 
$\sup_{(\sfz,\sfx)\in\Lambda_n(\csfd)}|(\sfz,\sfx)|^{\csfd+1}|\partialbs^\alpha R[19]|=O(r_n^{p_2})$, which follows from 
$\csfd+4$ IBP with (\ref{cestb1}), (\ref{cestb2}), (\ref{cesta}) and (\ref{r1}) to $R[19,1]$, as well as  
$\csfd+2$ IBP with (\ref{cesta}) and (\ref{r1}) to $R[19,2]$. 

\im $R[14]$. One factor $|(\sfz,\sfx)|$ is cancelled by $r_n^{1/2}$ offered by $\big(G^{(3)}_n(\sfz)\big)^2$. 
We apply $\csfd+6$ IBP with (\ref{c23}) and (\ref{r1})
to obtain $\sup_{(\sfz,\sfx)\in\Lambda_n(\csfd)}|(\sfz,\sfx)|^{\csfd+1}|\partialbs^\alpha R[14]|=O(r_n^{3/2})$. 

\im $R[13]$. Similar to $R[14]$. 
Cancelling one factor $|(\sfz,\sfx)|$ by $r_n^{1/2}$ in $G^{(2)}_n(\sfz)G^{(3)}_n(\sfz)$, 
we apply $\csfd+5$ IBP with (\ref{c22}), (\ref{c23}) and (\ref{r1}) 
to show $\sup_{(\sfz,\sfx)\in\Lambda_n(\csfd)}|(\sfz,\sfx)|^{\csfd+1}|\partialbs^\alpha R[13]|=O(r_n^{3/2})$. 

\im $R[7]$. It is essentially the same as $R[13]$. Therefore 
$\sup_{(\sfz,\sfx)\in\Lambda_n(\csfd)}|(\sfz,\sfx)|^{\csfd+1}|\partialbs^\alpha R[7]|=O(r_n^{3/2})$. 

\im $R[15]$. Similar to $R[14]$. 
Cancelling one factor $|(\sfz,\sfx)|$ by $r_n^{1/2}$ in $G^{(1)}_n(\theta;\sfz,\sfx)G^{(3)}_n(\sfz)$, 
we apply $\csfd+5$ IBP with (\ref{cesta}), (\ref{c23}) and (\ref{r1}) 
to show $\sup_{(\sfz,\sfx)\in\Lambda_n(\csfd)}|(\sfz,\sfx)|^{\csfd+1}|\partialbs^\alpha R[15]|=O(r_n^{3/2})$. 

\im $R[21]$. Essentially same as $R[15]$, therefore 
$\sup_{(\sfz,\sfx)\in\Lambda_n(\csfd)}|(\sfz,\sfx)|^{\csfd+1}|\partialbs^\alpha R[21]|=O(r_n^{3/2})$. 

\im $R[6]$. One factor cancellation and $\csfd+4$ IBP with (\ref{c22}) and (\ref{r1}) 
give $\sup_{(\sfz,\sfx)\in\Lambda_n(\csfd)}|(\sfz,\sfx)|^{\csfd+1}|\partialbs^\alpha R[6]|=O(r_n^{3/2})$. 

\im $R[8]$. Similarly to $R[6]$, 
$\csfd+4$ IBP with (\ref{cesta}), (\ref{c22}) and (\ref{r1}) gives 
$\sup_{(\sfz,\sfx)\in\Lambda_n(\csfd)}|(\sfz,\sfx)|^{\csfd+1}|\partialbs^\alpha R[8]|=O(r_n^{3/2})$. 

\im $R[20]$. This is essentially equivalent to $R[8]$. 
$\sup_{(\sfz,\sfx)\in\Lambda_n(\csfd)}|(\sfz,\sfx)|^{\csfd+1}|\partialbs^\alpha R[20]|=O(r_n^{3/2})$. 

\im $R[22]$. Apply $\csfd+4$ IBP with (\ref{cesta}) and (\ref{r1}) to obtain 
$\sup_{(\sfz,\sfx)\in\Lambda_n(\csfd)}|(\sfz,\sfx)|^{\csfd+1}|\partialbs^\alpha R[22]|=O(r_n^{3/2})$. 
\ei
In conclusion, 
\bea\label{17-811-21}
\sup_{(\sfz,\sfx)\in\Lambda(\csfd)}
|(\sfz,\sfx)|^{\csfd+1}\big|\partialbs^\alpha R^{(i)}_n(\sfz,\sfx)\big|
&=& 
o(r_n)
\eea
for $i=3$
and in particular (\ref{170806-2}) is valid for $i=3$. 

It is possible to obtain 
$\sup_{(\sfz,\sfx)\in\Lambda(\csfd)}
|(\sfz,\sfx)|^{\csfd+1}\big|\partialbs^\alpha R^{(i)}_n(\sfz,\sfx)\big|=O(r_n^{p_1})$, 
and hence 
(\ref{170806-2}) for $i=4$ with $\csfd+1$ IBP, (\ref{c10}) and (\ref{r1}). 
Each $R^{(i)}_n(\sfz,\sfx)$ $(i=5,...,12)$ is a difference of two terms. 
We will estimate each term separately. 
Apply ${\coloro \csfd+\beta_x+1}$ IBP with respect to $X_\infty$ 
to $\varsigma$ and use $\exp\big(2^{-1}G_\infty[(\tti\sfz)^{\otimes2}]\big)$ 
with non-degeneracy of $G_\infty$ (without IBP) 
to obtain 
\beas 
\sup_{(\sfz,\sfx)\in\bbR^{\csfd}}
|(\sfz,\sfx)|^{\csfd+1}
\big|E\big[\Psi(\sfz,\sfx){\mathfrak \varsigma}({\colred\tti}\sfz,{\colred\tti}\sfx)]\big|
&<&
\infty
\eeas
for $\varsigma={\mathfrak S}^{(3,0)}, {\mathfrak S}_{0}^{(2,0)}, 
{\mathfrak S}^{(2,0)}, {\mathfrak S}^{(1,1)}, {\mathfrak S}^{(1,0)}, {\mathfrak S}^{(0,1)}, 
{\mathfrak S}_{1}^{(2,0)}$ and ${\mathfrak S}_{1}^{(1,1)}$. 
Remark that we need 
$X_\infty\in\bbD^{\csfd+\beta_x+2,\infty}(\bbR^{\sfd_1})$, 
$W_\infty\in\bbD^{\csfd+\beta_x+1,\infty}(\bbR^\sfd)$ and 
$G_\infty\in\bbD^{\csfd+\beta_x+1,\infty}(\bbR^\sfd\otimes_+\bbR^\sfd)$ 
in this procedure. 
To estimate each first term, 
$\csfd+4=\ell-4$ IBP with respect to $X_\infty$ 
is sufficient because the degree of each $\varsigma$ is not greater than three. 
For that, Conditions (\ref{c1}), (\ref{c22}), (\ref{c23}), (\ref{c5}) and (\ref{c10}) work 
together with the factor $\exp\big(2^{-1}G_\infty[(\tti\sfz)^{\otimes2}]\big)$ 
with non-degeneracy of $G_\infty$. 
In this way, we obtain 
\beas 
\sup_n\sup_{(\sfz,\sfx)\in\bbR^{\csfd}}
|(\sfz,\sfx)|^{\csfd+1}
\big|E\big[\Psi(\sfz,\sfx)\varsigma_n({\colred\tti}\sfz,{\colred\tti}\sfx)]\big|
&<&
\infty
\eeas
for $\varsigma_n={\mathfrak S}_n^{(3,0)}, {\mathfrak S}_{0,n}^{(2,0)}, 
{\mathfrak S}_n^{(2,0)}, {\mathfrak S}_n^{(1,1)}, {\mathfrak S}_n^{(1,0)}, {\mathfrak S}_n^{(0,1)}, 
{\mathfrak S}_{1,n}^{(2,0)}$ and ${\mathfrak S}_{1,n}^{(1,1)}$. 
Consequently, (\ref{170806-2}) was verified for $i=5,...,12$. 

Furthermore,  [C] (iii)$^\flat$ and (\ref{r1}) gives $\partialbs^\alpha R^{(i)}_n(\sfz,\sfx)=o(r_n)$ for $i=5,...,12$. 
%Suppose that $\phi_n\in C^\infty(\bbR^\sfd\otimes\bbR^\sfd)\times\bbR^\csfd;[0,1])$ satisfy $\phi_n=1$ on 
Now it suffices to show that [C] (iii) implies [C] (iii)$^\flat$. 
For $\eta>0$, let 
\beas 
F^\eta_n(\sfz',\sfx')
&=&
E\big[\Psi(\sfz',\sfx')\psi\big(\eta(|G_\infty|+|W_\infty|+|X_\infty|)\big)
\bar{{\mathfrak T}}_n({\colred\tti}\sfz',{\colred\tti}\sfx')\big]
\qquad ((\sfz',\sfx')\in\bbC^\csfd)
\eeas
for $n\in\bbN\cup\{\infty\}$, where $\bar{{\mathfrak T}}_\infty={\mathfrak T}$. 
$F^\eta_n$ are analytic functions of $(\sfz',\sfx')$ for $\eta>0$ and $n\in\bbN\cup\{\infty\}$. 
Let 
\beas 
F^0_n(\sfz,\sfx)
&=&
E\big[\Psi(\sfz,\sfx)
\bar{{\mathfrak T}}_n({\colred\tti}\sfz,{\colred\tti}\sfx)\big]
\qquad ((\sfz,\sfx)\in\bbR^\csfd)
\eeas
for $n\in\bbN\cup\{\infty\}$. 
%Let $(\sfz,\sfx)\in\bbR^\csfd$. 
Then $\partialbs^\alpha F^\eta_n$ on $\bbR^\csfd$ is explicitly expressed by 
\beas 
\partialbs^\alpha F^\eta_n(\sfz,\sfx)
&=& 
E\big[\partialbs^\alpha\{\Psi(\sfz,\sfx)
\bar{{\mathfrak T}}_n({\colred\tti}\sfz,{\colred\tti}\sfx)\}
\psi\big(\eta(|G_\infty|+|W_\infty|+|X_\infty|)\big)
\big]
\eeas
for $\eta\geq0$, $n\in\bbN\cup\{\infty\}$ and $(\sfz,\sfx)\in\bbR^\csfd$. 
Remark that the differential operator $\partialbs^\alpha$ is in the real domain, so that 
this equation is valid even for $\eta=0$. 
On the other hand, $\partial_{(\sfz',\sfx')}^\alpha F^0(\sfz',\sfx')$ is not defined. 
Fix $(\sfz,\sfx)\in\bbR^\csfd$ and $\alpha\in\bbZ_+^\csfd$. Let $\ep>0$. 
Then there exists $\eta>0$ such that 
\bea\label{170812-1} 
\sup_{n\in\bbN\cup\{\infty\}}\sum_{a=0,\alpha}|\partialbs^aF^\eta_n(\sfz,\sfx)-\partialbs^aF^0_n(\sfz,\sfx)|&<&\ep
\eea
For $\eta>0$, the gap $F_n^\eta(\sfz',\sfx')-F_\infty^\eta(\sfz',\sfx')\to0$ locally uniformly as $n\to\infty$ 
because the coefficients of $\bar{\mathfrak T}_n-{\mathfrak T}$ converge to zero in $L^p$ for some $p>1$. 
Chauchy's integral formula for multivariate analytic functions ensures the convergence 
\bea\label{170812-2}
\partialbs^\alpha F^\eta_n(\sfz,\sfx) \to \partialbs^\alpha F^\eta_\infty(\sfz,\sfx)\qquad(n\to\infty)
\eea
for every $\eta>0$. 
Then (\ref{170812-1}) and (\ref{170812-2}) give 
$
\limsup_{n\to\infty}\big|\partialbs^\alpha F^0_\infty(\sfz,\sfx)-\partialbs^\alpha F^0_n(\sfz,\sfx)\big|
< 2\ep,
$
and hence 
\beas 
\lim_{n\to\infty}\partialbs^\alpha F^0_n(\sfz,\sfx)
&=&\partialbs^\alpha F^0_\infty(\sfz,\sfx).
\eeas 
By the equality in [C] (iii) (b), we obtain 
\beas 
\lim_{n\to\infty}
\partialbs^\alpha E\big[\Psi(\sfz,\sfx){\mathfrak T}_n({\colred\tti}\sfz,{\colred\tti}\sfx)\big]
&=& 
\partialbs^\alpha E\big[\Psi(\sfz,\sfx)\mathfrak T({\colred\tti}\sfz,{\colred\tti}\sfx)\big], 
\eeas
that is [C] (iii)$^\flat$. 
This completes the proof of Lemma \ref{170810-5}. 
\qed\onelineskip

The following is a slightly different set of conditions. % for [C]. 

\bd
\im[[C$^\natural$\!\!]]
{\bf (i)}  [C] (i) holds. 
\bd
\im[\hspace{-2mm}(ii)] (\ref{c1}), (\ref{c22}), (\ref{c23}), (\ref{c5}) and (\ref{c10}) holds for every $p>1$. 
Furthermore, there exists a positive constant $\kappa$ such that 
the following estimates hold: %{\colorr still need to check!}
\bea\label{c33n}
\|\big\langle DG^{(3)}_n,u_n\big\rangle_\mfh\|_{{\colorr \ell-1},p}
&=& O(r_n^{1+\kappa})%\qquad(k=2,3)
\eea
\bea\label{c42n}
\bigg\|\bigg\langle D\bigg(\big\langle DG^{(2)}_n,u_n\big\rangle_\mfh\bigg),u_n\bigg\rangle_\mfh\bigg\|_{\ell-3,p}
&=& O(r_n^{1+\kappa})%\qquad(k=2,3)
\eea
\bea\label{c7n}
\sum_{{\sf A}=W_\infty,X_\infty}
\|\big\langle D\langle D{\sf A},u_n\rangle_\mfh,u_n\big\rangle_\mfh\|_{\ell-2,p}
&=& O(r_n^{1+\kappa})
\eea
\bea\label{c10.1n}
\sum_{{\sf B}=\dotw_n,N_n,\dotx_n}
\bigg\| \big\langle D\langle D{\sf B},u_n\rangle_\HH ,u_n\big\rangle_\HH\bigg\|_{\ell-2,p}&=& O(r_n^{\kappa})
\eea
%
%
%!TEX encoding = UTF-8 Unicode
\im[\hspace{-2mm}(iii)] [C] (iii) holds. 
\im[\hspace{-2mm}(iv)] 
\hspace{0.5mm} {\bf (a)} $G_\infty^{-1}\in L^{\infty-}$. 
\bd\im[(b)]
There exists $\kappa>0$ such that 
\beas 
P\big[\Delta_{(M_n+W_\infty,X_\infty)}<s_n\big] &=& O(r_n^{1+\kappa})
\eeas
%and that 
%\beas 
%P\big[\Delta_{X_\infty}<s_n\big] &=& O(r_n^{1+\kappa}){\colorr remove}
%\eeas
for some positive random variables $s_n\in\bbD^{\ell-2,\infty}$ satisfying 
$\sup_{n\in\bbN}(\|s_n^{-1}\|_p+\|s_n\|_{\ell-2,p})<\infty$ for every $p>1$. 
\ed
\ed
\ed
\onelineskip
%%%

The functional $\psi_n$ is re-defined by $\psi_n=\psi(\xi_n)$ with 
{\colred 
\bea\label{170810-31} 
\xi_n 
&=& 
\frac{3s_n}{2s_n+12\Delta_n}
+\frac{e_n}{s_n^2}
+\frac{f_n}{\Delta_{X_\infty}^2}
\eea
}
\begin{en-text}
\bea\label{170810-31} 
\xi_n 
&=& 
\frac{3s_n}{2s_n+12\Delta_n}+r_n|\langle D\dotw_n,D\dotw_n\rangle_\HH|^2
+r_n|\langle DN_n,DN_n\rangle_\HH|^2
+r_n|\langle D\dotx_n,D\dotx_n\rangle_\HH|^2
\nn\\&&
%+r_n^{\ep}|W_\infty|^2
+\frac{e_n}{s_n^2}
+\frac{3s_n}{2s_n+12\Delta_{X_\infty}}
\eea
\end{en-text}
for $\Delta_n=\Delta_{(M_n+W_\infty,X_\infty)}$ this time. 
%A positive number.$\ep$ and 
{\colred The functional $f_n$ is defined as before, and}
$e_n$ will be specified in the proof of the following lemma. 

\begin{lemme}\label{170810-6} 
Under $[C^\natural]$, 
%the inequality $($\ref{170810-7}$)$ is valid 
%for every $\alpha\in\bbZ_+^{\check{\sfd}}$ and $\psi_n$ given by $($\ref{170810-31}$)$. 
the properties $($a$)$ and $($b$)$ of Lemma \ref{170810-5} hold true. 
\end{lemme}
\proof 
The plot of the proof is quite similar to that of Lemma \ref{170810-5}, however 
some modifications are necessary. 
Let $\ep$ be a positive number. We may assume that $r_n<1$ for all $n\in\bbN$. 
Instead of (\ref{170806-5}) and (\ref{170806-6}), we will use 
the non-degeneracy of the forms 
\bea\label{170810-25} 
%\sup_{\theta\in\big(\sqrt{1-r_n^\ep},1\big]}
\sup_{\theta\in(\sqrt{1-r_n^\ep},1]}
E\bigg[\Delta_{\check{M}_n(\theta)}^{-p}
{\bf 1}_{\{\sum_{j=1}^{\sf k}|D^j\Xi|_{\mfh^{\otimes j}}>0\}}
\bigg] &<& \infty
\eea
and 
\bea\label{170810-26}
%\sup_{\theta\in\big(0,\sqrt{1-r_n^\ep}\big]}
\sup_{\theta\in(0,\sqrt{1-r_n^\ep}]}
E\bigg[\Delta_{X_n(\theta)}^{-p}
{\bf 1}_{\{\sum_{j=1}^{\sf k}|D^j\Xi|_{\mfh^{\otimes j}}>0\}}\bigg] 
&<& \infty
\eea
for every $p>1$ and a suitably differentiable functional $\Xi$. 
For the same reason as before, we have 
\beas
%\sup_n\sup_{\theta\in(\sqrt{1-r_n^\ep},1]}
\big|\varphi_n(\theta,\sfz,\sfx;\Xi) \big|
&\simleq& 
|(\sfz,\sfx)|^{-{\sf k}}
\eeas
uniformly in $\theta\in(\sqrt{1-r_n^\ep},1]$, $(\sfz,\sfx)\in\bbR^{\check{\sfd}}$ and $n\in\bbN$. 
For $\theta\in(0,\sqrt{1-r_n^\ep}]$, we will use 
$\csfd+\beta_x+1$ IBP below, just like before, and this procedure gives some 
power of $|\sfz|$. 
To cancel the power of $|\sfz|$ (including $\sfz$'s come from random polynomials when we use it), we attach 
the factor $(1-\theta^2)$, and then a power of $(1-\theta^2)^{-1}$ appears. 
We can replace it by $r_n^{-\ep L}$, where $L$ is a definite number. 
Thus what we obtained is 
\bea\label{17-810-30}
\sup_n\sup_{\theta\in(0,1)}\sup_{(\sfz,\sfx)\in\bbR^\csfd}
r_n^{\ep L}|(\sfz,\sfx)|^{{\sf k}}\big|\varphi_n(\theta,\sfz,\sfx;\Xi) \big|
&<& 
\infty. 
\eea

We need non-degeneracy (\ref{170810-25}) and (\ref{170810-26}) 
to apply the estimate (\ref{17-810-30}). 
For our purposes, when the functional $\Xi$ has $\psi_n$ or its derivative of certain order, 
it is sufficient to show non-degeneracy of $\check{M}_n(\theta)$ and $X_n(\theta)$ 
under truncation by $\psi_n$ with $\xi_n$ of (\ref{170810-31}). 
We make $\ep$ sufficiently small. 
Then it is easy to see 
\beas 
\Delta_{\check{M}_n(\theta)}
&=&
\Delta_{(M_n+W_\infty,X_\infty)}+r_n^{\ep/2} d_n^{**}(\theta)
\eeas
for some functional $d_n^{**}(\theta)$ such that 
$\sup_{\theta\in(\sqrt{1-r_n^{{\colred \ep}}},1],n\in\bbN}r_n^{-\ep/2}\|\tilde{d}_n(\theta)\|_{\csfd+6,p}<\infty$ 
for every $p>1$. 
Define $e_n$ as before with the coefficients of $d_n^{**}(\theta)$. 
Then we see (\ref{r1}) holds and 
$\Delta_{\check{M}_n(\theta)}$ and $\Delta_{X_n(\theta)}$ 
have uniform non-degeneracy under $\psi_n$, as before. 

For proof of the lemma, it is sufficient to show (\ref{170806-2}) for $i=3,...,12$. 
We can take the same way as the proof of Lemma \ref{170810-5}. 
Indeed, estimations of $R^{(i)}_n(\sfz,\sfx)$ $(i=4,...,12)$ are the same. 
Only estimation of $R^{(3)}_n(\sfz,\sfx)$ is slightly different. 
We do the same way for estimation of $R[i]$ $(i=1,...,24)$ 
with (\ref{17-810-30}), but in this situation, the bounds $o(r_n)$ 
that appeared in the previous proof become $O(r_n^{1+\kappa'})$ for some 
positive constant $\kappa'$, thanks to [D] (ii). 
Taking a sufficiently small $\ep$ so that $\ep L<\kappa'$, we obtain (\ref{170810-7}) 
in the present situation. 
\qed\halflineskip

\begin{lemme}\label{170811-1}\footnote{The definition of $\psi_n$ varies, depending on $[C]$ or $[C^\natural]$.} 
Suppose that either $[C]$ or $[C^\natural]$ is fulfilled. Then, for each $m\in\bbZ_+$, 
\beas 
\sup_{(z,x)\in\bbR^{\check{\sfd}}}\big||(z,x)|^m
\big(g^0_n(z,x)-h^0_n(z,x)\big)\big| 
&=& 
o(r_n)
\eeas
as $n\to\infty$. 
%{\rm[}The definition of $\psi_n$ varies, depending on $[C]$ or $[C^\natural]$. {\rm]}
\end{lemme}
\proof 
$\csfd+3$ times IBP provides 
\beas 
\sup_n\sup_{(\sfz,\sfx)\in\bbR^\csfd} |(\sfz,\sfx)|^{\csfd+3}
|\hat{g}^\alpha_n(\sfz,\sfx)|
&<&
\infty.
\eeas
Therefore, 
\bea\label{170811-10} 
\int_{\bbR^\csfd\setminus\Lambda_n(\csfd)}|\hat{g}^\alpha_n(\sfz,\sfx)|d\sfz d\sfx
&=&
O(r_n^{3q})
\eea
for every $\alpha\in\bbZ_+^\csfd$. 
We apply $\csfd+3$ times IBP with respect to $X_\infty$ to $\partialbs^\alpha E[\Psi(\sfz,\sfx)\psi_n]$ 
and $\exp\big(2^{-1}G_\infty[(\tti\sfz)^{\otimes2}]\big)$ with non-degeneracy of $G_\infty$ to derive 
\beas 
\sup_n\sup_{(\sfz,\sfx)\in\bbR^\csfd} |(\sfz,\sfx)|^{\csfd+3}
\big|\partialbs^\alpha E[\Psi(\sfz,\sfx)\psi_n]\big|
&<&
\infty.
\eeas
We remark that the factor $\sfx$ does not emerge but some product of $\sfz$ can newly appear though 
cancelled by the exponential. 
So
\beas
\int_{\bbR^\csfd\setminus\Lambda_n(\csfd)}\big|\partialbs^\alpha E[\Psi(\sfz,\sfx)\psi_n]\big|d\sfz d\sfx
&=&
O(r_n^{3q})
\eeas
for every $\alpha\in\bbZ_+^\csfd$. 
Let $\varsigma$ be any random symbol, like ${\mathfrak S}^{(3,0)}$, that appears in 
the $r_n$-order term of ${\mathfrak S}_n$. 
We apply $\csfd+\beta_x+1$ times IBP with respect to $X_\infty$ to 
$\partialbs^\alpha E\big[\Psi(\sfz,\sfx)\varsigma({\colred \tti}\sfz,{\colred \tti}\sfx)\big]$, and next use the Gaussianity of $\Psi$ in $\sfz$ 
to show 
\beas 
\int_{\bbR^\csfd\setminus\Lambda_n(\csfd)}\big|\partialbs^\alpha E\big[\Psi(\sfz,\sfx)\varsigma({\colred \tti}\sfz,{\colred \tti}\sfx)\big]\big|d\sfz d\sfx
&=&
O(r_n^q)
\eeas
for every $\alpha\in\bbZ_+^\csfd$. 
Thus 
\bea\label{170811-11}
\int_{\bbR^\csfd\setminus\Lambda_n(\csfd)}|\hat{h}^\alpha_n(\sfz,\sfx)|d\sfz d\sfx
&=&
O(r_n^{3q})+O(r_n^{1+q}).
\eea
This term becomes $o(r_n)$ if we choose $q\in(1/3,1/2)$. 

Now 
\beas
\Delta_n^\alpha
&:=&
\sup_{(z,x)\in\bbR^\csfd}\big|(\sfz,\sfx)^\alpha \big(g^0_n(z,x)-h^0_n(z,x)\big)\big|
\\&=&
\sup_{(z,x)\in\bbR^\csfd}\frac{1}{(2\pi)^\csfd}\bigg| \int_{\bbR^\csfd}
e^{-z[\tti\sfz]-x[\tti\sfx]}\big(\hat{g}^\alpha_n(\sfz,\sfx)-\hat{h}^\alpha_n(\sfz,\sfx)\big)d\sfz d\sfx
\bigg|
\\&\leq&
\frac{1}{(2\pi)^\csfd}\int_{\bbR^\csfd\setminus\Lambda_n(\csfd)}\big|\hat{g}^\alpha_n(\sfz,\sfx)\big|d\sfz d\sfx
+\frac{1}{(2\pi)^\csfd}\int_{\bbR^\csfd\setminus\Lambda_n(\csfd)}\big|\hat{h}^\alpha_n(\sfz,\sfx)\big|d\sfz d\sfx
\\&&
+\frac{r_n}{(2\pi)^\csfd}\int_{\Lambda_n(\csfd)}
r_n^{-1}\big|\hat{g}^\alpha_n(\sfz,\sfx)-\hat{h}^\alpha_n(\sfz,\sfx)\big|d\sfz d\sfx. 
\eeas
By (\ref{170811-10}), (\ref{170811-10}) and the properties (a) and (b) provided either 
Lemma \ref{170810-5} or Lemma \ref{170810-6}, we obtain 
$\Delta^\alpha_n=o(r_n)$. 

Here is the main theorem in this section. 

\begin{thm}\label{asy.exp7} Suppose that either $[C]$ or $[C^\natural]$ is fulfilled. 
Then, for any positive numbers $M$ and $\gamma$, 
\beas 
\sup_{f\in\cale(M,\gamma)}\Delta_n(f) &=& o(r_n)
\eeas
as $n\to\infty$. 
\end{thm}
\proof 
The local density $g^0_n$ is a continuous version of the measure 
$\big(E[\psi_n|\check{Z}_n=(z,x)]dP^{\check{Z}_n}\big)/dzdx$, admits 
any order of moments and 
\beas 
E\big[f(\check{Z}_n)\psi_n\big] 
&=&
\int_{\bbR^\csfd} f(z,x)g^0_n(z,x)dzdx.
\eeas
Let $p=1+\kappa/2$, where $\kappa$ is the one given in [C] (iv) (b) or in [C$^\natural$] (iv) (b). 
Then 
\beas 
\sup_{f\in\cale(M,\gamma)}\big|E\big[f(\check{Z}_n)\big]-E\big[f(\check{Z}_n)\psi_n\big]\big|
&\leq&
\sup_{f\in\cale(M,\gamma)}\|f(\check{Z}_n)\|_{p/(p-1)}\|1-\psi_n\|_p
\yeq 
o(r_n). 
\eeas

%Denote by $p^{X_\infty}$ the density function of $X_\infty$, that exists by non-degeneracy of $X_\infty$. 
For $k_1,k_2\in\bbZ_+$, we have 
\beas &&
|z|^{k_1}|x|^{2k_2}\big|E\big[(1-\psi_n)\phi(z;W_\infty,G_\infty)\delta_x(X_\infty)\big]\big|
%\\&=&
%|z|^{k_1}|x|^{2k_2}E\big[(1-\psi_n)\phi(z;W_\infty,G_\infty)|X_\infty=x\big]p^{X_\infty}(x)
%\\&=&
%E\big[X_\infty^{2k_2}(1-\psi_n)|z|^{k_1}\phi(z;W_\infty,G_\infty)|X_\infty=x\big]p^{X_\infty}(x)
\\&=&
\big|E\big[X_\infty^{2k_2}(1-\psi_n)|z|^{k_1}\phi(z;W_\infty,G_\infty)\delta_x(X_\infty)\big]\big|
\\&\leq&
C(k_1,k_2)\|1-\psi_n\|_{\nu,p}
\eeas
for all $(z,x)\in\bbR^\csfd$, where $C(k_1,k_2)$ is a constant depending on $(k_1,k_2)$, 
where $\nu=2[1+\sfd_1/2]\leq\sfd_1+2$. 
Therefore, 
\beas&& 
\sup_{f\in\cale(M,\gamma)}
\bigg|\int_{\bbR^\csfd}f(z,x)E\big[X_\infty^{2k_2}(1-\psi_n)|z|^{k_1}\phi(z;W_\infty,G_\infty)\delta_x(X_\infty)\big]
dzdx\bigg|
\\&=& 
O(\|1-\psi_n\|_{\nu,p})
\yeq 
o(r_n). 
\eeas
This estimate makes it possible to replace $p_n$ by $h^0_n$. 

In this way, estimation of $\Delta_n(f)$ is reduced to 
\beas&&  
\bigg|\int_{\bbR^\csfd}f(z,x)g^0_n(z,x)dzdx-\int_{\bbR^\csfd}f(z,x)h^0_n(z,x)dzdx\bigg|
\\&\leq&
\int_{\bbR^\csfd}|f(z,x)|\big(1+|(z,x)|)^{-m}dzdx
\times
\sup_{(z,x)\in\bbR^\csfd}\big|(1+|(z,x)|)^m\big(g^0_n(z,x)-h^0_n(z,x)\big)\big|
\\&=&
o(r_n)
\eeas
by Lemma \ref{170811-1} if $m$ is chosen as $m>\csfd+\gamma$. 
\qed\onelineskip
}

{\colred It is easy to give the joint asymptotic expansion with a reference variable 
in the applications of the previous sections, while we do not give statements explicitly here.
In statistics, it is important because the reference variable will be the random Fisher information matrix, 
an asymptotically ancillary statistic, and so on. 
}

{\colred 
%\begin{rem}\rm 
On the other hand, 
it is also possible to give a similar asymptotic expansion of $E[f(Z_n)]$ without a reference variable $X_n$. 
In fact, our result already applies to such a case if we take variable $X_n=X_\infty\sim N(0,1)$ independent of 
other variables. %any variables else. 
The expansion formula is valid in particular for functions $f(z)$ of $z$. 
Integrating out $x$ from $p_n(z,x)$, we obtain a formula $\int p_n(z,x)dx$. 
Formally, this formula corresponds to the case $\beta_x=0$ and $\sfd_1=0$.  
%in the notation we used so far. 
As a matter of fact, some of differentiability conditions can be reduced 
due to lack of the reference variable $X_n$. 
%However, we would need to trace the proof once again to validate it. 
%\end{rem}
%
We shall give a simplified version of Theorem \ref{asy.exp7} with $[C^\natural]$ but 
without the reference variable $X_n$, 
among several possibilities. 
In what follows, we will only consider the variable
\beas 
Z_n &=& M_n+r_nN_n. 
\eeas
In this situation, %$W_n=W_\infty=\dotw_n=0$ and formally $X_n=X_\infty=\dotx_n=0$. 
we need the random symbols 
\beas
{\colred {\mathfrak S}^{(3,0)}_n(\tti\sfz)} &=&
\frac{1}{3}r_n^{-1}\bigg\langle D\big\langle DM_n[\tti\sfz],u_n[\tti\sfz]\big\rangle_\mfh,u_n[\tti\sfz]\bigg\rangle_\mfh
\>\equiv\>\frac{1}{3}{\sf qTor}[(\tti\sfz)^{\otimes3}],
\\
{\colred {\mathfrak S}^{(2,0)}_{0,n}(\tti\sfz)}
&=& 
\half r_n^{-1}G^{(2)}_n(\sfz)
\>=\>
\half r_n^{-1}\bigg(\big\langle DM_n[\tti\sfz],u_n[\tti\sfz]\big\rangle_\mfh-G_\infty[(\tti\sfz)^2]\bigg)
\>\equiv\>\half {\sf qTan}[(\tti\sfz)^{\otimes2}],
\\
{\colred {\mathfrak S}^{(1,0)}_n(\tti\sfz)}
&=&
N_n[\tti\sfz],
\\
{\mathfrak S}^{(2,0)}_{1,n}(\tti\sfz)
&=&
\bigg\langle DN_n[\tti\sfz],u_n[\tti\sfz]\bigg\rangle_\HH. 
\eeas
Let 
\beas 
\Psi(\sfz) &=& 
\exp\big(2^{-1}G_\infty[(\tti\sfz)^{\otimes2}]\big).
\eeas

We consider the following condition. 
Recall $\ell=\sfd+8$ when $\sfd_1=0$. 
%$\ell=\check{\sfd}+8$. % and denote by $\beta_x$ the maximum degree in $x$ of ${\colred {\mathfrak S}.}$ %[Estimate $\sum_{i=5}^{12}R^{(i)}_n$]. 
%{\colorr Let $\sfd_2=(\ell+\beta_x-7)\vee\big(2[(\sfd_1+2)/2]+2[(\beta_x+1)/2]\big)$=\sfd+1, where $[x]$ is the maximum integer not larger than $x$.}

\bd\im[[D\!\!]] 
{\bf (i)}  $u_n\in\bbD^{\ell+1,\infty}(\mfh\otimes\bbR^\sfd)$, 
%$G_\infty\in\bbD^{\ell\vee{\colorr (\ell+\beta_x-7)\vee \sfd_2},\infty}(\bbR^\sfd\otimes_+\bbR^\sfd)$, 
$G_\infty\in\bbD^{\ell{\colorr},\infty}(\bbR^\sfd\otimes_+\bbR^\sfd)$, 
%$W_n,
$N_n\in\bbD^{\ell,\infty}(\bbR^\sfd)$, 
%$W_\infty\in\bbD^{{\colorr \ell\vee(\ell+\beta_x-7)\vee \sfd_2},
%$W_\infty\in\bbD^{{\colorr \ell},\infty}(\bbR^\sfd)$, 
%$X_n\in\bbD^{{\colorr \ell},\infty}(\bbR^{\sfd_1})$, 
%$X_\infty\in\bbD^{{\colorr \ell\vee(\ell+\beta_x-6)\vee(\sfd_2+1)},\infty}(\bbR^{\sfd_1})$. 
%$X_\infty\in\bbD^{{\colorr \ell\vee(\sfd_2+1)},\infty}(\bbR^{\sfd_1})$. 
%{\colorr [Construct  $\psi_n\in\bbD^{\ell-2,\infty}(\bbR)$. ]}
% 
\bd
\im[\hspace{-2mm}(ii)]  
There exists a positive constant $\kappa$ such that 
the following estimates hold for every $p>1$: %{\colorr still need to check!}
\beas%\label{c1}
%\|u_n\|_{{\colorr \ell-2},p} &=& O(1)
\|u_n\|_{\ell,p} &=& O(1) % $\ell$ for bouindedness of $\ell-2$-derivative of $\psi_n$
\eeas
\beas%\label{c22}
\|G_n^{(2)}\|_{\ell-2,p} &=& O(r_n)%\|G_n^{(2)}\|_{\ell-1,p} &=& O(r_n)%\qquad(k=2,3)
\eeas
\beas%\label{c23}
\|G_n^{(3)}\|_{\ell-2,p} &=& O(r_n)%\qquad(k=2,3)
\eeas
\beas%\label{c33n}
\|\big\langle DG^{(3)}_n,u_n\big\rangle_\mfh\|_{\ell-1,p}
&=& O(r_n^{1+\kappa})%\qquad(k=2,3)
\eeas
\beas%\label{c42n}
\bigg\|\bigg\langle D\bigg(\big\langle DG^{(2)}_n,u_n\big\rangle_\mfh\bigg),u_n\bigg\rangle_\mfh\bigg\|_{\ell-3,p}
&=& O(r_n^{1+\kappa})%\qquad(k=2,3)
\eeas
\beas%\label{c10}
\|N_n\|_{\ell-1,p}&=&O(1)
\eeas
\beas%\label{c10.1n}
\bigg\| \big\langle D\langle DN_n,u_n\rangle_\HH ,u_n\big\rangle_\HH\bigg\|_{\ell-2,p}&=& O(r_n^{\kappa}).
\eeas
%
%
%!TEX encoding = UTF-8 Unicode
\im[\hspace{-2mm}(iii)] 
For each pair $({\mathfrak T}_n,{\mathfrak T})=({\mathfrak S}^{(3,0)}_n,{\mathfrak S}^{(3,0)})$, 
$({\mathfrak S}^{(2,0)}_{0,n},{\mathfrak S}^{(2,0)}_0)$, 
%$({\mathfrak S}^{(2,0)}_n,{\mathfrak S}^{(2,0)})$, 
%$({\mathfrak S}^{(1,1)}_n,{\mathfrak S}^{(1,1)})$, 
$({\mathfrak S}^{(1,0)}_n,{\mathfrak S}^{(1,0)})$, 
%$({\mathfrak S}^{(0,1)}_n,{\mathfrak S}^{(0,1)})$, 
$({\mathfrak S}^{(2,0)}_{1,n},{\mathfrak S}^{(2,0)}_1)$,  
%$({\mathfrak S}^{(1,1)}_{1,n},{\mathfrak S}^{(1,1)}_1)$,
the following conditions are satisfied. 
\bd
\im[(a)] ${\mathfrak T}$ is polynomial random symbol the coefficients of which are in 
%$\bbD^{\csfd+\beta_x+1,1+}=\bigcup_{p>1}\bbD^{\csfd+\beta_x+1,p}$. 
$L^{1+}=\cup_{p>1}L^p$. 
\im[(b)] 
For some $p>1$, there exists a polynomial random symbol $\bar{\mathfrak T}_n$ that has $L^p$ coefficients and  
the same degree as ${\mathfrak T}$, 
\beas 
E\big[\Psi(\sfz){\mathfrak T}_n({\colred\tti}\sfz)\big]
&=& E\big[\Psi(\sfz)\bar{\mathfrak T}_n({\colred\tti}\sfz)\big]
\eeas
and 
$\bar{\mathfrak T}_n\to{\mathfrak T}$ in $L^p$. 
\ed
 
\im[\hspace{-2mm}(iv)] 
\hspace{0.5mm} {\bf (a)} $G_\infty^{-1}\in L^{\infty-}$. 
\bd\im[(b)]
There exists $\kappa>0$ such that 
\beas 
P\big[\Delta_{M_n}<s_n\big] &=& O(r_n^{1+\kappa})
\eeas
%and that 
%\beas 
%P\big[\Delta_{X_\infty}<s_n\big] &=& O(r_n^{1+\kappa}){\colorr remove}
%\eeas
for some positive random variables $s_n\in\bbD^{\ell-2,\infty}$ satisfying 
$\sup_{n\in\bbN}(\|s_n^{-1}\|_p+\|s_n\|_{\ell-2,p})<\infty$ for every $p>1$. 
\ed
\ed
\ed
\onelineskip

In the present situation, the random symbol ${\mathfrak S}$ is defined by 
\beas 
{\mathfrak S}(\tti\sfz)
&=&
{\mathfrak S}^{(3,0)}(\tti\sfz)%[(\tti\sfz)^{\otimes3}]
+{\mathfrak S}^{(2,0)}_0(\tti\sfz)%[(\tti\sfz)^{\otimes2}]
%+{\mathfrak S}^{(2,0)}(\tti\sfz,\tti\sfx)%[(\tti\sfz)^{\otimes2}]
%\\&&
%+{\mathfrak S}^{(1,1)}(\tti\sfz,\tti\sfx)%[(\tti\sfz)^{\otimes2}]
+{\mathfrak S}^{(1,0)}(\tti\sfz)%[\tti\sfz]
%+{\mathfrak S}^{(0,1)}(\tti\sfz,\tti\sfx)%[\tti\sfx]
%\\&&
%{\colorg
+{\mathfrak S}^{(2,0)}_1(\tti\sfz).
%+{\mathfrak S}^{(1,1)}_1(\tti\sfz,\tti\sfx).
%}
\eeas
Let ${\mathfrak S}_n=1+r_n{\mathfrak S}$ and define $\hat{p}_n(z)$ by 
\beas 
\hat{p}(z) &=& E\big[{\mathfrak S}_n(\partial_z)^*\phi(z;0,G_\infty)\big]
\eeas
with naturally defined adjoint operation ${\mathfrak S}_n(\partial_z)^*$. 
We follow the proof of Theorem \ref{asy.exp7} but with 
\beas 
\varphi_n(\theta,\sfz;\Xi)
&=&
E\big[e^{\lambda_n(\theta;\sfz)}\Xi\big]
\eeas
for $\varphi_n(\theta,\sfz,\sfx;\Xi)$, where 
\beas 
\lambda_n(\theta;\sfz) 
&=& 
\theta M_n[\tti\sfz]+2^{-1}(1-\theta^2)G_\infty[(\tti\sfz)^{\otimes2}]+\theta r_nN_n[\tti\sfz].
\eeas
Then, in place of (\ref{17-810-30}), we obtain 
\beas 
\sup_n\sup_{\theta\in(0,1)}\sup_{\sfz\in\bbR^\sfd}r_n^{\ep L}
|\sfz|^{\sf k}\big|\varphi_n(\theta,\sfz;\Xi)\big|
&<& \infty. 
\eeas
For this estimate for $\theta\in(0,\sqrt{1-r_n^\ep}]$, only non-degeneracy of $G_\infty$ is used. 
In this way, we can prove the validity of the asymptotic expansion by $\hat{p}_n$. 
Denote by $\hat{\cale}(M,\gamma)$ the set of measurable functions $f:\bbR^\sfd\to\bbR$ such that 
$|f(z)|\leq M(1+|z|)^\gamma$ for all $z\in\bbR^\sfd$. 
Let 
\beas 
\hat{\Delta}_n(f) 
&=& 
\bigg|E\big[f(Z_n)\big]-\int_{\bbR^\sfd}f(z)\hat{p}_n(z)dz\bigg|
\eeas
for $f\in\hat{\cale}(M,\gamma)$. 
\begin{thm} Suppose that Condition $[D]$ is satisfied. 
Then, for any positive numbers $M$ and $\gamma$, 
\beas 
\sup_{f\in\hat{\cale}(M,\gamma)}\hat{\Delta}_n(f)&=&o(r_n)
\eeas
as $n\to\infty$. 
\end{thm}
}
%%%%%%%%%%%%%%%%%%%%%%%%%%%%%%%
%\section{Double Skorohod integrals}\label{20160816-1}

%
\begin{en-text}
$\frac{1}{3}$と次の項で$v_n$にしてIBPする．
\end{en-text}
%

%\subsection{Asymptotic expansion of double stochastic integral: differentiable $f$}

%
\begin{en-text}
Two representations: with derivatives of $\varphi$ or with 
$_{C^{2k}_b(\bbR^\sfd)}\langle \varphi,p_n\rangle_{C^{-2k}_b(\bbR^\sfd)}$. 

Replace $X_\infty$ by $X_\infty+\ep \eta_1$ and $Z_n+\ep \eta_0$, then $p_n^\ep$ is well-defined. 
The integral involving $p_n^\ep$ w. r. t. $x$ substitute $X_\infty+\ep \eta_1$ for $x$ in the formula. 
Finally let $\ep\down0$ to get a final expression of the formula. 
\bea\label{20160816-3} &&
E[\int\int g(z,x)\phi(z;0,G_\infty+\ep^2 I_\sfd)\delta_x(X_\infty+\ep \eta_1)dzdx]
\nn\\&=&
\int\int g(z,x)E[\phi(z;0,G_\infty+\ep^2 I_\sfd)|X_\infty+\ep \eta_1=x]p^{X_\infty+\ep \eta_1}(x)dzdx
\nn\\&=&
\int E\big[g(z,X_\infty+\ep \eta_1)E[\phi(z;0,G_\infty+\ep^2 I_\sfd)|X_\infty+\ep \eta_1]\big]dz
\nn\\&=&
\int E\big[g(z,X_\infty+\ep \eta_1)\phi(z;0,G_\infty+\ep^2 I_\sfd)\big]dz
\eea
So, first apply the expansion formula for $X_\infty+\ep \eta_1$ and $Z_n+\ep \eta_0$. 
The error bound has been obtained with derivatives $\varphi$, and 
the error bound is stable when $\ep\down0$. 
But the general formula with Watanabe's delta function is rewritten by the above formula. 
The right-hand side of (\ref{20160816-3}) is also stable when $\ep\down0$. Thus, we obtain an expansion formula 
with the shape of (\ref{20160816-3}) with $\ep=0$. 
\end{en-text}
%

%\subsection{Asymptotic expansion of double stochastic integral: non-degeneracy and measurable $f$}

%%%%%%%%%%%%%%%%%%%%%%%%%%%%

%\bibliographystyle{plain}
%\bibliography{bibtex20080401}
\bibliographystyle{spmpsci}      % mathematics and physical sciences
\bibliography{bibtex-20160515-20160926-20170811+}   % name your BibTeX data base

\end{document}